\documentclass[12pt]{article}

\usepackage{amsmath}
\usepackage{amssymb}
\usepackage{epsfig}
\usepackage{tikz}
\usetikzlibrary{snakes,arrows}

\newenvironment{proof}[1][Proof]{\begin{trivlist}
\item[\hskip \labelsep {\bfseries #1}]}{$\blacksquare$ \end{trivlist}}

\newtheorem{df}{Definition}[section]
\newtheorem{ex}{Example}[section]
  \newtheorem{theorem}{Theorem}[section]
\newtheorem{remark}{Remark}
  \newtheorem{lemma}[theorem]{Lemma}
  
  \newtheorem{corollary}[theorem]{Corollary}

\newcommand{\qed}{\nobreak \ifvmode \relax
 \else \ifdim\lastskip<1.5em \hskip-\lastskip
\hskip1.5em plus0em minus0.5em \fi \nobreak
 \vrule height0.75em width0.5em depth0.25em\fi}
\begin{document}


\title{\bf The Intersection Spectrum of Skolem Sequences and its Applications to \mbox {\boldmath ${\lambda}$}-Fold Cyclic Triple Systems}
\author{
 Nabil Shalaby\hspace{2in} Daniela Silvesan \\
 Department of Mathematics and Statistics \\
 Memorial University of Newfoundland \\
 St. John's, Newfoundland \\
 CANADA A1C 5S7}
 \maketitle

\begin{abstract}
A Skolem sequence of order $n$ is a sequence
$S_n=(s_{1},s_{2},\ldots,s_{2n})$ of $2n$ integers containing each
of the integers $1,2,\ldots,n$ exactly twice, such that two
occurrences of the integer $j \in \{1,2,\ldots,n\}$ are separated by
exactly $j-1$ integers. We prove that the necessary conditions are sufficient for
existence of two Skolem sequences of order $n$ with
$0,1,2,\ldots,n-3$ and $n$ pairs in same positions. Further, we
apply this result to the fine structure of  cyclic two, three and four-fold triple systems, and also to the fine structure of  $\lambda$-fold
directed triple systems and  $\lambda$-fold Mendelsohn triple
systems.
\end{abstract}

{\bf {Keywords:}} Skolem and Langford sequences; combinatorial designs; cyclic triple systems; Mendelsohn triple systems.


\section{Introduction}

A {\it $\lambda$-fold triple system} of order $v$, denoted by TS$(v,\lambda)$, is a pair $(V,\cal{B})$ where $V$ is a $v$-set of points and $\cal{B}$ is a set of $3$-subsets ({\it blocks}) such that any $2$-subset of $V$ appears in precisely $\lambda$ blocks. If $\lambda=1$, the system is called a {\it Steiner triple system}, denoted by STS$(\lambda)$ is a permutation on $V$ leaving $\cal{B}$ invariant.  A TS$(v,\lambda)$ is {\it cyclic} if its automorphism group contains a $v$-cycle. A cyclic TS$(v,\lambda)$ is denoted by CTS$(v,\lambda)$. If $\lambda=1$, the system is called a {\it cyclic Steiner triple system} and is denoted by CSTS$(v)$.

Several researchers investigated cyclic designs and many results have been found. For more details the reader is directed to \cite{crc,colbournrosa}. In \cite{skolem}, Skolem introduced the idea of what is now known as a Skolem sequence of
order $n$, for $n\equiv0,1\ (\textup{mod}\ 4)$. In \cite{okeefe}, O'Keefe extended this idea to that of the hooked Skolem
sequence of order $n$, for $n\equiv2,3\ (\textup{mod}\ 4)$, the existence of which for all admissible $n$, along with
that of Skolem sequences, would lead to the
constructions of CSTS$(6n+1)$. Some authors used Skolem sequences and their generalizations to construct various cyclic designs, see for example \cite{bilington}. One reason for the interest in cyclic designs and designs with repeated blocks is their applications in statistics, see for example \cite{foody}.

In \cite{lindnerrosa}, Lindner and Rosa
determined the number of repeated triples in a TS$(v,2)$ for
$v\equiv1,3\ (\textup{mod}\ 6)$. In \cite{hoffman}, Rosa and Hoffman extended this
determination to the case $v\equiv0,4\ (\textup{mod}\ 6)$. In
\cite{wallis}, Lindner and Wallis determined the number of repeated directed triples in
a DTS$(v,2)$. Necessary conditions for $(c_1,c_2,c_3)$ ($c_i$ is the number of blocks repeated $i$ times in a system) to be the
fine structure of a threefold triple system were shown to be
sufficient for $v\equiv1,3\ (\textup{mod}\ 6)$, $v\geq 19$ in
\cite{colbournmathonrosa} and for $v\equiv5\ (\textup{mod}\ 6)$, $v\geq 17$ in \cite{colbournmathonshalaby}. Also the fine structure of a threefold
directed triple system was determined in \cite{milici}. 

In this paper, we find similar results for cyclic designs. We use Skolem sequences and the existing constructions of Langford sequences found in \cite{bermond,linekmor,simpson}. We also introduce new constructions that enable us to construct, for all admissible orders, two Skolem sequences of order $n$ that have $0,1,\ldots,n-3$ and $n$ pairs in common (two sequences have $m$ pairs in common if they have $m$ pairs in the same positions). Using these results, we determine, with a few possible exceptions, the fine structure of a cyclic two-fold triple system for $v\equiv1,3\ (\textup{mod}\ 6)$, and the fine structure of both a
cyclic three-fold and a cyclic four-fold triple system, for $v\equiv1,7\ (\textup{mod}\ 24)$. Then, we
extend these results to the fine structure of a $\lambda$-fold
directed triple system and a $\lambda$-fold Mendelsohn triple
system for $\lambda=2,3,4$. There are
a few possible exceptions in the fine structure which cannot be solved
using our technique. These possible exceptions can probably be solved
using other techniques.

\section{Basic Definitions and Known Results}

In this section we provide the basic definitions and known results necessary for the further proven results.

\begin{df}{\label{defskolem}}
A (hooked) Skolem sequence of order $n$ is a sequence $S_n=(s_{i})$ of $2n$ [2n+1] integers which satisfies the conditions:
\begin{enumerate}
\item for every $k\in \{1,2,\ldots,n\}$ there are exactly two elements $s_{i},s_{j}\in S$ such that $s_{i}=s_{j}=k$,
\item if $s_{i}=s_{j}=k,\;i<j$, then $j-i=k$,
\item in a hooked sequence $s_{2n}=0$.
\end{enumerate}

A (Hooked) Skolem sequences is also written as collections of ordered pairs $\{(a_i,b_i):1\leq i\leq n,\;b_i-a_i=i\}$ with $\cup_{i=1}^{n}\{a_i,b_i\}=\{1,2,\ldots,2n\}$. Two sequences of order $n$ are disjoint if no element $k$ occupies the same two locations in both sequences. We denote a sequence that has no pairs in common with $S_n$ by $d(S_n)$. Two sequences have $m$ pairs in common if $m$ distinct entries occur in the same positions in the two sequences. We denote different Skolem sequences of order $n$ with $0,1,\ldots,n-3,n$ pairs in common by $sp(S_n)$.
\end{df}

\begin{ex}
The two Skolem sequences of order $4$, $S_4=(1,1,3,4,2,3,2,4)$ and $S_4'=(2,3,2,4,3,1,1,4)$ have one pair in common.
\end{ex}

\begin{df}
Given a Skolem sequence $S_n=(s_{1},s_{2},\ldots,s_{2n})$, the {\it reverse} $\stackrel{\leftarrow}{S_n}=(s_{2n},\ldots,s_1)$ is also a Skolem sequence.
\end{df}

\begin{theorem} \cite{skolem} A Skolem sequence of order $n$ exists if and only if
$n\equiv0,1\ \mathrm{(mod\ 4)}$.
\end{theorem}
\begin{theorem} \cite{okeefe} A hooked Skolem sequence of order $n$ exists if and only if
$n\equiv2,3\ \mathrm{(mod\ 4)}$.
\end{theorem}

\begin{df}{\label{defnear}}
Let $m,n$ be positive integers, with $m\leq n$. A (hooked) near-Skolem sequence of order $n$ and defect $m$, is a sequence $m$-near $S_n=(s_{i})$ of $2n-2$ $(2n-1)$ integers $s_{i}\in \{1,2,\ldots, m-1,m+1,\ldots,n\}$ which satisfies the following conditions:
\begin{enumerate}
\item for every $k\in \{1,2,\ldots,m-1,m+1,\ldots,n\}$, there are exactly two elements $s_{i},s_{j}\in S$ such that $s_{i}=s_{j}=k$,
\item if $s_{i}=s_{j}=k$, then $j-i=k$,
\item in a hooked sequence $s_{2n-2}=0$.
\end{enumerate}
\end{df}

\begin{theorem} \cite{shalaby1} An $m$-near Skolem sequence of order $n$ exists if and only if $n\equiv 0,1\ \mathrm{(mod\ 4)}$ and $m$ is odd, or $n\equiv 2,3\ \mathrm{(mod \ 4)}$ and $m$ is even.
\end{theorem}
\begin{theorem} \cite{shalaby1} A hooked $m$-near Skolem sequence of order $n$ exists if and only if $n\equiv0,1\ \mathrm{(mod\ 4)}$ and $m$ is even, or $n\equiv2,3\ \mathrm{(mod\ 4)}$ and $m$ is odd.
\end{theorem}

\begin{df}{\label{deflangford}}
A (hooked) Langford sequence of order $n$ and defect $d$, $n>d$ is a sequence $L_d^n=(l_{i})$ of $2n$ $(2n+1)$ integers which satisfies:
\begin{enumerate}
\item for every $k\in \{d,d+1,\ldots,d+n-1\}$, there exist exactly two elements $l_{i},l_{j}\in L$ such that $l_{i}=l_{j}=k$,
\item if $l_{i}=l_{j}=k$ with $i<j$, then $j-i=k$,
\item in a hooked sequence $l_{2n}=0$.
\end{enumerate}
\end{df}

Note that some authors use the term length instead of the order for the Langford sequence. For some authors the order of a Langford sequence is $n+d-1$.

\begin{theorem} \cite{bermond, linekmor, simpson} The necessary and sufficient conditions for the existence of a Langford sequence are:
\begin{enumerate}
\item $n\geq 2d-1$, and
\item $n\equiv 0,1\ \mathrm{(mod\ 4)}$ for $d$ odd, $n\equiv 2,3\ \mathrm{(mod\ 4)}$ for $d$ even.
\end{enumerate}
\end{theorem}

\begin{theorem} \cite{bermond, linekmor, simpson} The necessary and sufficient conditions for the existence of a hooked Langford sequence $(d,d+1,\ldots,d+n-1)$ are:
\begin{enumerate}
\item $n(n+1-2d)+2\geq 2$, and
\item $n\equiv 2,3\ \mathrm{(mod\ 4)}$ for $d$ odd, $n\equiv 0,1\ \mathrm{(mod\ 4)}$ for $d$ even.
\end{enumerate}
\end{theorem}

\begin{df}
A $k$-extended Skolem sequence of order $n$ is sequence $k-ext\,S_n=(s_1,s_2,...,s_{2n+1})$ in which $s_k=0$ and for each $j\in\{1,2,...n\}$ there exists a unique $i\in\{1,2,...,n\}$ such that $s_i=s_{i+j}=j$.
\end{df}

\begin{theorem} \cite{baker} The necessary and sufficient conditions for the existence of a $k$-extended Skolem sequence are $n\equiv0,1\ (\textup{mod}\ 4)$ for $k$ odd, and $n\equiv2,3\ (\textup{mod}\ 4)$ for $k$ even.

\end{theorem}

\begin{df}
A $(p,q)$-extended Rosa sequence of order $n$ is a sequence $R_n(p,q)=(r_1,\ldots,r_{2n+2})$ of $2n+2$ integers $(0\leq p,q\leq 2n+2)$ containing each of the symbols $0,1,\ldots,n$ exactly twice, such that two occurrences of the integer $j>0$ are separated by exactly $j-1$ symbols and $r_p=0, r_q=0$.
\end{df}

\begin{theorem} \cite{linekshalaby} With the exception of $R_1(2,3)\,\textup{and}\, R_4(5,6)$, a $R_n(p,q)$ exists if and only if
\begin{enumerate}
\item $p\not\equiv q\ (\textup{mod}\ 2)\,\textup{and}\, n\equiv0,1\ (\textup{mod}\ 4), \,\textup{or if}$
\item $p\equiv q\ (\textup{mod}\ 2)\,\textup{and}\, n\equiv2,3\ (\textup{mod}\ 4)$
\end{enumerate}

\end{theorem}

\begin{df}
A k-extended Langford sequence of defect $d$ and order $n$ is a partition of $[1,2n+1]-\{k\}$ into differences $[d,d+n-1]$. A hooked k-extended Langford sequence of defect $d$ and order $n$ is a partition of $[1,2n+1]-\{2,k\}$ into differences $[d,d+n-1]$.
\end{df}

\begin{theorem} \cite{linekmor} If $n\geq 2d-1$ and $n\not\in [2d+2,8d-5]$, then the set $[1,2n+1]-\{k\}$ can be partition into differences $[d,d+n-1]$ whenever $(n,k)\equiv (0,1),(1,d),(2,0),(3,d+1)\ \mathrm{(mod\ (4,2))}$.
\end{theorem}

\begin{theorem} \cite{crc} A STS$(v)$ exists if and only if $v\equiv 1,3\ \mathrm{(mod\ 6)}$.
\end{theorem}

\begin{theorem} \cite{crc} A CSTS$(v)$ exists whenever $v\equiv 1,3\ \mathrm{(mod\ 6)}$ and $v\neq 9$.
\end{theorem}

Skolem and hooked Skolem sequences of order $n$ give rise to cyclic Steiner triple systems of order $6n+1$. From a Skolem sequence or a hooked Skolem sequence of order $n$, we construct the pairs $(a_i,b_i)$ such that $b_i-a_i=i$ for $1\leq i\leq n$. The set of all triples $(i,a_i+n,b_i+n)$ for $1\leq i\leq n$, is a solution to the Heffter first difference problem. These triples yield the base blocks for a CSTS$(6n+1)$: $\{0,a_i+n,b_i+n\}$, $1\leq i\leq n$. Also,  $\{0,i,b_i+n\}$, $1\leq i\leq n$ is another set of base blocks of an CSTS$(6n+1)$.

\begin{ex}
$S_4=(1,1,4,2,3,2,4,3)$ yields the pairs $\{(1,2),(4,6),(5,8),\linebreak (3,7)\}$. These pairs yield the base blocks for two CSTS$(25)s$:
\begin{enumerate}
\item $\{0,5,6\},\{0,8,10\},\{0,9,12\},$ and $\{0,7,11\}\ (\textup{mod}\;25)$;
\item $\{0,1,6\},\{0,2,10\},\{0,3,12\},$ and $\{0,4,11\}\ (\textup{mod}\;25)$.
\end{enumerate}
\end{ex}

\begin{df}
A $\lambda$-fold directed triple system (Mendelsohn triple system), denoted by DTS$(v,\lambda)$ (respectively by MTS$(v,\lambda)$), is a pair $(V,B)$ where $V$ is a $v$-set while $B$ is a collection of ordered $3$-subsets of $V$ (directed cycles of length $3$), such that each ordered pair of distinct elements of $V$ is contained in exactly $\lambda$ triples.
\end{df}

\begin{ex}
 $\{[0,2,1],[1,3,2],[2,0,3],[3,1,0]\}$ is a DTS$(4,1)$ and $\{\langle0,1,2\rangle,\linebreak \langle2,1,0\rangle\}$ is a MTS$(3,1)$.
\end{ex}

 Note that each cyclic triplet $[a,b,c]$ in a directed system contains the ordered pairs $(a,b)$, $(a,c)$ and $(b,c)$, and that each cyclic triplet $\langle a,b,c\rangle$ in a Mendelsohn system
contains the ordered pairs $(a,b)$, $(b,c)$ and $(c,a)$.

\section{Langford Sequences and Their Reverses}

In this section we check the known Langford sequences constructions if they are reverse disjoint or have pairs in common.

The existence of two disjoint Langford sequences for all admissible orders is still an open question. However, we noticed that the reverses of the known constructions \cite{bermond,linekmor,simpson} yield two disjoint Langford sequences with some finite number of exceptions. We arrange the results we found in Table \ref{n=4t}. On the last column of the table we give complete reference for the construction (reference, theorem number or table number, row number). To see the sequences used to get the results in Table \ref{n=4t}, the reader has to consult \cite{bermond,linekmor,simpson}. In \cite{linekmor}, Table $1d$ take $\delta=0$ and do not take the alternative in line (5) of the table. Since \cite{bermond} is difficult to obtain we describe the constructions from it in Appendix C of the Supplement.

\begin{table}[!h]
\begin{center}
\begin{tabular}{|c|c|c|c|c|}
\hline
 & n & d & Pairs in common & Reference\\
\hline
1 & $4t$ & $4s$  & $1 \;\textup{if}\; s\equiv1\ \mathrm{(mod\;3)}$ &  \cite{simpson}, $1(11)$\\
 && $s\geq1$ & $1 \;\textup{if}\; s\equiv2\ \mathrm{(mod\;3)}$ & \cite{simpson}, $1(12)$\\
&&$t\geq 2s$& $0 \;\textup{if}\; s\equiv 0\ \mathrm{(mod\;3)}$ & \cite{simpson}, $1$\\
 & & $4s+1$  & $0\;\textup{if}\;s\equiv 1,2\ \mathrm{(mod\;3)}$ & \cite{simpson}, $1$\\
&&$s\geq1$& $2\;\textup{if}\;s\equiv 0\ \mathrm{(mod\;3)}$ & \cite{simpson}, $1(11),(12)$\\
&& $t\geq2s+1$&&\\
 &  & $4s+2$  & $1\;\textup{if}\;s\equiv 0,2\ \mathrm{(mod\;3)}$ & \cite{simpson}, $1(14)$\\
&&$s\geq1$& $3\;\textup{if}\;s\equiv1\ \mathrm{(mod\;3)}$ & \cite{simpson}, $1(11),(12),(14)$\\
&& $t\geq2s+1$&&\\
 &  & $4s-1$  & $0\;\textup{if}\;s=1\; \textup{or}\;2$ & \cite{simpson}, $1$\\
&&$s\geq1$& $s-2\;\textup{if}\;s\equiv2\ \mathrm{(mod\;3)}$ & \cite{simpson}, $1(8)$\\
&&$t\geq2s$&$s\geq 5$&\\
&&& $s-1\;\textup{if}\;s\equiv 0,1\ \mathrm{(mod\;3)}$ & \cite{simpson}, $1(8),(12)$\\
&&&$s\geq 3$&\\
\hline
2 & $4t$ & $2t-e$   & $2 \;\textup{if}\; e\equiv 1\ \mathrm{(mod\;3)}$ & \cite{bermond}, $3(3),(4)$\\
&&$t\geq2e+1$& $0$ if $e\equiv 0,2\ \mathrm{(mod\;3)}$&\\
\hline
3 & $2d-1$  & $d\geq 2$ & $0$ &\cite{bermond}, $2$\\
&$\mathrm{(mod\;4)}$ & $d$ even & &\\
 & & $d\geq3$  & $0$ if $ n\neq 2d-1$ &\cite{bermond}, $2$\\
&& $d$ odd & $1$ if $n=2d-1$ & \cite{bermond}, $2(4)$\\
\hline
4 & $2d$ & $0\ \mathrm{(mod\;6)}$  & $1$ & \cite{linekmor} $1d,(4)$\\
& & $2\ \mathrm{(mod\;6)}$  & $1$ & \cite{linekmor} $1d,(3)$\\
 & & $4 \ \mathrm{(mod\;6)}$  & $0$& \cite{linekmor} $1d$\\
\hline
5 & $2d-1$ & $0\ \mathrm{(mod\;3)}$ & $0,1$ if $d=3$ & \cite{linekmor}, $0a$ and  \cite{bermond}, $2$\\
&&& $0,1,2$ if $d\geq 6$ & \cite{linekmor}, $0a$ and  \cite{bermond}, $2$\\
& & $1\ \mathrm{(mod\;3)}$ & $0,1,2$ if $d\geq 4$ & \cite{linekmor}, $0a$ and  \cite{bermond}, $2$ \\
& & $2\ \mathrm{(mod\;3)}$ & $0$ if $d=2$ & \cite{linekmor}, $0a$ and  \cite{bermond}, $2$ \\
& &  & $0,1,3$ if $d=5$ & \cite{linekmor}, $0a$ and  \cite{bermond}, $2$ \\
& &  & $0,1,2,3$ if $d\geq 8$ & \cite{linekmor}, $0a$ and  \cite{bermond}, $2$ \\

\hline
\end{tabular}
\caption {The number of pairs in common between Langford sequences and their reverses} \label{n=4t}
\end{center}
\end{table}

Below we describe how to find the number of pairs in common between the Langford sequences found in \cite{simpson} and their reverses.

Consider $n=4t,\;d=4s$, with $s\geq1$ and $t\geq2s$. Note that the pair $(a_i,b_i)$ will appear in a sequence and its reverse if and only if $a_i+b_i=2n+1=8t+1$. This does not happen in the rows labeled $(1)-(10)$ and $(13)$ of Simpson's table.

In row $(11)$, $(2t-3s+2+j)+(6t-s+3+2j)=8t-4s+5+3j=8t+1\Leftrightarrow j=\frac{4s-4}{3}\in\mathbb{Z}$ so the pair is in the sequence and its reverse if and only if $s\equiv 1\ \mathrm{(mod\;3)}$.

In row $(12)$, $(2t+1+j)+(6t-s+2+2j)=8t-s+3+3j=8t+1\Leftrightarrow j=\frac{s-2}{3}\in \mathbb{Z}$ so the pair is in the sequence and its reverse if and only if $s\equiv 2\ \mathrm{(mod\;3)}$.

If $s\equiv 0\ \mathrm{(mod\;3)}$, no pair occurs in the sequence and its reverse.

Similarly, we check the number of pairs in common between sequences 2 and 3 and their reverses. For sequence 4, the pair $(a_i,b_i)$ will appear in a sequence and its reverse if and only if $a_i+b_i=2n+3$.

Table \ref{n=4t}, sequence 5 gives the number of pairs in common between six different Langford sequences. Here we deal with a special Langford sequence, $L_d^{2d-1}$. The Langford sequence $L_d^{2d-1}$ \cite{linekmor}, Table $0a$, can be modified by moving the pair $n=2d-1$ from the end of the sequence to the beginning of it. For example, $L_{3}^{5}=(6,7,3,4,5,3,6,4,7,5)$ becomes $L_{3}^{5}=(5,6,7,3,4,5,3,6,4,7)$. We call the second sequence, the {\it modified Langford sequence}.

Taking the Langford sequence from \cite{linekmor}, Table $0a$ and the modified Langford sequence we have two different Langford sequences of order $2d-1$ that are disjoint. This is obvious since every position in the modified Langford sequence is shifted by one position to the right.

Taking the Langford sequence from
\cite{linekmor}, Table $0a$ and the Langford sequence from \cite{bermond}, Theorem $2$ we get two Langford sequences of order $2d-1$ which are disjoint if $d\equiv 1\ \mathrm{(mod\;3)}$, and have one pair in common if $d\equiv 0,2\ \mathrm{(mod\;3)}$. This can be seen by checking every position from the first sequence with the corresponding position on the second sequence. We give below the general constructions for these sequences.

The general construction for $L_d^{n=2d-1=4t+3}$ with $d>2$ in \cite{bermond} (Theorem 2, $d$ even) is:

\begin{center}
$3d-3,3d-5,\ldots,2d+1,2d-2,2d-4,\ldots,d+2,2d,3d-2,3d-4,\ldots,2d+2,d,2d-1,2d-3,\ldots,d+1,d+2,d+4,\ldots,2d-2,d,2d+1,2d+3,\ldots,3d-3,2d,d+1,d+3,\ldots,2d-1,2d+2,2d+4,\ldots,3d-2$.
\end{center}

The general construction for $L_d^{2d-1}$ in \cite{bermond} (Theorem 2, $d$ odd) is:

\begin{center}
$3d-3,\ldots,2d+2,2d-1,\ldots,d+2,2d,3d-2,\ldots,2d+1,d,2d-2,\ldots,d+1,d+2,\ldots,2d-1,d,2d+2,\ldots,3d-3,2d,d+1,\ldots,2d-2,2d+1,\ldots,3d-2.$
\end{center}

The general construction for the Langford sequence from \cite{linekmor}, Table $0a$ is:

\begin{center}
$2d,2d+1,\ldots,3d-2,d,d+1,d+2,\ldots,2d-1,d,2d,d+1,2d+1,\ldots,2d-1$.
\end{center}

Taking the modified Langford sequence and the Langford sequence from \cite{bermond}, we get one pair in common if $d=5$ or $d\equiv\ 1(\textup{mod}\;3),d\neq 4$ and two pairs in common if $d=4$ or $d\equiv\;0,2\ (\textup{mod}\;3),d\neq 5$.

Taking the modified Langford sequence and the reverse of the Langford sequence from \cite{bermond}, we get one pair in common if $d=3$ or $d=4$, no pairs in common if $d\equiv\;0,2\ (\textup{mod}\;3),d\neq 3$ and two pairs in common if $d\equiv \;1\ (\textup{mod}\;3),d\neq 4$.

Taking the modified Langford sequence and the reverse of the Langford sequence from \cite{linekmor}, we get three pairs in common for $d\equiv\;2\ (\textup{mod}\;3),d\geq 5$.

\section{New constructions}

In this section, we present two new constructions that manipulate Skolem and Langford sequences. 

\subsection{First Main Construction}
We construct new Skolem sequences by adjoining a (hooked) Skolem sequence to a (hooked) Langford sequence. In order for a Langford sequence to exist it must be twice as long as a Skolem sequence. For this reason we divide the problem of finding pairs in common in three major cases:

{\bf Case 1. The number of pairs in the interval  \mbox {\boldmath ${[0,\lfloor\frac{n}{3}\rfloor]}$}}. To show that two Skolem sequences intersect in $[0,\lfloor\frac{n}{3}\rfloor]$ pairs, we adjoin a Skolem sequence to a Langford sequence, i.e., $S_d L^{n-d}_{d+1}=S_{n}$. Then we reverse the Langford sequence and keep the Skolem sequence in its normal position, i.e., $S_d \stackrel{\leftarrow}{L^{n-d}_{d+1}}=S_{n}$. We get two different Skolem sequences of order $n$ with some pairs in common. These pairs will be in the interval $[0,\lfloor\frac{n}{3}\rfloor]$.

{\bf Case 2. The number of pairs in the interval  \mbox {\boldmath ${(2\lfloor\frac{n}{3}\rfloor,n]}$}}. To show that two Skolem sequences of order $n$ intersect in $[2\lfloor\frac{n}{3}\rfloor,n]$ pairs, we keep the Langford sequences in their normal positions and reverse the Skolem sequences: $S_d L_{d+1}^{n-d}=S_{n}$ and $\stackrel{\leftarrow}{S_d} L_{d+1}^{n-d}=S_{n}$. We obtain two $S_{n}$ with some pairs in common. These pairs will be in the interval $[2\lfloor\frac{n}{3}\rfloor,n]$.

{\bf Case 3. The number of pairs in the interval  \mbox {\boldmath ${(\lfloor\frac{n}{3}\rfloor, 2\lfloor\frac{n}{3}\rfloor]}$}}. To find these pairs we need a different construction which is outlined in the next section.

\subsection{Second Main Construction}

We construct new Skolem sequences using three different strings:

{\bf String \mbox {\boldmath ${A_t^n}$}:} this string (or shell) is a sequence formed by even and odd numbers starting with $n$ and some free spaces (holes) in the middle of the sequence. Let $t$ be the order of $A$ (i.e., the number of pairs in $A$). The sequence is:
\begin{center}
$n,n-2,\ldots,n-t+1, n-1, n-3,\ldots,n-t+2,\underbrace{--------} _{\mbox{$n-t-\lfloor\frac{t}{2}\rfloor$ free spaces}},n-t+1,\ldots,n-2,n,n-t+2,\ldots,n-3,n-1$.
\end{center}

For $n=12$ and $t=7$, the sequence is:
\begin{center}
$A_7^{12}=12,10,8,6,11,9,7,\underbrace{--}_{\mbox{2 spaces}},6,8,10,12,7,9,11$.
\end{center}

{\bf String \mbox {\boldmath ${B}$}:} this string is the space inside string $A$. We try to fit in this hole a (hooked) Skolem sequence, a $k$-extended Skolem sequence, a $2$-near Skolem sequence or a Langford sequence. In the example above, we can fit a Skolem sequence of order $1$: $S_{1}=(1,1)$.

{\bf String \mbox {\boldmath ${C}$}:} in this string we form a Langford or hooked Langford sequence from the elements left from strings $A$ and $B$.

For $n=12$ and $t=7$, the elements left from string $A_7^{12}$
and $B$ are $2,3,4,5$. So, we form a Langford sequence of
defect $2$ and order $4$: $L_{2}^{4}=(5,2,4,2,3,5,4,3).$

The Skolem sequence of order $12$ formed by this construction is:
\begin{center} $S_{12}=(\underbrace{12,10,8,6,11,9,7}_{\mbox{$A_7^{12}$}},\underbrace{1,1}_{\mbox{B}},\underbrace{6,8,10,12,7,9,11}_{\mbox{$A_7^{12}$}},\underbrace{5,2,4,2,3,5,4,3}_{\mbox{C}})$.
\end{center}
We form other Skolem sequences of order $12$ by reversing Strings $A_7^{12}$, $B$ or $C$:
\begin{enumerate}
\item keep strings $A_7^{12}$ and $B$ in their normal positions and reverse string $C$. This Skolem sequence of order $12$ with the initial Skolem sequence of order $12$ will have eight pairs in common, that is, seven pairs from string $A_7^{12}$ and one pair from string $B$.
\begin{center}
$(\underbrace{12,10,8,6,11,9,7}_{\mbox{$A_7^{12}$}},\underbrace{1,1}_{\mbox{B}},\underbrace{6,8,10,12,7,9,11}_{\mbox{$A_7^{12}$}},\underbrace{3,4,5,3,2,4,2,5}_{\mbox{$\stackrel{\leftarrow}{C}$}})$
\end{center}
\item keep strings $B$ and $C$ in their normal positions and reverse string $A_7^{12}$. This Skolem sequences of order $12$ with the initial Skolem sequence of order $12$ will have five pairs in common, that is, four pairs from string $C$ and one pair from string $B$.
\begin{center}
$(\underbrace{{11,9,7,12,10,8,6}}_{\mbox{$\stackrel{\leftarrow}{A_7^{12}}$}},\underbrace{1,1}_{\mbox{B}},\underbrace{{7,9,11,6,8,10,12}}_{\mbox{$\stackrel{\leftarrow}{A_7^{12}}$}},\underbrace{5,2,4,2,3,5,4,3}_{\mbox{C}})$.
\end{center}
\end{enumerate}

When we reverse string $A_t^n$, the odd numbers take place of the even numbers and vice versa. Therefore, $A_t^n$ and $\stackrel{\leftarrow}{A_t^n}$ are always disjoint.

Here is another way to construct a Skolem sequence of order $n$ using the second main construction. We denote a hooked Skolem sequence of order $n$ that has $s$ other pairs around it by $hS_{n} + s\; \textup{pairs}$.

\begin{ex}
For $n=16$ and $t=7$, we have a hole of six spaces between the two regions that form string $A_7^{16}$ of the construction. We fill this hole with a hooked Skolem sequences of order $2$ and two other pairs, $8$ and $9$ (these pairs are the largest pairs left from string $A_7^{16}$). Then we add a Langford sequence of order $5$ and defect $3$ at the end of the sequence. The new sequence is:
\begin{center}
$\underbrace{9}_{\mbox{B}},\underbrace{16,14,12,10,15,13,11}_{\mbox{$A_7^{16}$}},\underbrace{2,9,2,1,1,8}_{\mbox {B}},\underbrace{10,12,14,16,11,13,15}_{\mbox{$A_7^{16}$}},\underbrace{8}_{\mbox{B}},\linebreak \underbrace{5,6,7,3,4,5,3,6,4,7}_{\mbox {C}}$.
\end{center}
In this case the sequence can be written on short $A_7^{16}$, $B=hS_2+2\;\textup{pairs}, C=L_3^5$.
\end{ex}

We denote the number of pairs found in sequence $B$ by $ord(B)$. In the example above, $ord(B)=4$.

\section{The Intersection Spectrum of Skolem Sequences}

For $1\leq n\leq 9$ and $n\neq 5$, two (hooked) Skolem sequences intersect in $0,1,\ldots,n-3,n$ pairs (see Appendix B of the Supplement). Two Skolem sequences of order $5$ intersect in  $0,1,5$ pairs. Using the intersection spectrum of two (hooked) Skolem sequences of order $n$ for
 $1\leq n\leq 9$, we determine in Theorem \ref{main1} the intersection spectrum of two Skolem sequences of order $10\leq n\leq 29$. Then, using the intersection spectrum of two Skolem sequences of order $n$ for
 $1\leq n\leq 29$, we determine the intersection spectrum of two Skolem sequences of order $30\leq n\leq 89$. Continue the same procedure, we determine the intersection spectrum of two Skolem sequences of order $n$ for any $n$.

Since we want to find the intersection spectrum of two Skolem sequences for a specific value of $n$, we use the proof of the Theorem \ref{main1} backwards, i.e., to find the intersection spectrum of two Skolem sequences of order $n$ we use the intersection spectrum of two Skolem sequences of order $n_1$ where $n_1<\lfloor\frac{n}{3}\rfloor$. If $n_1\geq 10$, we find the intersection spectrum of two Skolem sequences of order $n_1$
 by using the intersection spectrum of two Skolem sequences of order $n_2$ where $n_2<\lfloor\frac{n_1}{3}\rfloor$. We continue the same procedure until $n_i\leq 9$ for some integer $i$. At this point the intersection spectrum of two Skolem sequences of a specific order $n$
 is satisfied.

We divide the proof of the Theorem \ref{main1} into six cases, $n\equiv 0,1,4,5,8,9\linebreak \mathrm{(mod\;12)}$. Then we divide each case into three subcases depending on the position of the pairs, $[0,\lfloor\frac{n}{3}\rfloor], (\lfloor\frac{n}{3}\rfloor,2\lfloor\frac{n}{3}\rfloor], (2\lfloor\frac{n}{3}\rfloor,n]$. In each case we have three types of constructions: general constructions where we use our two main constructions, recursive constructions where we use the existing intersection spectrum of two Skolem sequences of a smaller order and special constructions where we adjoin two or three Langford sequences and manipulate them to found the pairs not covered by the other two constructions.

\begin{theorem}\label{main1}
The necessary conditions are sufficient for two Skolem sequences of
order $n$ to intersect in $\{0,1,\ldots,n-3,n\}$ pairs.
\end{theorem}

\begin{proof} For $n$ pairs in common we take two copies of the same sequence. For two disjoint Skolem sequences of order $n$ or for two disjoint $L_2^n$, see \cite{shalaby}. For a near-Skolem sequence, see \cite{shalaby1}. For some small cases ($12\leq n\leq 100$) our constructions did not work. These small cases are not covered by the theorem and are listed in Appendix A for the case $n\equiv 0(\textup{mod}\, 12)$ and in Appendix A of the Supplement, for the remaining of the cases.

{\bf Necessity:} Two Skolem sequences of order $n$ do not intersect in $n-1$ pairs (obvious).

Two Skolem sequences do not intersect in $n-2$ pairs. Suppose $S_1$ and $S_2$ are two Skolem sequences of order $n$ that intersect in $n-2$ pairs. There are two differences left, say $a$ and $b$, and $4$ positions, $s_1,s_2,s_3,s_4$ left. Without loss of generality we assume that $s_1<s_2<s_3<s_4$ and that $a$ is in positions $s_1$ and $s_i$ in $S_1$. In $S_2$: $a$ cannot be in position $s_1$, so $b$ is in position $s_1$; $b$ cannot be then in position $s_i$, so $a$ is in $s_i$; $b$ cannot be in $s_4$, since $s_4-s_i>b$, so $a$ is in $s_4$. Similarly, $a$ could not be in $s_4$ in $S_1$ since $s_4-s_1>a$. Therefore, $a=s_i-s_1=s_4-s_i$ and $b=s_4-s_j=s_j-s_1$ for some $j\neq 1,i,4$, so $s_i=\frac{s_4-s_1}{2}=s_j$; a contradiction.

{\bf Sufficiency:  \mbox {\boldmath ${n\equiv 0,1,4,5,8,9\ \mathrm{(mod\; 12)}}$}}

{\bf The number of pairs in the interval  \mbox {\boldmath ${(2\lfloor\frac{n}{3}\rfloor,n]}$}:} To get $n-d$ pairs in common, take: $S_d L_{d+1}^{n-d}$ and $d(S_d) L_{d+1}^{n-d}$, for $d=3,4,\ldots,\lfloor\frac{n-1}{3}\rfloor$.

To get $2\lfloor\frac{n}{3}\rfloor+1$ pairs in common for $n\equiv 5\ \mathrm{(mod \;12)}$ take: $A=A_{2\lfloor\frac{n}{3}\rfloor-9}^n,\;B=L_3^8,\;C=hL_{11}^{\lfloor\frac{n}{3}\rfloor+1}*(2,0,2),D=(1,1)$; $A,B,C,D$ and $A,B,\stackrel{\leftarrow}{C}$ ($1$ pair, see the Supplement, Table 9, sequence 3), $D$.

To get $2\lfloor\frac{n}{3}\rfloor+1$ pairs in common for $n\equiv 8\ \mathrm{(mod\; 12)}, n\geq 56$ take: $A=A_{2\lfloor\frac{n}{3}\rfloor-9}^n$, $B=L_2^8$, $C=L_{10}^{\lfloor\frac{n}{3}\rfloor+3}, D=(1,1)$; $A,B,C,D$ and $A,B,\stackrel{\leftarrow}{C}$($1$ pair from Table \ref{n=4t}, sequence 1), $D$.

To prove the sufficiency for the pairs in the intervals ${[0,\lfloor\frac{n}{3}\rfloor]}$ and ${(\lfloor\frac{n}{3}\rfloor, 2\lfloor\frac{n}{3}\rfloor]}$ we split the remaining of the proof into six different cases.

{\bf Case 1.  \mbox {\boldmath ${n\equiv 0\ \mathrm{(mod\; 12)}}$}}

{\bf The number of pairs in the interval  \mbox {\boldmath ${[0,\lfloor\frac{n}{3}\rfloor]}$}:} To get the number of pairs in the interval $[0,\lfloor\frac{n}{3}\rfloor]$ for $n=12m$ and $m\geq3$ we adjoin a Skolem sequence of order $4m-3$ to a Langford sequence $L_{4m-2}^{8m+3}$, i.e., $S_{4m-3} L_{4m-2}^{8m+3}=S_{12m}$. Then we adjoin a different Skolem sequence of order $4m-3$ to the reverse of the Langford sequence, i.e., $S_{4m-3}' \stackrel{\leftarrow}{L_{4m-2}^{8m+3}}=S_{12m}'$.

We use the existing intersection spectrum of two Skolem sequences of order $4m-3$. Two Skolem sequences of order $4m-3$ intersect in $0,1,\ldots,4m-6,4m-3$ pairs (see Appendix B of the Supplement, for $1\leq 4m-3 \leq 9$ and Cases $1,5,9\ \mathrm{(mod\;12)}$ for $4m-3\geq 12$). For all the cases we use our first construction recursively as illustrated in Table \ref{spectrum}.

\begin{table}[hp]
\begin{center}
\begin{tabular}{|c|c|c|}
\hline
Skolem sequences & Pairs in common & Cases used to get the spectrum\\
\hline
$S_{12m},m\geq3$ & $sp(S_{4m-3})$ & $1,5,9\ \mathrm{(mod\;12)}$\\
\hline
$S_{12m+1},m\geq1$ & $sp(S_{4m})$ & $0,4,8\ \mathrm{(mod\;12)}$\\
\hline
$S_{12m+4},m\geq1$ & $sp(S_{4m+1})$ & $1,5,9\ \mathrm{(mod\;12)}$\\
\hline
$S_{12m+5},m\geq1$ & $sp(S_{4m+1})$ & $1,5,9\ \mathrm{(mod\;12)}$\\
\hline
$S_{12m+8},m\geq1$ & $sp(S_{4m+1})$ & $1,5,9\ \mathrm{(mod\;12)}$\\
\hline
$S_{12m+9},m\geq1$ & $sp(S_{4m})$ & $0,4,8\ \mathrm{(mod\;12)}$\\
\hline
\end{tabular}
\caption{Recursive constructions} \label{spectrum}
\end{center}
\end{table}

To get $0,1,\ldots,\lfloor\frac{n}{3}\rfloor-6, \lfloor\frac{n}{3}\rfloor-3$ pairs in common for $n\geq 36$ take: $sp(S_{\lfloor\frac{n}{3}\rfloor-3}) L^{n-\lfloor\frac{n}{3}\rfloor+3}_{\lfloor\frac{n}{3}\rfloor-2}$ and $sp(S_{\lfloor\frac{n}{3}\rfloor-3}) \stackrel{\leftarrow}{L^{n-\lfloor\frac{n}{3}\rfloor+3}_{\lfloor\frac{n}{3}\rfloor-2}}$ (the Langford sequence and its reverse have $0$ pairs in common from Table \ref{n=4t}, sequence 3).

To get $\lfloor\frac{n}{3}\rfloor-5$ pairs in common for $n=12m$ with $m\equiv 2\ \mathrm{(mod\;3)}$ and $m\geq 5$ take: $L_2^{\frac{n}{12}-1}, L_{\frac{n}{12}+1}^{\frac{n}{4}-4},(1,1)L_{\frac{n}{3}-3}^{\frac{2n}{3}+4}$ and $d(L_2^{\frac{n}{12}-1})$,$L_{\frac{n}{12}+1}^{\frac{n}{4}-4},\stackrel{\leftarrow}{(1,1)L_{\frac{n}{3}-3}^{\frac{2n}{3}+4}}$ ($s$ pairs, see the Supplement, Table 8, sequence 2, where $s=\frac{n}{12}-1$).

To get $\lfloor\frac{n}{3}\rfloor-5$ pairs in common for $n=12m$ with $m\equiv 0,1\ \mathrm{(mod\;3)}$ take: $L_2^{\lfloor\frac{n}{3}\rfloor-5},L_{\lfloor\frac{n}{3}\rfloor-3}^{n-\lfloor\frac{n}{3}\rfloor+4}(1,1)$ and $L_2^{\lfloor\frac{n}{3}\rfloor-5},\stackrel{\leftarrow}{L_{\lfloor\frac{n}{3}\rfloor-3}^{n-\lfloor\frac{n}{3}\rfloor+4}(1,1)}$ ($0$ pairs, see the Supplement, Table 8, sequence 2).

To get $\lfloor\frac{n}{3}\rfloor-4$ pairs in common for $n=12m$ with $m\equiv 0,2\ \mathrm{(mod\;3)}$ and $m\geq 2$ take: $S_{\lfloor\frac{n}{3}\rfloor-4} L^{n-\lfloor\frac{n}{3}\rfloor+4}_{\lfloor\frac{n}{3}\rfloor-3}$ and  $S_{\lfloor\frac{n}{3}\rfloor-4} \stackrel{\leftarrow}{L^{n-\lfloor\frac{n}{3}\rfloor+4}_{\lfloor\frac{n}{3}\rfloor-3}}$ ($0$ pairs from Table \ref{n=4t}, sequence 1).

To get $\lfloor\frac{n}{3}\rfloor-4$ pairs in common for $n=12m$ with $m\equiv 1\;\mathrm{(mod\;3)}$ and $m\geq 7$ take: $S_{\lfloor\frac{n}{3}\rfloor-4} L_{\lfloor\frac{n}{3}\rfloor-3}^{n-\lfloor\frac{n}{3}\rfloor+4}$ and $S_{\lfloor\frac{n}{3}\rfloor-4} \stackrel{\leftarrow}{L_{\lfloor\frac{n}{3}\rfloor-3}^{n-\lfloor\frac{n}{3}\rfloor+4}}$ ($0$ pairs from Table \ref{n=4t}, sequence 2).

To get $\lfloor\frac{n}{3}\rfloor-2$ pairs in common for $n=12m$ with $m\equiv 0,1\;\mathrm{(mod\;3)}$ and $m\geq 3$ take: $2-\textup{near}\;S_{\lfloor\frac{n}{3}\rfloor-2} hL^{n-\lfloor\frac{n}{3}\rfloor+2}_{\lfloor\frac{n}{3}\rfloor-1}*(2,0,2)$ and  $2-\textup{near}\;S_{\lfloor\frac{n}{3}\rfloor-2} \stackrel{\leftarrow}{hL^{n-\lfloor\frac{n}{3}\rfloor+2}_{\lfloor\frac{n}{3}\rfloor-1}}*(2,0,2)$ ($1$ pair, see the Supplement, Table 9, sequence 3).

To get $\lfloor\frac{n}{3}\rfloor-2$ pairs in common for $n=12m$ with $m\equiv 2\;\mathrm{(mod\;3)}$ take: $L_3^{\lfloor\frac{n}{3}\rfloor-3}(1,1),hL_{\lfloor\frac{n}{3}\rfloor}^{2\lfloor\frac{n}{3}\rfloor+1}*(2,0,2)$ and $L_3^{\lfloor\frac{n}{3}\rfloor-3}(1,1),\stackrel{\leftarrow}{hL_{\lfloor\frac{n}{3}\rfloor}^{2\lfloor\frac{n}{3}\rfloor+1}}*(2,0,2)$ ($0$ pairs, see the Supplement, Table 9, sequence 2).

To get $\lfloor\frac{n}{3}\rfloor-1$ pairs in common for $n\geq 24$ take: $2-\textup{near}\; S_{\lfloor\frac{n}{3}\rfloor-1} hL^{n-\lfloor\frac{n}{3}\rfloor+1}_{\lfloor\frac{n}{3}\rfloor}*(2,0,2)$ and $2-\textup{near}\; S_{\lfloor\frac{n}{3}\rfloor-1} \stackrel{\leftarrow}{hL^{n-\lfloor\frac{n}{3}\rfloor+1}_{\lfloor\frac{n}{3}\rfloor}}*(2,0,2)$ ($1$ pair, see the Supplement, Table 9, sequence 1).

To get $ \lfloor\frac{n}{3}\rfloor$ pairs in common for $n\geq 72$ take: $A=A_{2\lfloor\frac{n}{3}\rfloor-9}^n,\,B=L_4^7,\,C=hL_{11}^{\lfloor\frac{n}{3}\rfloor-1},\,D=S_3$; $A,B,C,D$ and $\stackrel{\leftarrow}{A},\stackrel{\leftarrow}{B}$ ($1$ pair from Table \ref{n=4t}, sequence 5), $C, d(D)$.

{\bf The number of pairs in the interval  \mbox {\boldmath ${(\lfloor\frac{n}{3}\rfloor, 2\lfloor\frac{n}{3}\rfloor]}$}:} We split this case in another eight subcases. These subcases will follow the same technique with a few differences from case to case.

{\bf  \mbox {\boldmath ${n=12m}$} with  \mbox {\boldmath ${m\equiv 2\;\mathrm{(mod\;8)}}$}:} We find the number of pairs in the interval $(\lfloor\frac{n}{3}\rfloor, 2\lfloor\frac{n}{3}\rfloor]$ using our second main construction recursively. The proof includes four sequences of intervals and nine isolated cases. Using our second main construction we get seven isolated cases ($\lfloor\frac{n}{3}\rfloor+1,\lfloor\frac{n}{3}\rfloor+2,2\lfloor\frac{n}{3}\rfloor-5,2\lfloor\frac{n}{3}\rfloor-4,2\lfloor\frac{n}{3}\rfloor-3,2\lfloor\frac{n}{3}\rfloor-1,2\lfloor\frac{n}{3}\rfloor$, denoted by $|$ in Figure \ref{figure}). Using special construction we get another two isolated cases ($\lfloor\frac{n}{3}\rfloor+3,2\lfloor\frac{n}{3}\rfloor-2$, denoted by $*$ in Figure \ref{figure}). We get the rest of the pairs in sequences of intervals as follows. We construct $\lfloor\frac{m}{8}\rfloor$ Skolem sequences in a certain way (we call these Skolem sequences, Sequence 1). Then we manipulate these Skolem sequences to obtain two sequences of intervals, the first left sequence and the first right sequence of intervals (first, third, etc in Figure \ref{figure}). We also construct other Skolem sequences and call them Sequence 2. We manipulate these Skolem sequences to obtain other two sequences of intervals, the second left and the second right sequence of intervals (second, fourth, etc in Figure \ref{figure}). The union of all these intervals and the isolated cases covers the whole interval ($\lfloor\frac{n}{3}\rfloor,2\lfloor\frac{n}{3}\rfloor$]. The nine isolated cases are described below.

\begin{figure}
\begin{center}
\begin{tikzpicture}[style=thick]

\draw (-9.1,5) node {$[\frac{n}{3}]$};
\draw (5,5) node {$2[\frac{n}{3}]$};
\draw (-2,4.5) node {$[-|-|-*----------------------|-|-|-*-|-]$};

\draw[gray] (-6.3,4) node {{$[----]$}};
\draw[gray] (1.2,4) node {{$[----]$}};

\draw (-5.5,3.5) node  {$[-----]$};
\draw (0.5,3.5) node {$[-----]$};

\draw[gray] (-5,3) node {{$[------]$}};
\draw[gray] (-0.2,3) node {{$[------]$}};

\draw (-4.5,2.5) node {$[-------]$};
\draw (-0.7,2.5) node {$[-------]$};

\draw (-2.8,2) node {$.$};
\draw (-2.8,1.75) node {$.$};
\draw (-2.8,1.5) node {$.$};

\draw (-4,0.75) node {{$[-------]$}};
\draw (-1,0.75) node {{$[-------]$}};


\draw (-8,2) node {${Two~left}$};
\draw (-8,1.6) node {$sequences$};
\draw (-8,1.2) node {$of~intervals$};

\draw (3.5,2) node {$Two~right$};
\draw (3.5,1.6) node {$sequences$};
\draw (3.5,1.2) node {$of~intervals$};

\end{tikzpicture}

\end{center}
\caption{Sequences 1 and 2}
\label{figure}
\end{figure}
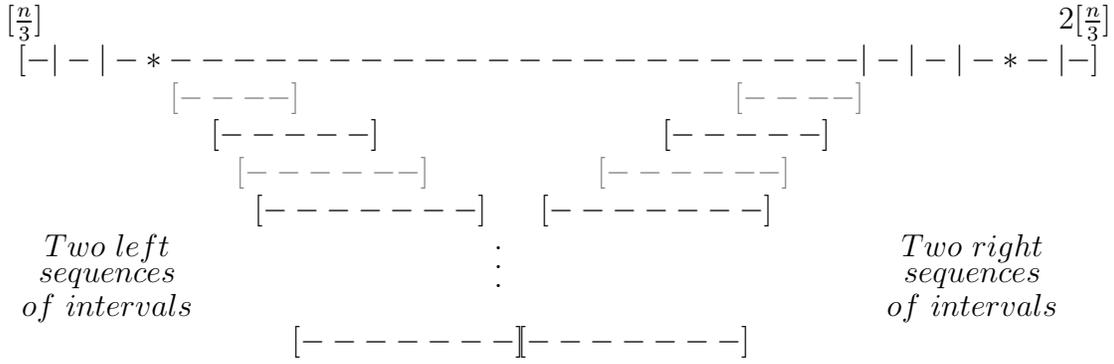

To get $\lfloor\frac{n}{3}\rfloor+1,\lfloor\frac{n}{3}\rfloor+2$ pairs in common for $n\geq24$ take: $A=A_{2\lfloor\frac{n}{3}\rfloor-5}^n;\;B=S_{4},\;C=L_{5}^{\lfloor\frac{n}{3}\rfloor+1}$; $A,sp(B),C$ and $\stackrel{\leftarrow}{A},sp(B),C$.

To get $11,\ldots,16,18$ pairs in common for $n=24$ take: $A,sp(B),C$ and $A,sp(B),\stackrel{\leftarrow}{C}$ ($0,1,3$ pairs from Table \ref{n=4t}, sequence 5).

To get $\lfloor\frac{n}{3}\rfloor +3$ pairs in common for $n\geq 72$ take: $A=A_{2\lfloor\frac{n}{3}\rfloor-9}^n,\;B=L_4^7,\;C=hL_{11}^{\lfloor\frac{n}{3}\rfloor-1},\;D=S_3$; $A,B,C,D$ and $\stackrel{\leftarrow}{A},\stackrel{\leftarrow}{B}$ ($1$ pair from Table \ref{n=4t}, sequence 5), $C,D$.

To get $2\lfloor\frac{n}{3}\rfloor-5, 2\lfloor\frac{n}{3}\rfloor-4,2\lfloor\frac{n}{3}\rfloor-1$ pairs in common for $n\geq36$ take: $A,sp(B),C$ and $A,sp(B),\stackrel{\leftarrow}{C}$ ($0$ pairs from Table \ref{n=4t}, sequence 3).

To get $2\lfloor\frac{n}{3}\rfloor-3$ pairs in common for $n\geq72$ take: $A=A_{2\lfloor\frac{n}{3}\rfloor-17}^n,\;B=S_{13},\;C=L_{14}^{\lfloor\frac{n}{3}\rfloor+4}$; $A,B,C$ and $A,B,\stackrel{\leftarrow}{C}$ ($1$ pair from Table \ref{n=4t}, sequence 1).

To get $2\lfloor\frac{n}{3}\rfloor -2$ pairs in common for $n\geq 48$ take: $A=A_{2\lfloor\frac{n}{3}\rfloor-9}^n,\;B=L_2^7,\;C=L_9^{\lfloor\frac{n}{3}\rfloor+1}(1,1)$; $A,B,C$ and $A,B,\stackrel{\leftarrow}{C}$ ($0$ pairs, see the Supplement, Table 8, sequence 1).

To get $2\lfloor\frac{n}{3}\rfloor$ pairs in common for $n\geq12$ take: $A=A_{2\lfloor\frac{n}{3}\rfloor-1}^n,\;B=S_{1},\;C=L_{2}^{\lfloor\frac{n}{3}\rfloor}$; $A,B,C$ and $A,B,d(C)$.

We get the rest of the pairs in sequences of intervals.

{\bf Sequence 1:} for $n\geq 72$,
\begin{center}
$A_1=A_{t_{1}=2\lfloor\frac{n}{3}\rfloor-17}^n,\;B_{1}=S_{13},\;C_{1}=L_{d_1=14}^{m_1=\lfloor\frac{n}{3}\rfloor+4}$

$A_2=A_{t_{2}=t_1-16}^n,\;B_{2}=S_{ord(B_1)+12},\;C_{2}=L_{d_2=d_1+12}^{m_2=m_1+4}$

$\ldots$

$A_{\lfloor\frac{m}{8}\rfloor}=A_{t_{\lfloor\frac{m}{8}\rfloor}=t_{\lfloor\frac{m}{8}\rfloor-1}-16}^n,\;B_{\lfloor\frac{m}{8}\rfloor}=S_{ord (B_{\lfloor\frac{m}{8}\rfloor-1})+12},\;C_{\lfloor\frac{m}{8}\rfloor}=L_{d_{\lfloor\frac{m}{8}\rfloor}=d_{\lfloor\frac{m}{8}\rfloor-1}+12}^{m_{\lfloor\frac{m}{8}\rfloor}=m_{\lfloor\frac{m}{8}\rfloor-1}+4}$
\end{center}
{\bf The first left sequence of intervals} covers the pairs in the intervals:
\begin{center}
$[a_1=\lfloor\frac{n}{3}\rfloor+4,b_1=\lfloor\frac{n}{3}\rfloor+14];\;A_1,sp(B_1),C_1$ and $\stackrel{\leftarrow}{A_1},sp(B_1),C_1$

$[a_2=a_1+4,b_2=b_1+16];\;A_2,sp(B_2),C_2$ and $\stackrel{\leftarrow}{A_2},sp(B_2),C_2$

$\ldots$

$[a_{\lfloor\frac{m}{8}\rfloor}=a_{\lfloor\frac{m}{8}\rfloor-1}+4,b_{\lfloor\frac{m}{8}\rfloor}=b_{\lfloor\frac{m}{8}\rfloor-1}+16];\;A_{\lfloor\frac{m}{8}\rfloor},sp(B_{\lfloor\frac{m}{8}\rfloor}),C_{\lfloor\frac{m}{8}\rfloor}$ and $\stackrel{\leftarrow}{A_{\lfloor\frac{m}{8}\rfloor}},sp(B_{\lfloor\frac{m}{8}\rfloor}),C_{\lfloor\frac{m}{8}\rfloor}$.
\end{center}
{\bf The first right sequence of intervals} covers the pairs in the intervals:
\begin{center}
$[e_1=2\lfloor\frac{n}{3}\rfloor-16,f_1=2\lfloor\frac{n}{3}\rfloor-6];\;A_1,sp(B_1),C_1$ and $A_1,sp(B_1),\stackrel{\leftarrow}{C_1}$

$[e_2=e_1-16,f_2=f_1-4];\;A_2,sp(B_2),C_2$ and $A_2,sp(B_2),\stackrel{\leftarrow}{C_2}$

$\ldots$

$[e_{\lfloor\frac{m}{8}\rfloor}=e_{\lfloor\frac{m}{8}\rfloor-1}-16,f_{\lfloor\frac{m}{8}\rfloor}=f_{\lfloor\frac{m}{8}\rfloor-1}-4];\;A_{\lfloor\frac{m}{8}\rfloor},sp(B_{\lfloor\frac{m}{8}\rfloor}),C_{\lfloor\frac{m}{8}\rfloor}$ and $A_{\lfloor\frac{m}{8}\rfloor},sp(B_{\lfloor\frac{m}{8}\rfloor}),\stackrel{\leftarrow}{C_{\lfloor\frac{m}{8}\rfloor}}$.
\end{center}
{\bf Sequence 2:} for $n\geq72$,
\begin{center}
$A_1'=A_{t_{1}=2\lfloor\frac{n}{3}\rfloor-21}^n,\;B_{1}'=S_{16},\;C_{1}'=L_{d_1=17}^{m_1=\lfloor\frac{n}{3}\rfloor+5}$

$A_2'=A_{t_{2}=t_1-16}^n,\;B_{2}'=S_{ord(B_1)+12},\;C_{2}'=L_{d_2=d_1+12}^{m_2=m_1+4}$

$\ldots$

$A_{\lfloor\frac{m}{8}\rfloor}'=A_{t_{\lfloor\frac{m}{8}\rfloor}=t_{\lfloor\frac{m}{8}\rfloor-1}-16}^n,\;B_{\lfloor\frac{m}{8}\rfloor}'=S_{ord (B_{\lfloor\frac{m}{8}\rfloor-1})+12},\;C_{\lfloor\frac{m}{8}\rfloor}'=L_{d_{\lfloor\frac{m}{8}\rfloor}=d_{\lfloor\frac{m}{8}\rfloor-1}+12}^{m_{\lfloor\frac{m}{8}\rfloor}=m_{\lfloor\frac{m}{8}\rfloor-1}+4}$
\end{center}
{\bf The second left sequence of intervals} covers the pairs in the interval:
\begin{center}
$[c_1=\lfloor\frac{n}{3}\rfloor+5,d_1=\lfloor\frac{n}{3}\rfloor+18];\;A_1',sp(B_1'),C_1'$ and $\stackrel{\leftarrow}{A_1'},sp(B_1'),C_1'$

$[c_2=c_1+4,d_2=d_1+16];\;A_2',sp(B_2'),C_2'$ and $\stackrel{\leftarrow}{A_2'},sp(B_2'),C_2'$

$\ldots$

$[c_{\lfloor\frac{m}{8}\rfloor}=c_{\lfloor\frac{m}{8}\rfloor-1}+4,d_{\lfloor\frac{m}{8}\rfloor}=d_{\lfloor\frac{m}{8}\rfloor-1}+16];\;A_{\lfloor\frac{m}{8}\rfloor}',sp(B_{\lfloor\frac{m}{8}\rfloor}'),C_{\lfloor\frac{m}{8}\rfloor}'$ and $\stackrel{\leftarrow}{A_{\lfloor\frac{m}{8}\rfloor}'},sp(B_{\lfloor\frac{m}{8}\rfloor}'),C_{\lfloor\frac{m}{8}\rfloor}'$.
\end{center}
{\bf The second right sequence of intervals} covers the pairs in the interval:
\begin{center}
$[g_1=2\lfloor\frac{n}{3}\rfloor-21,h_1=2\lfloor\frac{n}{3}\rfloor-8];\;A_1',sp(B_1'),C_1'$ and $A_1',sp(B_1'),\stackrel{\leftarrow}{C_1'}$

$[g_2=g_1-16,h_2=h_1-4];\;A_2',sp(B_2'),C_2'$ and $A_2',sp(B_2'),\stackrel{\leftarrow}{C_2'}$

$\ldots$

$[g_{\lfloor\frac{m}{8}\rfloor}=g_{\lfloor\frac{m}{8}\rfloor-1}-16,h_{\lfloor\frac{m}{8}\rfloor}=h_{\lfloor\frac{m}{8}\rfloor-1}-4];\;A_{\lfloor\frac{m}{8}\rfloor}',sp(B_{\lfloor\frac{m}{8}\rfloor}'),C_{\lfloor\frac{m}{8}\rfloor}'$ and $A_{\lfloor\frac{m}{8}\rfloor}',sp(B_{\lfloor\frac{m}{8}\rfloor}'),\stackrel{\leftarrow}{C_{\lfloor\frac{m}{8}\rfloor}'}$.
\end{center}
The union of all these intervals and the nine isolated cases covers the entire interval $(\lfloor\frac{n}{3}\rfloor, 2\lfloor\frac{n}{3}\rfloor]$, i.e., $\{\lfloor\frac{n}{3}\rfloor+1,\lfloor\frac{n}{3}\rfloor+2\} \cup \{2\lfloor\frac{n}{3}\rfloor,2\lfloor\frac{n}{3}\rfloor-1,2\lfloor\frac{n}{3}\rfloor-3,2\lfloor\frac{n}{3}\rfloor-4,2\lfloor\frac{n}{3}\rfloor-5\} \cup \{\lfloor\frac{n}{3}\rfloor+3,2\lfloor\frac{n}{3}\rfloor-2\}
 \cup_{i=1}^{\lfloor\frac{m}{8}\rfloor} [a_i,b_i]  \cup_{i=1}^{\lfloor\frac{m}{8}\rfloor} [c_i,d_i] \cup_{i=1}^{\lfloor\frac{m}{8}\rfloor} [e_i,f_i] \cup_{i=1}^{\lfloor\frac{m}{8}\rfloor} [g_i,h_i]=(\lfloor\frac{n}{3}\rfloor, 2\lfloor\frac{n}{3}\rfloor]$.

Note that the sequences of intervals described above overlap. Also note that we are not using all the intervals with all the orders. For some orders the pairs are covered by the first left and first right sequence of intervals, while for some other orders we have to use all the intervals.

For $n=12m$ with $m\equiv 0,1,3,4,5,6,7\;\mathrm{(mod\;8)}$ we are using the same sequences. For $m\equiv 3,4,5,6,7\;\mathrm{(mod\;8)}$, Sequence 1 has one extra step, i.e., $A_1,A_2,\ldots,A_{\lfloor\frac{m}{8}\rfloor+1}$, and Sequence 2 is not needed.

{\bf Case 2.  \mbox {\boldmath ${n\equiv 1\;\mathrm{(mod \;12)}}$}}

{\bf The number of pairs in the interval  \mbox {\boldmath ${[0,\lfloor\frac{n}{3}\rfloor]}$}:} To get $0,1,\ldots,\lfloor\frac{n}{3}\rfloor$ pairs in common for $n\geq 13$ take: $sp(S_{\lfloor\frac{n}{3}\rfloor})L^{n-\lfloor\frac{n}{3}\rfloor}_{\lfloor\frac{n}{3}\rfloor+1}$ and  $sp(S_{\lfloor\frac{n}{3}\rfloor})\stackrel{\leftarrow}{L^{n-\lfloor\frac{n}{3}\rfloor}_{\lfloor\frac{n}{3}\rfloor+1}}$ ($0,1,2 $ or $3$ pairs from Table \ref{n=4t}, sequence 5). We use here the existing intersection spectrum of two Skolem sequences of a smaller order, that is, two $S_{\lfloor\frac{n}{3}\rfloor}$ have $0,1,\ldots,\lfloor\frac{n}{3}\rfloor-3$ and $\lfloor\frac{n}{3}\rfloor-3$ pairs in common.

{\bf The number of pairs in the interval  \mbox {\boldmath ${(\lfloor\frac{n}{3}\rfloor, 2\lfloor\frac{n}{3}\rfloor]}$}:} To get $4,5,7$ pairs in common for $n=13$ take: $S_4 L_5^{9}$ and $S_4 \stackrel{\leftarrow}{L_5^{9}}$ ($0,1,3$ pairs from Table \ref{n=4t}, sequence 5).

To get $\lfloor\frac{n}{3}\rfloor$, $\lfloor\frac{n}{3}\rfloor+1$, $\lfloor\frac{n}{3}\rfloor+2$ for $n\geq 25$ take: $S_{\lfloor\frac{n}{3}\rfloor} L_{\lfloor\frac{n}{3}\rfloor+1}^{2\lfloor\frac{n}{3}\rfloor+1}$ and  $S_{\lfloor\frac{n}{3}\rfloor} \stackrel{\leftarrow}{L_{\lfloor\frac{n}{3}\rfloor+1}^{2\lfloor\frac{n}{3}\rfloor+1}}$ ($0,1,2$ pairs from Table \ref{n=4t}, sequence 5).

To get $2\lfloor\frac{n}{3}\rfloor-3$ pairs in common for $n\geq 37$ take: $A=A_{2\lfloor\frac{n}{3}\rfloor-3}^n, B=L_2^3,C=L_5^{\lfloor\frac{n}{3}\rfloor}(1,1)$; $A,B,C$ and $A,\stackrel{\leftarrow}{B}$ ($0$ pairs from Table \ref{n=4t}, sequence 5), $\stackrel{\leftarrow}{C}$ ($0$ pairs, see the Supplement, Table 8, sequence 2).

To get $2\lfloor\frac{n}{3}\rfloor-2$ pairs in common for $n\geq49$ take: $A=A_{2\lfloor\frac{n}{3}\rfloor-11}^n;\;B=S_{9};\;C=L_{10}^{\lfloor\frac{n}{3}\rfloor+3}$; $A,B,C$ and $A,B,\stackrel{\leftarrow}{C}$ ($0$ pairs from Table \ref{n=4t}, sequence 3).

To get $2\lfloor\frac{n}{3}\rfloor-1$ pairs in common for $n\geq 73$ take:  $A=A_{2\lfloor\frac{n}{3}\rfloor-15}^n;\;B=S_{12};\;C=L_{13}^{\lfloor\frac{n}{3}\rfloor+4}$; $A,B,C$ and $A,B,\stackrel{\leftarrow}{C}$ ($2$ pairs from Table \ref{n=4t}, sequence 1).

To get $2\lfloor\frac{n}{3}\rfloor$ pairs in common for $n\geq 37$ take: $A=A_{2\lfloor\frac{n}{3}\rfloor-3}^n, B=L_2^3,C=L_5^{\lfloor\frac{n}{3}\rfloor}(1,1)$; $A,B,C$ and $A,B,\stackrel{\leftarrow}{C}$ ($0$ pairs, see the Supplement, Table 8, sequence 2).

Sequences 1 and 2, and the left and right intervals are obtained using the same procedure as in $n\equiv 0\;\mathrm{(mod\;12)}$. The only difference is in the number of steps for sequence 1 and sequence 2. Table \ref{steps} gives the number of steps in sequences 1 and 2. Since the procedure of obtaining sequences 1 and 2 and the left and right intervals are identical with the case $n\equiv 0\;(mod\;12)$, we only list below the starting sequences and intervals.

Sequence 1 for $n\geq 49$, starts with $A_1=A_{t_{1}=2\lfloor\frac{n}{3}\rfloor-11}^n,\;B_{1}=S_{9},\;C_{1}=L_{d_1=10}^{m_1=\lfloor\frac{n}{3}\rfloor+3}$. The first left sequence of intervals starts with $[a_1=\lfloor\frac{n}{3}\rfloor+3,b_1=\lfloor\frac{n}{3}\rfloor+9]$ and the first right sequence of intervals starts with $[e_1=2\lfloor\frac{n}{3}\rfloor-11,f_1=2\lfloor\frac{n}{3}\rfloor-5]$.

Sequence 2 for $n\geq 73$, starts with $A_1'=A_{t_{1}=2\lfloor\frac{n}{3}\rfloor-15}^n,\;B_{1}'=S_{12},\;C_{1}'=L_{d_1=13}^{m_1=\lfloor\frac{n}{3}\rfloor+4}$. The second left sequence of intervals starts with $[c_1=\lfloor\frac{n}{3}\rfloor+4,d_1=\lfloor\frac{n}{3}\rfloor+13]$ and the second right sequence of intervals starts with $[g_1=2\lfloor\frac{n}{3}\rfloor-13,h_1=2\lfloor\frac{n}{3}\rfloor-4]$.

\begin{table}[hp]
\begin{center}
\begin{tabular}{|c|c|c|}
\hline
$m$ & Sequence 1 & Sequence 2\\
\hline
$0,1,2,3,4,5\;\mathrm{(mod\;8)}$ & $\lfloor\frac{m}{8}\rfloor+1$ & $\lfloor\frac{m}{8}\rfloor$\\
\hline
$6\;\mathrm{(mod\;8)}$ & $\lfloor\frac{m}{8}\rfloor+1$ & $\lfloor\frac{m}{8}\rfloor+1$\\
\hline
$7\;\mathrm{(mod\;8)}$ & $\lfloor\frac{m}{8}\rfloor+2$ & $\lfloor\frac{m}{8}\rfloor+1$\\
\hline
\end{tabular}
\caption{The number of steps in sequences 1 and 2} \label{steps}
\end{center}
\end{table}

The cases $n\equiv 4,5,8,9 (\textup{mod}\; 12)$ are listed in the Supplement.
\end{proof}

\section{Applications to  \mbox {\boldmath ${\lambda}$}-Fold Cyclic Triple System}

Our aim in this section is to produce necessary and sufficient
conditions for a vector to be the fine structure of a CTS$(v,2)$ for
$v\equiv1,3\;\mathrm{(mod\;6)}$, $v\neq 9$ and of a CTS$(v,\lambda)$ for
$v\equiv 1,7\;\mathrm{(mod\;24)}$ and $\lambda=3,4$. Then we extend this to the
fine structure of a cyclic directed triple system and a cyclic
Mendelsohn triple system.

\begin{df} The fine structure of a CTS$(v,\lambda)$ is the vector $(c_{1},c_{2},\ldots,c_{\lambda})$, where $c_{i}$ is the number of base blocks repeated exactly $i$ times in the cyclic triple system.
\end{df}

\begin{df}
Let $\mathbb{Z}_{v}$ be an additive abelian group. Then the subsets $D_i=\{d_{i1},d_{i2},d_{i3}\}$ form a $(v,3,\lambda)$ difference system if every non-zero element of $\mathbb{Z}_{v}$ occurs $\lambda$ times among the differences $d_{ix}-d_{iy}$ ($i=1,\ldots,n$ and $x,y=1,2,3$). The sets $D_{i}$ are called base blocks.
\end{df}

\begin{remark}
If $D_1,\ldots,D_n$ form a $(v,3,\lambda)$ difference system then the translates of the base blocks, $D_i+g=\{d_{i1}+g,d_{i2}+g,d_{i3}+g\}$, $g\in \mathbb{Z}_v$ form a CTS$(v,\lambda)$.
\end{remark}

We define $Int_c(v)=\{k$: there exist two cyclic triple systems of order $v$
with $k$ base blocks in common\}.

\subsection{The Fine Structure of a CTS\mbox {\boldmath ${(v,2)}$}}

We use Skolem sequences to find the fine structure of a CTS$(v,2)$ for
$v\equiv 1,3\;\mathrm{(mod\;6)}$ and $v\neq 9$.

\begin{theorem}\label{th1}
For all $n\in \mathbb{N}$, $Int_{c}(6n+1)=\{0,1,2,\ldots,n\}.$
\end{theorem}
\begin{proof} We divide the proof into two cases depending on $n$. 

{\bf Case 1.  \mbox {\boldmath ${n\equiv 0,1\;\mathrm{(mod\;4)}}$}}. Let $S_{n}$ be a Skolem sequence of order $n$ and let $\{(a_i,b_i)| 0\leq i\leq n$\}, be the pairs determined by this Skolem sequence.

To get two CSTS$(6n+1)$ which are disjoint take the base blocks of the form $\{\{0,a_i+n,b_i+n\}\ \mathrm{(mod\;6n+1)}|i=1,\ldots,n\}$ and $\{\{0,i,b_{i}+n\}\ \mathrm{(mod\;6n+1)}| i=1,\ldots,n\}$.

To get two CSTS$(6n+1)$ with $j=1,\ldots,n-1$ base blocks in common, take the base blocks below for  $j=1,\ldots, n-1$:
\begin{enumerate}
\item $\{\{0,a_i+n,b_i+n\}\ \mathrm{(mod\;6n+1)}|i=1,\ldots,j\}$ and $\{\{0,i,b_{i}+n\}\ \mathrm{(mod\;6n+1)}|\linebreak i=j+1,\ldots,n\}$
\item $\{\{0,a_{i}+n,b_{i}+n\}\ \mathrm{(mod\;6n+1)}|i=1,\ldots,n\}.$
\end{enumerate}

For two systems with $n$ base blocks in common, take $\{0,a_i+n,b_i+n\}\ \mathrm{(mod\;6n+1)}|i=1,\ldots,n\}$ twice.

{\bf Case 2.  \mbox {\boldmath ${n\equiv 2,3\;\mathrm{(mod\;4)}}$}}. Let $hS_{n}$ be a hooked Skolem sequence of order $n$ and apply the same arguments as in Case $1$.
\end{proof}

\begin{corollary}\label{fine}
For all $n\in \mathbb{N}$, the fine structure of a CTS$(6n+1,2)$ is the vector $(2n-2i,i)$ for $0\leq i\leq n$.
\end{corollary}

\begin{proof} Two CSTS$(6n+1)$ from Theorem \ref{th1} gives a cyclic two-fold triple system of order $6n+1$ with $0,1,\ldots,n$ base blocks repeated twice.
\end{proof}

\begin{corollary}\label{dts}
For all $n\in \mathbb{N}$, the fine structure of a cyclic DTS$(6n+1,2)$ (respectively a cyclic MTS$(6n+1,2)$) is the vector $(4n-4i,i)$ for $0\leq i\leq n$.
\end{corollary}

\begin{proof} Replace each block $\{a,b,c\}$ of a CTS$(6n+1,2)$ in the Corollary \ref{fine} with $\{[a,b,c]$, $[c,b,a]\}$ (respectively $\{\langle a,b,c\rangle$, $\langle a,c,b\rangle\}$). The collection of blocks form a cyclic directed triple system (respectively a cyclic Mendelsohn triple system) of order $v$ with $0,2,\ldots,2n$ directed triples repeated twice.
\end{proof}

\begin{theorem}\label{th2}
For all $n\in \mathbb{N}, n\neq 1$, $Int_{c}(6n+3)=\{1,2,\ldots,n+1\}$.
\end{theorem}

\begin{proof}
Similar to the proof of Theorem \ref{th1}.
\end{proof}

\begin{corollary}\label{dts2}
For all $n\in \mathbb{N}, n\neq 1$, the fine structure of a cyclic DTS$(6n+3,2)$ and of a cyclic MTS$(6n+3,2)$ is the vector $(4n-4i+4,2i)$ for $1\leq i\leq n+1$.
\end{corollary}

\begin{proof} Replace each block $\{a,b,c\}$ in the Corollary \ref{fine} with $\{[a,b,c]$, $[c,b,a]\}$ (respectively $\{\langle a,b,c \rangle$, $\langle a,c,b\rangle\}$).
\end{proof}

\subsection{The Fine Structure of a CTS\mbox {\boldmath ${(v,3)}$}}

We use next the intersection spectrum of two Skolem sequences of
order $n$ to determine the fine structure of a CTS$(v,3)$, a cyclic
DTS$(v,3)$ and a cyclic MTS$(v,3)$ for $v\equiv 1,7\;\mathrm{(mod\;24)}$.
Since any two of $(c_1,c_2,c_3)$ determine the third, we use the
following notation for the fine structure: $(t,s)$ is said to be the fine structure of a CTS$(v,3)$ if $c_2=t$
and $c_3=s$ (and hence $c_1=3n-2t-3s$, where $n=\frac{v-1}{6}$). We
need to know $(t,s)$ which can possibly arise as a fine structure.
For $v\equiv 1,7\;\mathrm{(mod\;24)}$, $n=\frac{v-1}{6}$, we define $CFine(v,3)=\{(t,s)|0\leq s\leq n, 0\leq t\leq n-s\}$.

$CFine(v,3)$ is the set of all fine structures which arise in a CTS$(v,3)$. We use the intersection spectrum of two Skolem sequences of order $n$ to show, $\{(t,s)|0\leq s\leq n, 0\leq t\leq n-s\}$ is the fine structure of a CTS$(v,3)$ with the possible exceptions $\{(0,n-1),(0,n-2),(1,n-2)\}$.

\begin{lemma}
Let CTS$(v,3)$ be a cyclic three-fold triple system and $n=\frac{v-1}{6}$. If $(t,s)\in CFine(v,3)$, then $0\leq s\leq n$ and $0\leq t\leq n-s$.
\end{lemma}

\begin{proof} Let CTS$(v,3)$ be a cyclic three-fold triple system, $v\equiv 1,7\ \mathrm{(mod\;24)}$ and $n=\frac{v-1}{6}$. Then there exists a difference system $D_1,\ldots,D_n$ such that all the non-zero elements in $\mathbb{Z}_v$ appear as a difference in $D_1,\ldots,D_n$. Each difference will appear three times since $\lambda=3$.

{\bf Case 1.} Let $s> n$ and $s=n+r$, where $r\geq 1$. Assume $s$ base blocks appear three times in a system. Since each base block uses six distinct differences, there will be $6n+6r$ distinct differences used. This is impossible since there are $6n$ distinct differences in total.

{\bf Case 2.} Let $t>n-s$ and $t=n-s+r$, where $r\geq 1$. If $s$ base blocks are repeated three times, these will use $6s$ distinct differences, each three times. So there will be $6n-6s$ distinct differences left for the other base blocks.
If $t=n-s+r$, where $r\geq 1$ base blocks appear twice, these will have to use $6n-6s+6r$ distinct differences. This is impossible since there are only $6n-6s$ distinct differences left.
\end{proof}


\begin{theorem}\label{1}
Let $v\equiv 1, 7\;\mathrm{(mod\;24)}$ with $v\geq7$ and $n=\frac{v-1}{6}$. Then the vector $(n-i,i)$ for $i=0,\ldots,n$ is the fine structure of a CTS$(v,3)$.
\end{theorem}

\begin{proof} Let $S_{n}$ be a Skolem sequence of order $n$ and let $\{(a_{i},b_{i})|1\leq i\leq n\}$ be the pairs determined by this Skolem sequence.

To get a CTS$(v,3)$ with the fine structure $(n,0)$, take the following three systems $\{\{0,a_{i}+n,b_{i}+n\}\ \mathrm{(mod\;6n+1)},i=1,\ldots,n\}$, $\{\{0,i,b_{i}+n\}\ \mathrm{(mod\;6n+1)},i=1,\ldots,n\}$ and $\{\{0,a_{i}+n,b_{i}+n\}\ \mathrm{(mod\;6n+1)},i=1,\ldots,n\}$.

To get a CTS$(v,3)$ with the fine structure $(n-i,i),i=1,\ldots,n-1$, repeat the process below for $j=1,\ldots,n-1$:
\begin{enumerate}
\item $\{\{0,a_{i}+n,b_{i}+n\},i=1,\ldots,n\}$
\item $\{\{0,a_{i}+n,b_{i}+n\},i=1,\ldots,j\}$ and $\{\{0,i,b_{i}+n\}$ for $i=j+1,\ldots,n\}$
\item $\{\{0,a_{i}+n,b_{i}+n\},i=1,\ldots,n\}$.
\end{enumerate}

To get $(0,n)$ as a fine structure, take $\{0,a_{i}+n,b_{i}+n\},i=1\ldots,n$ three times.
\end{proof}


\begin{theorem}\label{2}
Let $v\equiv 1,7\;\mathrm{(mod\;24)}$ with $v\geq25$ and $n=\frac{v-1}{6}$, and let $p$ be the number of pairs in common between two Skolem sequences of order $n$. The vector $(p-i+j,i)$ with $i=0,\ldots,p$ and $j=0,\ldots,n-i$ is the fine structure of a CTS$(v,3)$.
\end{theorem}

\begin{proof} Let $S_{n}$ and $S_{n}'$ to be two Skolem sequences of order $n$ with $p$ pairs in common, $0\leq p\leq n-3$. Let $\{(a_{i},b_{i})|1\leq i\leq n\}$ and $\{(\alpha_{i},\beta_{i})|1\leq i\leq n\}$ be the pairs determined by these two Skolem sequences and take the following three CSTS$(6n+1)$ for $i=1,\ldots,n$: $\{\{0,a_{i}+n,b_{i}+n\}\ \mathrm{(mod\;6n+1)}\}$, $\{\{0,i,b_{i}+n\}\ \mathrm{(mod\;6n+1)}\}$ and $\{\{0,\alpha_{i}+n,\beta_{i}+n\}\ \mathrm{(mod\;6n+1)}\}$.

Since the two Skolem sequences intersect in $p$ pairs, exactly $p$ base blocks from the first system will be identical with $p$ base blocks in the third system.
Replacing the base blocks $\{0,a_{i}+n,b_{i}+n\}$ by the base blocks $\{0,i,b_{i}+n\}$ as in Theorem \ref{1}, we get a CTS$(v,3)$ with the fine structure $(p-i+j,i)$ for $i=0,\ldots,p$ and $j=0,\ldots,n-i$.
\end{proof}


\begin{theorem}\label{3}
Let $v\equiv 1, 7\;\mathrm{(mod\;24)}$ with $v\geq25$ and $n=\frac{v-1}{6}$. The vector $(j,i)$ with $i=0,\ldots,n-3$ and $j=0,\ldots,n-i$ is the fine structure of a CTS$(v,3)$ where $v\neq 31$. For $v=31$, $(j,i)$ with $i=0,1$ and $j=0,\ldots,n-i$ is the fine structure of a CTS$(31,3)$.
\end{theorem}
\begin{proof} Let $v\equiv 1, 7\;\mathrm{(mod\;24)}$ with $v\geq25, v\neq 31$ and apply Theorem \ref{2} for $p=0,\ldots,n-3$.
Let $v=31$ and apply Theorem \ref{2} for $p=0,1$.
\end{proof}

And now we can generalize these results to the fine structure of a cyclic  DTS$(v,3)$ and a cyclic MTS$(v,3)$. The fine structure is defined in a similar way, with the only exception that $c_1=6n-4t-6s$.

\begin{corollary}\label{cor1}
Let $v\equiv 1, 7\;\mathrm{(mod\;24)}$ with $v\geq25,\;v\neq31$. Let $n=\frac{v-1}{6}$ and let $p$ be the number of pairs in common between two Skolem sequences of order $n$. Then the vector $(2j,2i)$ for $i=0,\ldots,n-3$ and $j=0,\ldots,n-i$, and the vector $(2n-2i,2i)$ for $i=0,\ldots,n$ is the fine structure of a cyclic DTS$(v,3)$ and of a cyclic MTS$(v,3)$. For $v=31$, $(2j,2i)$ for $i=0,1$ and $j=0,\ldots,5-i$, and the vector $(10-2i,2i)$ for $i=0,\ldots,5$ is the fine structure of a DTS$(31,3)$ and of a MTS$(31,3)$.
\end{corollary}

\begin{proof} Replace each triple $\{a,b,c\}$ of the cyclic Steiner triple systems in Theorem \ref{3} by the directed triples $\{[a,b,c]$, $[c,b,a]\}$ (respectively by the cyclic triples $\{\langle a,b,c\rangle$, $\langle a,c,b\rangle \}$).
\end{proof}

\subsection{The Fine Structure of a CTS\mbox {\boldmath ${(v,4)}$}}

We use the intersection spectrum of two Skolem sequences of order $n$ to determine the fine structure of a CTS$(v,4)$, where $v\equiv 1,7\;\mathrm{(mod\;24)}$ and $n=\frac{v-1}{6}$. Since any three of $(c_1,c_2,c_3,c_4)$ determine the fourth, we use the following notation for the fine structure: $(t,s,u)$ is said to be the fine structure of a CTS$(v,4)$ if $c_2=t$, $c_3=s$, $c_4=u$ and hence, $c_1=4n-4u-3s-2t$.

We define $CFine(v,4)=\{(t,s,u)| 0\leq u\leq n, \;0\leq s\leq n-u,\;0\leq t\leq 2n-2u-2s\}$. Since two Skolem sequences of order $n$ do not intersect in $n-1$
or $n-2$ pairs, we are not able to find the entire spectrum of the
fine structure of a CTS$(v,4)$. We use the intersection spectrum of
two Skolem sequences to find the fine structure of CTS$(v,4)$ with a
few possible exceptions. The following vectors are the possible
exceptions in the fine structure of a CTS$(v,4)$:
$(0,0,n-1),(1,0,n-1),(0,0,n-2),(1,0,n-2),(2,0,n-2),(3,0,n-2),(0,1,n-2),(1,1,n-2)(0,n-2-i,i),(1,n-2-i,i),(2,n-2-i,i),(3,n-2-i,i),(0,n-1-i,i),(1,n-1-i,i)$,
$0\leq i\leq n-3$. These vectors cannot be found using Skolem sequences.

\begin{theorem}
Let $v\equiv 1,7\;\mathrm{(mod\;24)}$ and $n=\frac{v-1}{6}$. If $(t,s,u)\in CFine(v,4)$ then $0\leq u\leq n, \;0\leq s\leq n-u$ and $0\leq t\leq 2n-2u-2s$.
\end{theorem}

\begin{proof} Let CTS$(v,4)$ be a cyclic four-fold triple system, $v\equiv 1,7\;\mathrm{(mod\;24)}$ and $n=\frac{v-1}{6}$. Then there exists a difference system $\{D_1,\ldots,D_n\}$ such that all the non-zero elements in $\mathbb{Z}_v$, which are $6n$ in total, appear as a difference in $\{D_1,\ldots,D_n\}$. Each difference will appear four times since $\lambda=4$.

{\bf Case 1.} Assume $u> n$ and let $u=n+r$, where $r\geq 1$. If $u$ base blocks appear four times in a system, and each base block gives six distinct differences, they will use $6n+6r$ distinct differences. This is impossible since they are $6n$ in total.

{\bf Case 2.} Assume $s>n-u$ and let $s=n-u+r$, $r\geq 1$. So there are $6n$ distinct differences in total and if $u$ base blocks are repeated four times, these will use $6u$ differences, each four times. So there will be $6n-6u$ distinct differences left for the other base blocks.
If $s=n-u+r$, $r\geq 1$ base blocks appear three times in a system, these will have to use $6n-6u+6r$ differences. This is impossible since they are only $6n-6s$ differences left.

{\bf Case 3.} Assume $t>2n-2u-2s$ and let $t=2n-2u-2s+r$, $r\geq 1$. There are $u$ base blocks which appear four times in the system and will use $6u$ differences. There are $s$ base blocks which are repeated three times and use $6s$ differences each three times.

{\bf Case 4.} Assume $t=2n-2u-2s+r$ base blocks are repeated twice. These will use $6(2n-2s-2u+r)$ differences. This is impossible since there are only $6n-6u-6s$ differences left.
\end{proof}

\begin{theorem}
Let $v\equiv 1,7\;\mathrm{(mod\;24)}$ where $v\geq 7$ and $n=\frac{v-1}{6}$. Then the vector $(2n-2i,i-j,j)$ with  $0\leq j\leq n$ and $j\leq i\leq n$ is the fine structure of a CTS$(v,4)$.
\end{theorem}
\begin{proof} Take a Skolem sequence twice and use the same argument as in Theorem \ref{1}.
\end{proof}

\begin{theorem}\label{9}
Let $v\equiv 1,7\;\mathrm{(mod\;24)}$ with $v\geq25$ and $n=\frac{v-1}{6}$, and let $p$ be the number of pairs in common between two Skolem sequences of order $n$. Then the vector $(k,i,j)$, with $0\leq j\leq p$, $0\leq i\leq p-j$ and $2p-2i-2j\leq k\leq 2n-2i-2j$ is the fine structure of a CTS$(v,4)$.
\end{theorem}
\begin{proof} Similar to Theorem \ref {2}.
\end{proof}

\begin{theorem}\label{10}
Let $v\equiv 1,7\;\mathrm{(mod\;24)}$ with $v\geq25$ and $n=\frac{v-1}{6}$. Then the vector $(k,i,j)$ with $0\leq j\leq n-3$, $0\leq i\leq n-3-j$ and $0\leq k\leq 2n-2i-2j$ is the fine structure of a CTS$(v,4)$.
\end{theorem}

\begin{proof} Apply Theorem \ref{9} for $p=0,1,\ldots,n-3$.
\end{proof}

We can extend these results to the fine structure of a cyclic DTS$(v,4)$ and a cyclic MTS$(v,4)$. Here $c_1=8n-8j-6i-4k$. Again, we will not get the full spectrum.

\begin{corollary}
Let $v\equiv 1,7\;\mathrm{(mod\;24)}$ with $v\geq25$. Then the vector $(2k,2i,2j)$ with $0\leq j\leq n-3$, $0\leq i\leq n-3-j$ and $0\leq k\leq 2n-2i-2j$ and the vector $(4n-4i,2i-2j,2j)$ with $0\leq j\leq n$ and $j\leq i\leq n$ are the fine structure of a cyclic DTS$(v,4)$ and of a cyclic MTS$(v,4)$.
\end{corollary}

\begin{proof} Replace each triple $\{a,b,c\}$ of the cyclic systems in Theorem \ref{9} and Theorem \ref{10} by the directed triples $\{[a,b,c]$, $[c,b,a]\}$ (respectively by the cyclic triples $\{\langle a,b,c\rangle$, $\langle a,c,b\rangle \}$).
\end{proof}

\section{Conclusion and Future Research}

We proved that two Skolem sequences of
order $n$ have $0,1,\ldots,n-3,n$ pairs in common. Using the
intersection spectrum of two Skolem sequences we determined the fine structure of
a CTS$(v,3)$, a cyclic DTS$(v,3)$ and a cyclic MTS$(v,3)$ for $v\equiv
1,7\;\mathrm{(mod\;24)}$. We also determined, with a few possible exceptions, the fine
structure of a CTS$(v,4)$, a cyclic DTS$(v,4)$ and a cyclic MTS$(v,4)$
and we determined the number of possible repeated base blocks in
a CTS$(v,2)$, a cyclic DTS$(v,2)$ and a cyclic MTS$(v,2)$  for $v\equiv
1,3\;\mathrm{(mod\;6)}$ and $v\neq 9$. Open questions include:
\begin{enumerate}
\item Complete the solution for all the remaining cases for two disjoint Langford sequences of order $n$ and all admissible defects.
\item Find the intersection spectrum of two hooked Skolem sequences of order $n$.
\item Find the intersection spectrum of two Rosa and two hooked Rosa sequences of order $n$.
\item Find the exceptions in the fine structure of a four-fold triple system for $v\equiv 1,7\;\mathrm{(mod\;24)}$.
\item Find the fine structure of a CTS$(v,3)$ and of a CTS$(v,4)$ for $v\equiv 13, 19\;\mathrm{(mod\:24)}$.
\end{enumerate}



\begin{center}
 {\bf APPENDIX A}
\end{center}

Here we list the small cases for $n\equiv 0(\textup{mod}\;12)$. For some cases we used special constructions and for others we used a computer. We used a hill-climbing algorithm as in \cite{shalaby4} to construct Skolem sequences. To simplify our writing, we write a Skolem sequence by using each element once. For
example, $(1,1,4,2,3,2,4,3)$ will be written $1,4,2,3$. We place
the first element, say $i$, in the first position, then, we place
the same element, $i+1$ positions to its right. We place the
second element, say $j$, in the next available position from the left,
and place the same element $j+1$ positions to its right.
Continuing this procedure until all the elements are placed, we get
a Skolem sequence of order $n$. When two Skolem sequences have pairs
in common we underline the pairs which are in common. For example,
$\underline{1},4,2,3$ and $3,4,2$ means we have two Skolem sequences
of order $4$ with one pair in common. For the second Skolem sequence we will
not write the underlined pairs since we already know the positions
for those pairs. So to construct the second Skolem sequence, we
place the underlined elements on the same positions as in the first
sequence and then, we use the same procedure as before to place all
the other elements.  For simplicity, for a large sequence, we use $1,\ldots,9,10=a,\ldots,35=z,36=A,\ldots,61=Z$.

\mbox {\boldmath ${n\equiv 0\;\mathrm{(mod\;12)}}$}

For \underline{$n=12$} the small cases are $\{1,2,3,4,6,7\}$.

$1$ pair:  $S_1,L_2^{11}$ and $S_1,d(L_2^{11})$.

$2$ pairs: $4,\underline{a},b,7,\underline{9},c,5,8,6,3,2,1$ and $c,8,6,b,7,1,3,4,5,2$.

$3$ pairs: $\underline{2},8,4,7,\underline{a},c,9,b,1,5,6,\underline{3}$ and $5,6,9,b,4,c,8,7,1$.

$4$ pairs: $2-\textup{near}\; S_3,hL_4^9*(2,0,2)$ and $2-\textup{near} S_3,\stackrel{\leftarrow}{hL_4^9*(2,0,2)}$ ($2$ pairs, see the Supplement, Table 9, sequence 1)

$6$ pairs: $\underline{c},\underline{a},\underline{8},\underline{6},\underline{4},2,b,\underline{5},9,7,3,1$ and $b,9,7,3,2,1$.

$7$ pairs: $4,\underline{9},\underline{5},\underline{a},\underline{6},\underline{8},7,c,b,\underline{3},\underline{1},2$
and $c,4,b,7,2$.

For \underline{$n=24$} the small cases are $\{1,2,3,4,5,6,7,8\}$.

$1-5,8$ pairs: $hS_7$ with two different $hL_8^{17}$ which have one pair in common. The two hooked Langford $hL_8^{17}$ are:
$8,\underline{m},n,j,k,f,d,b,h,o,l,9,i,g,e,c,a$ and $o,k,i,g,e,c,a,8,n,l,j,h,f,d,b,9$.

$6$ pairs: $\underline{c},m,f,2,4,j,\underline{n},i,6,7,\underline{o},h,l,k,e,8,a,g,\underline{9},d,b,5,\underline{3},\underline{1}$ and
$j,m,8,2, b,h,5,l,i,a,k,d,\linebreak f,6,g,e,7,4$.

For \underline{$n=36$} the small cases are $\{12,15,16,18,21,22\}$.

$14,15,19$ pairs: $A=A_{17}^{36},\;B=S_{5}+1\;\textup{pair},\;C=hL_{6}^{13}$; $A,sp(B),C$ and $\stackrel{\leftarrow}{A},sp(B),C$.

$12$ pairs: $hS_{11}$ with two $hL_{12}^{25}$ with one pair in common. The two $L_{11}^{25}$ are:
$s,0,c,r,A,q,u,i,\linebreak l,n,c,g,y,\underline{v},o,x,t,m,w,e,f,k,j,p,h,d$ and $t,0,g,k,q,u,y, z,o,i,s,w,m,A,l,x,e,r,p,c,f,h,\linebreak n,d,j$.

$16$ pairs: $\underline{A},\underline{y},\underline{z},\underline{v},\underline{t},\underline{x},\underline{q},\underline{w},\underline{n},\underline{l},
\underline{u},i,f,\underline{s},b,\underline{r},6,7,\underline{p},2,\underline{o},\underline{m},k,j,h,e,g,d,a,5,3,c,4,9,8,1$
and $h,e,a, 3,8,1,k,j,i,g,d,f,7,2,6,c,5,b,9,4$.

$18$ pairs: $\underline{A},\underline{y},\underline{z},\underline{v},\underline{t},\underline{x},\underline{q},\underline{w},\underline{n},\underline{l},
\underline{u},i,f,\underline{s},b,\underline{r},6,7,\underline{p},2,\underline{o},\underline{m},\underline{k},\underline{j},h,e,g,d,a,5,3,c,4,9,8,1$
and $g,h,a, 3,8,1,i,f,d,e,7,2,6,c,5,b,9,4$.

$21$ pairs: $\underline{A},\underline{y},\underline{z},\underline{v},\underline{t},\underline{x},\underline{q},\underline{w},\underline{n},\underline{l},
\underline{u},i,f,\underline{s},b,\underline{r},6,7,\underline{p},2,\underline{o},\underline{m},\underline{k},\underline{j},\underline{h},e,\underline{g},d,a, 5,3,c,4,9,8,1$
and $d,a, b,4,3,f,c,6,e,5,7,1,8,9,2$.

$22$ pairs: $\underline{A},\underline{y},\underline{z},\underline{v},\underline{t},\underline{x},\underline{q},\underline{w},\underline{n},\underline{l},
\underline{u},i,\underline{f},\underline{s},b,\underline{r},6,7,\underline{p},2,\underline{o},\underline{m},\underline{k},\underline{j},\underline{h},e,
\underline{g},d,a, 5,3,c,4,9,8,1$
and $a,9,4,3,d,c,8,e,5,6,2,b,7,1$.

For \underline{$n=48$} the small cases are $\{12,16,19,20,22-26,29\}$.

$16$ pairs: $S_{15}$ with two $L_{16}^{33}$ that have one pair in common. The two $L_{16}^{33}$ are:
$w,0,M,s,x,H,\linebreak E,t,L,u,n,C,D,r,I,G,J,A,h,p,K, v,B,\underline{F},l,i,g,z,j,y,q,k, m,o$ and $J,0,E,H,B,m,G,I,\linebreak K,y,u,M,i,q,l,t,D,h,j,L,C,z,s,x,A,n,p, v,g,k,w,r,o$.

$12$ pairs: $S_4,hL_5^{11},hL_{16}^{33}$ and $d(S_4),hL_5^{11},hL_{16}^{33}$ ($1$ pair in common from above).

$25,26,27,28,31$ pairs: $A=A_{23}^{48},\;B=hS_{6}+2\;\textup{pairs},\;C=L_{7}^{17}$; $A,sp(B),C$ and $A,sp(B),\stackrel{\leftarrow}{C}$ ($0$ pairs from Table \ref{n=4t}, sequence 3).

$19-24,27$ pairs: $A=A_{21}^{48},\;B=S_8+1\,\textup{pair}, C=6-\textup{ext} L_9^{18}$; $A,sp(B),C$ and $\stackrel{\leftarrow}{A},sp(B),C$.

$29$ pairs: $A=A_{27}^{48},\;B=S_4,\;C=L_5^{17}$; $A,B,C$ and $A,B$ (1 pair), $C$ (1 pair).

The two $L_5^{17}$ with one pair in common are: $7,l,d,i,5,j,b,\underline{f},9,k,h,c,g,a,e\linebreak 8,6$ and $i,c,e,j,h,5,d,8,l,g,k,a,b,9,6,7$.

For \underline{$n=60$} the small cases are $\{20,23,24,26-34,37\}$.

$20$ pairs: $hS_{19}$ with two $hL_{20}^{41}$ with one pair in common. The two $hL_{20}^{41}$ are:
$S,0,W,I,V,q,E,\linebreak U,N,Y,C,J,O,P,B,o,A,X,v,K,k,L,l,T,R,G,M,s,m,Q,n,r, \underline{p},w,z,D,H,x,F,u,t,y$ and $T,0,S,X,K,U,u,R,H,v,G,W,x,I,t,w,L,A,V,Y,q \linebreak l,N,F,k,Q,s,C,J,P,B,O,y,M,o,D,E,z,r,t,m$.

$23$ pairs: $A=A_{31}^{60},\;B=L_2^7,\;C=L_9^{21},\;D=(1,1)$; $A,B,C,D$ and $\stackrel{\leftarrow}{A},B$ (1 pair), $C,D$

The two $L_2^7$ with one pair in common are:
$8,3,5,6,3,7,4,5,8,6,4,2,7,2$ and $6,8,3,7,4,3,6,\linebreak 5,4,8,7,2,5,2$.

$33$ pairs: $A=A_{27}^{60},\,B=R_9(5,16),\,C=hL_{10}^{22}$; $A,B,C$ and $\stackrel{\leftarrow}{A},B,C$

$34$ pairs: $A=A_{31}^{60},\;B=L_2^7,\;C=L_9^{21},\;D=(1,1)$; $A,B,C,D$ and $A,\stackrel{\leftarrow}{B},C$ (2 pairs), $D$.

The two $L_9^{21}$ with 2 pairs in common are:
$e,s,f,m,q,a,k,p,t,i,\underline{n},o,b,r,c,l,\underline{j}\linebreak g,d,9,h$ and $r,e,s,h,o,c,q,t,l,9,k,m,b,p,a,g,d,i,f$.

$24-32,35$ pairs: $A=A_{25}^{60},\,B=hS_{11}+1\, \textup{pair},\,C=9-\textup{ext}\,L_{12}^{23}$; $A,B,C$ and $\stackrel{\leftarrow}{A},B,C$

$37,38,39$ pairs: $A=A_{25}^{60},\;B=10-\textup{ext}\;S_{11}+1\;\textup{pair},\;C=L_{12}^{23}$; $A,B,C$ and $A,B,\stackrel{\leftarrow}{C}$ ($0,1,2$ pairs from Table \ref{n=4t}, sequence 5).

The rest of the cases are listed in Appendix A of the Supplement.


\begin{thebibliography}{99}
\bibitem{baker} C.\ Baker, Extended Skolem sequences, {\em J. Combin. Des.} {\bf 3(5)} (1995) 363-379.
\bibitem{bilington} E.J.\ Billington, Cyclic balanced ternary designs with block size three and index two, {\em Ars Combin.} {\bf 23B} (1987) 215-232.
\bibitem{shalaby}C.\ Baker and N.\ Shalaby, Disjoint Skolem sequences and related disjoint structures, {\em J. Combin. Des.} {\bf 1(5)} (1993) 329-345.
\bibitem{bermond}J.-C.\ Bermond, A.E. Brouwer and A. Germa, Systemes de triplets et differences associees, {\em Colloq. CRNS, Problemes Combinatoires et Theorie des Graphes}, Orsay, 1976, 35-38.
\bibitem{crc} C.J.\ Colbourn and R.\ Mathon, Steiner systems in: The CRC Handbook of combinatorial designs, edited by C.J. Colbourn, J.H. Dinitz, CRC Press, 2006, 66-75.
\bibitem{colbournrosa} C.J.\ Colbourn and A.\ Rosa, Triple Systems, Clarendon Press, Oxford, 1999, 247-266.
\bibitem{colbournmathonrosa}C.J.\ Colbourn, R.A.\ Mathon, A.\ Rosa, N.\ Shalaby, The fine structure of three-fold triple systems: $v\equiv 1,3\;\mathrm{(mod\;6)}$, {\em Discrete Math.} {\bf 92} (1991) 49-64.
\bibitem{colbournmathonshalaby}C.J.\ Colbourn, R.A.\ Mathon and N.\ Shalaby, The fine structure of three-fold triple systems: $v\equiv5\;\mathrm{(mod\;6)}$, {\em Australas. J. Combin.} {\bf 3} (1991) 75-92.
\bibitem{foody} W.\ Foody and A.\ Hedayat, On theory and applications of BIB Designs with repeated blocks, {\em Ann. Statist.} {\bf 5(5)} (1977) 932-945.
\bibitem{linekmor}V.\ Linek and S.\ Mor, On partitions of $\{1,2...2m+1\}-\{k\}$ into differences $d,...d+m-1$: Extended Langford sequences of large defect, {\em J. Combin. Des.} {\bf 12} (2004) 421-442.
\bibitem{linekshalaby} V.\ Linek and N.\ Shalaby, The existence of $(p,q)$-extended Rosa sequences, {\em Discrete Math.} {\bf 308} (2008) 1583-1602.
\bibitem{lindnerrosa}C.C.\ Lindner and A.\ Rosa, Steiner triple systems having a prescribed number of triples in common, {\em Canad. J. Math.} {\bf 27} (1975) 1166-1175. Corrigendum: {\bf 30} (1978) 896.
\bibitem{wallis} C.C.\ Lindner and W.D.\ Wallis, Embeddings and prescribed intersections of transitive triple systems, {\em Ann. Discrete Math.} {\bf 15} (1982) 265-272.
\bibitem{milici} S.\ Milici and G.\ Quattrocchi, The Fine Structure of Threefold Directed Triple Systems, {\em Australas. J. Combin.} {\bf 16} (1997) 125-154.
\bibitem{okeefe}E.S.\ O'Keefe, Verification of a conjecture of Th. Skolem, {\em Math. Scand.} {\bf 9} (1961) 80-82.
\bibitem{hoffman} A.\ Rosa and D.G.\ Hoffman, The number of repeated blocks in twofold triple system, {\em J. Combin. Theory Ser. A} {\bf 41} (1986) 61-88.
\bibitem{shalaby1}N.\ Shalaby, The existence of near-Skolem and hooked near-Skolem sequences, {\em Discrete Math.} {\bf 135} (1994) 303-319.
\bibitem{shalaby4} A.\ Sharaf Eldin, N.\ Shalaby and F.\ Al-Thukair, Construction of Skolem sequences, {\em Int. J. Comput. Math.} {\bf 70} (1998) 333-345.
\bibitem{simpson}J.E.\ Simpson, Langford sequences, perfect and hooked, {\em Discrete Math.} {\bf 44} (1983) 97-104.
\bibitem{skolem}Th.\ Skolem, On certain distributions of integers in pairs with given differences, {\em Math. Scand.} {\bf 5} (1957) 57-68.
\end{thebibliography}
\end{document}


\begin{center}
{\bf \large Supplement for the paper} 

\vspace{0.25in}

{\bf \Large The Intersection Spectrum of Skolem Sequences and its Applications to \mbox {\boldmath ${\lambda}$}-Fold Cyclic Triple Systems}

\vspace{0.25in}
 Nabil Shalaby\hspace{2in} Daniela Silvesan \\
 Department of Mathematics and Statistics \\
 Memorial University of Newfoundland \\
 St. John's, Newfoundland \\
 CANADA A1C 5S7
\end{center}

\begin{abstract}
In \cite{silvesan}, the intersection spectrum of two Skolem sequences of order $n$ was found. Then the intersection spectrum was applied to get the the fine structure
 of $2$-fold, $3$-fold and $4$-fold cyclic triple systems of order $v$.

For a better understanding of \cite{silvesan}, we add a few more explanations and examples here. 

In \cite {silvesan}, we splitted the proof of the Main Theorem 5.1 into six different cases. We showed two cases in details there and we include here the remaining four cases of the proof.

For some small orders, the constructions we used in \cite{silvesan} did not work. We include in Appendix $A$ the small cases for $n\equiv 1,4,5,8,9(\textup{mod}\,12)$.
 
In Appendix $B$, we give examples that cover the intersection spectrum of two Skolem sequences and two hooked Skolem sequences of order $n$ for every $1\leq n\leq 9$.

In \cite{silvesan}, we used the constructions of Langford sequences found in \cite{bermond}. We give these constructions in Appendix $C$. 

We also include in Appendix $D$ two more tables that gives pairs in common between Langford sequences necessary to complete the proof of our results.
\end{abstract}

\section {Introduction}

A {\it $\lambda$-fold triple system} of order $v$, denoted by TS$(v,\lambda)$, is a pair $(V,\cal{B})$ where $V$ is a $v$-set of points and $\cal{B}$ is a set of
 $3$-subsets ({\it blocks}) such that any $2$-subset of $V$ appears in precisely $\lambda$ blocks. If $\lambda=1$, the system is called a
 {\it Steiner triple system}, denoted by STS$(\lambda)$ is a permutation on $V$ leaving $\cal{B}$ invariant.  A TS$(v,\lambda)$ is {\it cyclic} if
 its automorphism group contains a $v$-cycle. A cyclic TS$(v,\lambda)$ is denoted by CTS$(v,\lambda)$. If $\lambda=1$, the system is called a
 {\it cyclic Steiner triple system} and is denoted by CSTS$(v)$.

Several researchers investigated cyclic designs and many results have been found. For more details about the history of cyclic designs and for the definitions
 used the reader is directed to see \cite {silvesan}. 

\section {The remaining four cases of the proof of Theorem 5.1}

\begin{theorem}
The necessary conditions are sufficient for two Skolem sequences of
order $n$ to intersect in $\{0,1,\ldots,n-3,n\}$ pairs.
\end{theorem}

\begin{proof}

{\bf Case 3.  \mbox {\boldmath ${n\equiv 4\;\mathrm{(mod \;12)}}$}}

{\bf The number of pairs in the interval  \mbox {\boldmath ${[0,\lfloor\frac{n}{3}\rfloor]}$}:} To get $0,1,2,3,5$ pairs in common for $n=16$ take: $sp(S_5) L_6^{11}$ and $sp(S_5) \stackrel{\leftarrow}{L_6^{11}}$ ($0,1,2$ pairs from \cite{silvesan}, Table 1, sequence 5).

To get $0,1,\ldots,\lfloor\frac{n}{3}\rfloor$ for $n\geq 28$ take: $sp(S_{\lfloor\frac{n}{3}\rfloor}) L^{n-\lfloor\frac{n}{3}\rfloor}_{\lfloor\frac{n}{3}\rfloor+1}$ and  $sp(S_{\lfloor\frac{n}{3}\rfloor}) \stackrel{\leftarrow}{L^{n-\lfloor\frac{n}{3}\rfloor}_{\lfloor\frac{n}{3}\rfloor+1}}$ ($0,1,2$ pairs from \cite{silvesan}, Table 1, sequence 5).

{\bf The number of pairs in the interval  \mbox {\boldmath ${(\lfloor\frac{n}{3}\rfloor, 2\lfloor\frac{n}{3}\rfloor]}$}:} To get $\lfloor\frac{n}{3}\rfloor$, $\lfloor\frac{n}{3}\rfloor+1$, $\lfloor\frac{n}{3}\rfloor+2$ pairs in common for $n\geq 16$ take: $S_{\lfloor\frac{n}{3}\rfloor} L_{\lfloor\frac{n}{3}\rfloor+1}^{2\lfloor\frac{n}{3}\rfloor+1}$ and  $S_{\lfloor\frac{n}{3}\rfloor} \stackrel{\leftarrow}{L_{\lfloor\frac{n}{3}\rfloor+1}^{2\lfloor\frac{n}{3}\rfloor+1}}$ ($0,1,2$ pair from \cite{silvesan}, Table 1, sequence 5).

To get $2\lfloor\frac{n}{3}\rfloor-3$ pairs in common for $n\geq 64$ take: $A=A_{2\lfloor\frac{n}{3}\rfloor-15}^n;\;B=S_{12};\;C=L_{13}^{\lfloor\frac{n}{3}\rfloor+4}$; $A,B,C$ and $A,B,\stackrel{\leftarrow}{C}$ ($0$ pairs from \cite{silvesan}, Table 1, sequence 3).

To get $2\lfloor\frac{n}{3}\rfloor-2$ pairs in common for $n\geq 28$ take: $A=A_{2\lfloor\frac{n}{3}\rfloor-3}^n,B=L_2^3,C=L_5^{\lfloor\frac{n}{3}\rfloor},D=(1,1)$; $A,B,C,D$ and $A,\stackrel{\leftarrow}{B}$ ($0$ pairs from \cite{silvesan}, Table 1, sequence 5), $\stackrel{\leftarrow}{C}$ ($0$ pairs from \cite{silvesan}, Table 1, sequence 3), $D$.

To get $2\lfloor\frac{n}{3}\rfloor-1$ pairs in common for $n\geq 52$ take: $A=A_{2\lfloor\frac{n}{3}\rfloor-11}^n;\;B=S_{9};\;C=L_{10}^{\lfloor\frac{n}{3}\rfloor+3}$; $A,B,C$ and $A,B,\stackrel{\leftarrow}{C}$ ($1$ pair from \cite{silvesan}, Table 1, sequence 1).

To get $2\lfloor\frac{n}{3}\rfloor$ pairs in common for $n\geq 52$ take: $A=A_{2\lfloor\frac{n}{3}\rfloor-7}^n;\;B=2-near\;S_{7};\;C=hL_{8}^{\lfloor\frac{n}{3}\rfloor+1}*(2,0,2)$; $A,B,C$ and $A,B,\stackrel{\leftarrow}{C}$ ($1$ pair from Appendix E, Table \ref{hooked}, sequence 3).

\begin{table}[hp]
\begin{center}
\begin{tabular}{|c|c|c|}
\hline
$m$ & Sequence 1 & Sequence 2\\
\hline
$0,1,2,3,4\;\mathrm{(mod\;8)}$ & $\lfloor\frac{m}{8}\rfloor+1$ & $\lfloor\frac{m}{8}\rfloor$\\
\hline
$5,6,7\;\mathrm{(mod\;8)}$ & $\lfloor\frac{m}{8}\rfloor+1$ & $\lfloor\frac{m}{8}\rfloor+1$\\
\hline
\end{tabular}
\caption{The number of steps in sequences 1 and 2} \label{steps2}
\end{center}
\end{table}

Sequence 1 for $n\geq 52$, starts with $A_1=A_{t_{1}=2\lfloor\frac{n}{3}\rfloor-11}^n,\;B_{1}=S_{9},\;C_{1}=L_{d_1=10}^{m_1=\lfloor\frac{n}{3}\rfloor+3}$. The first left sequence of intervals starts with $[a_1=\lfloor\frac{n}{3}\rfloor+3,b_1=\lfloor\frac{n}{3}\rfloor+9]$ and the first right sequence of intervals starts with $[e_1=2\lfloor\frac{n}{3}\rfloor-10,f_1=2\lfloor\frac{n}{3}\rfloor-4]$.

Sequence 2 for $n\geq 64$, starts with $A_1'=A_{t_{1}=2\lfloor\frac{n}{3}\rfloor-15}^n,\;B_{1}'=S_{12},\;C_{1}'=L_{d_1=13}^{m_1=\lfloor\frac{n}{3}\rfloor+4}$. The second left sequence of intervals starts with $[c_1=\lfloor\frac{n}{3}\rfloor+4,d_1=\lfloor\frac{n}{3}\rfloor+13]$ and the second right sequence of intervals starts with $[g_1=2\lfloor\frac{n}{3}\rfloor-15,h_1=2\lfloor\frac{n}{3}\rfloor-6]$. Note that sequence 2 is not needed when $m\equiv 0,1,2,3,4\;\mathrm{(mod\;8)}$.

{\bf Case 4. \mbox {\boldmath ${n\equiv 5\;\mathrm{(mod \;12)}}$}}

{\bf The number of pairs in the interval  \mbox {\boldmath ${[0,\lfloor\frac{n}{3}\rfloor]}$}:} To get $1,2$ pairs in common for $n=17$ take: $sp(S_5) L_6^{12}$ and $sp(S_5) \stackrel{\leftarrow}{L_6^{12}}$ ($1$ pair from \cite{silvesan}, Table 1, sequence 4).

To get $3,4$ pairs in common for $n=17$ take: $sp(S_5) L_6^{12}$ and $sp(S_5) \stackrel{\leftarrow}{L_6^{12}}$ ($3$ pairs from \cite{silvesan}, Table 1, sequence 1).

To get $0,1,\ldots,\lfloor\frac{n}{3}\rfloor-3,\lfloor\frac{n}{3}\rfloor$ pairs in common for $n\geq 29$ take: $sp(S_{\lfloor\frac{n}{3}\rfloor}) L^{n-\lfloor\frac{n}{3}\rfloor}_{\lfloor\frac{n}{3}\rfloor+1}$ and  $sp(S_{\lfloor\frac{n}{3}\rfloor}) \stackrel{\leftarrow}{L^{n-\lfloor\frac{n}{3}\rfloor}_{\lfloor\frac{n}{3}\rfloor+1}}$ ($0$ pairs from \cite{silvesan}, Table 1, sequence 2).

To get $1,\ldots,\lfloor\frac{n}{3}\rfloor-2,\lfloor\frac{n}{3}\rfloor+1$ pairs in common for $\lfloor\frac{n}{3}\rfloor+1=4s+2$ and $s\equiv 0,2\;\mathrm{(mod\;3)}$ take: $sp(S_{\lfloor\frac{n}{3}\rfloor}) L^{n-\lfloor\frac{n}{3}\rfloor}_{\lfloor\frac{n}{3}\rfloor+1}$ and  $sp(S_{\lfloor\frac{n}{3}\rfloor}) \stackrel{\leftarrow}{L^{n-\lfloor\frac{n}{3}\rfloor}_{\lfloor\frac{n}{3}\rfloor+1}}$ ($1$ pair from \cite{silvesan}, Table 1, sequence 1).

To get $3,\ldots,\lfloor\frac{n}{3}\rfloor,\lfloor\frac{n}{3}\rfloor+3$ pairs in common for $\lfloor\frac{n}{3}\rfloor+1=4s+2$ and $s\equiv 1\;\mathrm{(mod\;3)}$ take: $sp(S_{\lfloor\frac{n}{3}\rfloor}) L^{n-\lfloor\frac{n}{3}\rfloor}_{\lfloor\frac{n}{3}\rfloor+1}$ and  $sp(S_{\lfloor\frac{n}{3}\rfloor}) \stackrel{\leftarrow}{L^{n-\lfloor\frac{n}{3}\rfloor}_{\lfloor\frac{n}{3}\rfloor+1}}$ ($3$ pairs from \cite{silvesan}, Table 1, sequence 1).

To get $\lfloor\frac{n}{3}\rfloor-1$ pairs in common for $n\geq 17$ take: $S_{\lfloor\frac{n}{3}\rfloor-1} L^{n-\lfloor\frac{n}{3}\rfloor+1}_{\lfloor\frac{n}{3}\rfloor}$ and  $S_{\lfloor\frac{n}{3}\rfloor-1} \stackrel{\leftarrow}{L^{n-\lfloor\frac{n}{3}\rfloor+1}_{\lfloor\frac{n}{3}\rfloor}}$ ($0$ pairs from \cite{silvesan}, Table 1, sequence 3).

{\bf The number of pairs in the interval  \mbox {\boldmath ${(\lfloor\frac{n}{3}\rfloor, 2\lfloor\frac{n}{3}\rfloor]}$}:} To get $\lfloor\frac{n}{3}\rfloor+1$ pairs in common for $n\geq 17$ and $\lfloor\frac{n}{3}\rfloor+1=4s+2, s\equiv 0,2\;\mathrm{(mod\;3)}$ take: $S_{\lfloor\frac{n}{3}\rfloor} L_{\lfloor\frac{n}{3}\rfloor+1}^{2\lfloor\frac{n}{3}\rfloor+2}$ and $S_{\lfloor\frac{n}{3}\rfloor} \stackrel{\leftarrow}{L_{\lfloor\frac{n}{3}\rfloor+1}^{2\lfloor\frac{n}{3}\rfloor+2}}$ ($1$ pair from \cite{silvesan}, Table 1, sequence 1).

To get $\lfloor\frac{n}{3}\rfloor+3$ pairs in common for $n\geq 17$ and $\lfloor\frac{n}{3}\rfloor+1=4s+2, s\equiv 1\;\mathrm{(mod\;3)}$ take: $S_{\lfloor\frac{n}{3}\rfloor} L_{\lfloor\frac{n}{3}\rfloor+1}^{2\lfloor\frac{n}{3}\rfloor+2}$ and $S_{\lfloor\frac{n}{3}\rfloor} \stackrel{\leftarrow}{L_{\lfloor\frac{n}{3}\rfloor+1}^{2\lfloor\frac{n}{3}\rfloor+2}}$ ($3$ pairs from \cite{silvesan}, Table 1, sequence 1).

To get $\lfloor\frac{n}{3}\rfloor$ pairs in common for $n\geq 17$ and $\lfloor\frac{n}{3}\rfloor+1\equiv 1,3,4,5\;\mathrm{(mod\;6)}$ take: $S_{\lfloor\frac{n}{3}\rfloor} L_{\lfloor\frac{n}{3}\rfloor+1}^{2\lfloor\frac{n}{3}\rfloor+2}$ and $S_{\lfloor\frac{n}{3}\rfloor} \stackrel{\leftarrow}{L_{\lfloor\frac{n}{3}\rfloor+1}^{2\lfloor\frac{n}{3}\rfloor+2}}$ ($0$ pairs from \cite{silvesan}, Table 1, sequence 4).

To get $\lfloor\frac{n}{3}\rfloor+1$ pairs in common for $n\geq 17$ and $\lfloor\frac{n}{3}\rfloor+1\equiv 0,2\;\mathrm{(mod\;6)}$ take: $S_{\lfloor\frac{n}{3}\rfloor} L_{\lfloor\frac{n}{3}\rfloor+1}^{2\lfloor\frac{n}{3}\rfloor+2}$ and $S_{\lfloor\frac{n}{3}\rfloor} \stackrel{\leftarrow}{L_{\lfloor\frac{n}{3}\rfloor+1}^{2\lfloor\frac{n}{3}\rfloor+2}}$ ($1$ pair from \cite{silvesan}, Table 1, sequence 4).

To get $\lfloor\frac{n}{3}\rfloor+2, \lfloor\frac{n}{3}\rfloor+3,\lfloor\frac{n}{3}\rfloor+7$ pairs in common for $n\geq 29$ take: $A=A_{2\lfloor\frac{n}{3}\rfloor-5}^n;\;B=S_{5};\;C=L_{6}^{\lfloor\frac{n}{3}\rfloor+2}$; $A,sp(B),C$ and $A,sp(B),\stackrel{\leftarrow}{C}$ ($0$ pairs from \cite{silvesan}, Table 1, sequence 3).

To get $2\lfloor\frac{n}{3}\rfloor-5,2\lfloor\frac{n}{3}\rfloor-4,2\lfloor\frac{n}{3}\rfloor$ pairs in common for $n\geq 29$ take: $A=A_{2\lfloor\frac{n}{3}\rfloor-5}^n;\;B=S_{5};\;C=L_{6}^{\lfloor\frac{n}{3}\rfloor+2}$; $A,sp(B),C$ and $\stackrel{\leftarrow}{A},sp(B),C$.

To get $2\lfloor\frac{n}{3}\rfloor-3$ pairs in common for $n\geq 77$ take: $A=A_{2\lfloor\frac{n}{3}\rfloor-5}^n,\;B=L_3^5,\;C=hL_8^{\lfloor\frac{n}{3}\rfloor}*(2,0,2),D=(1,1)$; $A,D,B,C$ and $A,D,\stackrel{\leftarrow}{B}$ ($1$ pairs from \cite{silvesan}, Table 1, sequence 5),$\stackrel{\leftarrow}{C}$ ($0$ pairs from Appendix E, Table \ref{hooked}, sequence 1).

To get $2\lfloor\frac{n}{3}\rfloor-2$ pairs in common for $n\geq 77$ take: $A=A_{2\lfloor\frac{n}{3}\rfloor-13}^n,\;B=L_2^{11},\;C=L_{13}^{\lfloor\frac{n}{3}\rfloor+3}(1,1)$; $A,B,C$ and $A,B,\stackrel{\leftarrow}{C}$ ($0$ pairs from Appendix E, Table \ref{append}, sequence 2).

To get $2\lfloor\frac{n}{3}\rfloor-1$ pairs in common for $n\geq 53$ take: $A=A_{2\lfloor\frac{n}{3}\rfloor-9}^n;\;B=S_{8};\;C=L_{9}^{\lfloor\frac{n}{3}\rfloor+3}$; $A,B,C$ and $A,B,\stackrel{\leftarrow}{C}$ ($0$ pairs from \cite{silvesan}, Table 1, sequence 1).

Sequence 1 for $n\geq 53$, starts with $A_1=A_{t_{1}=2\lfloor\frac{n}{3}\rfloor-9}^n,\;B_{1}=S_{8},\;C_{1}=L_{d_1=9}^{m_1=\lfloor\frac{n}{3}\rfloor+3}$. The first left sequence of intervals starts with $[a_1=\lfloor\frac{n}{3}\rfloor+3,b_1=\lfloor\frac{n}{3}\rfloor+8]$ and the first right sequence of intervals starts with $[e_1=2\lfloor\frac{n}{3}\rfloor-9,f_1=2\lfloor\frac{n}{3}\rfloor-4]$.

Sequence 2 for $n\geq 89$, starts with $A_1'=A_{t_{1}=2\lfloor\frac{n}{3}\rfloor-21}^n,\;B_{1}'=S_{17},\;C_{1}'=L_{d_1=18}^{m_1=\lfloor\frac{n}{3}\rfloor+6}$. The second left sequence of intervals starts with $[c_1=\lfloor\frac{n}{3}\rfloor+6,d_1=\lfloor\frac{n}{3}\rfloor+20]$ and the second right sequence of intervals starts with $[g_1=2\lfloor\frac{n}{3}\rfloor-21,h_1=2\lfloor\frac{n}{3}\rfloor-7]$. Note that sequence 2 is not needed when $m\equiv 3,4\;\mathrm{(mod\;8)}$.


\begin{table}[hp]
\begin{center}
\begin{tabular}{|c|c|c|}
\hline
$m$ & Sequence 1 & Sequence 2\\
\hline
$0,1,2\;\mathrm{(mod\;8)}$ & $\lfloor\frac{m}{8}\rfloor$ & $\lfloor\frac{m}{8}\rfloor$\\
\hline
$3,4\;\mathrm{(mod\;8)}$ & $\lfloor\frac{m}{8}\rfloor+1$ & $\lfloor\frac{m}{8}\rfloor$\\
\hline
$5,6,7\;\mathrm{(mod\;8)}$ & $\lfloor\frac{m}{8}\rfloor+1$ & $\lfloor\frac{m}{8}\rfloor+1$\\
\hline
\end{tabular}
\caption{The number of steps in sequences 1 and 2} \label{steps3}
\end{center}
\end{table}

{\bf Case 5.  \mbox {\boldmath ${n\equiv 8\;\mathrm{(mod\; 12)}}$}}

{\bf The number of pairs in the interval  \mbox {\boldmath ${[0,\lfloor\frac{n}{3}\rfloor]}$}:} To get $0,1,5$ pairs in common for $n=20$ take: $sp(S_5) L_6^{15}$ and $sp(S_5) \stackrel{\leftarrow}{L_6^{15}}$ ($0$ pairs from \cite{silvesan}, Table 1, sequence 3).

To get $0,1,\ldots,\lfloor\frac{n}{3}\rfloor-4,\lfloor\frac{n}{3}\rfloor-1$ pairs in common for $n\geq 32$ take: $sp(S_{\lfloor\frac{n}{3}\rfloor-1}) L^{n-\lfloor\frac{n}{3}\rfloor+1}_{\lfloor\frac{n}{3}\rfloor}$ and  $sp(S_{\lfloor\frac{n}{3}\rfloor-1}) \stackrel{\leftarrow}{L^{n-\lfloor\frac{n}{3}\rfloor-1}_{\lfloor\frac{n}{3}\rfloor}}$ ($0$ pairs from \cite{silvesan}, Table 1, sequence 3).

To get $\lfloor\frac{n}{3}\rfloor-3$ pairs in common for $n=12m+8$ with $m\equiv 1\;\mathrm{(mod\;3)}$ and $m\geq 4$ take: $L_2^{\frac{n-8}{12}},L_{\frac{n+16}{12}}^{\lfloor\frac{n}{3}\rfloor-\frac{n+28}{12}},(1,1)L_{\lfloor\frac{n}{3}\rfloor-1}^{n-\lfloor\frac{n}{3}\rfloor+2}$ and $d(L_2^{\frac{n-8}{12}})$,$L_{\frac{n+16}{12}}^{\lfloor\frac{n}{3}\rfloor-\frac{n+28}{12}},\stackrel{\leftarrow}{(1,1)L_{\lfloor\frac{n}{3}\rfloor-1}^{n-\lfloor\frac{n}{3}\rfloor+2}}$ ($s$ pairs from Appendix E, Table \ref{append}, sequence 2, where $s=\frac{1}{4}\lfloor\frac{n}{3}\rfloor-\frac{1}{2}$).

To get $\lfloor\frac{n}{3}\rfloor-3$ pairs in common for $n=12m+8$ with $m\equiv 0,2\;\mathrm{(mod\;3)}$ take:
$L^{\lfloor\frac{n}{3}\rfloor-3}_2,L^{n-\lfloor\frac{n}{3}\rfloor+2}_{\lfloor\frac{n}{3}\rfloor-1}(1,1)$ and $L^{\lfloor\frac{n}{3}\rfloor-3}_2,\stackrel{\leftarrow}{L^{n-\lfloor\frac{n}{3}\rfloor+2}_{\lfloor\frac{n}{3}\rfloor-1}(1,1)}$ ($0$ pairs from Appendix E, Table \ref{append}, sequence 2).

To get $\lfloor\frac{n}{3}\rfloor-2$ pairs in common for $n\geq 20$ take: $L_2^{\lfloor\frac{n}{3}\rfloor-2},L^{2\lfloor\frac{n}{3}\rfloor+3}_{\lfloor\frac{n}{3}\rfloor}(1,1)$ and  $L_2^{\lfloor\frac{n}{3}\rfloor-2},\stackrel{\leftarrow}{L^{2\lfloor\frac{n}{3}\rfloor+3}_{\lfloor\frac{n}{3}\rfloor}(1,1)}$ ($0$ pairs from Appendix E, Table \ref{append}, sequence 3).

To get $\lfloor\frac{n}{3}\rfloor$ pairs in common for $n\geq 32$ take: $A=A_{2\lfloor\frac{n}{3}\rfloor-1}^n,\;B=L_2^3,\;C=L_5^{\lfloor\frac{n}{3}\rfloor-1},\;D=(1,1)$; $A,B,C,D$ and $\stackrel{\leftarrow}{A},\stackrel{\leftarrow}{B}$ ($0$ pairs from \cite{silvesan}, Table 1, sequence 5), $C,D$.

{\bf The number of pairs in the interval  \mbox {\boldmath ${(\lfloor\frac{n}{3}\rfloor, 2\lfloor\frac{n}{3}\rfloor]}$}:} To get $\lfloor\frac{n}{3}\rfloor+2,\lfloor\frac{n}{3}\rfloor+3,\lfloor\frac{n}{3}\rfloor+7$ pairs in common
for $n\geq32$ take: $A=A_{2\lfloor\frac{n}{3}\rfloor-5}^n;\;B=S_{5};\;C=L_{6}^{\lfloor\frac{n}{3}\rfloor+2}$; $A,sp(B),C$ and $\stackrel{\leftarrow}{A},sp(B),C$.

To get $2\lfloor\frac{n}{3}\rfloor-3$ pairs in common for $n\geq 92$ take: $A=A_{2\lfloor\frac{n}{3}\rfloor-17}^n,\;B=2-\textup{near}\; S_{15},\;C=hL_{16}^{\lfloor\frac{n}{3}\rfloor+2}*(2,0,2)$; $A,B,C$ and $A,B,\stackrel{\leftarrow}{C}$ ($0$ pairs from Appendix E, Table \ref{hooked}, sequence 3).

To get $2\lfloor\frac{n}{3}\rfloor-2,2\lfloor\frac{n}{3}\rfloor-1$ pairs in common for $n\geq32$ take: $A=A_{2\lfloor\frac{n}{3}\rfloor-5}^n;\;B=S_{5};\;C=L_{6}^{\lfloor\frac{n}{3}\rfloor+2}$; $A,sp(B),C$ and $A,sp(B),\stackrel{\leftarrow}{C}$ ($3$ pairs from \cite{silvesan}, Table 1, sequence 1).

To get $2\lfloor\frac{n}{3}\rfloor$ pairs in common for $n\geq44$ take: $A=A_{2\lfloor\frac{n}{3}\rfloor-5}^n;\;B=2-near\;S_{6};\;C=hL_{7}^{\lfloor\frac{n}{3}\rfloor+1}*(2,0,2)$; $A,B,C$ and $A,B,\stackrel{\leftarrow}{C}$ ($0$ pairs from Appendix E, Table \ref{hooked}, sequence 1).

To get $\lfloor\frac{n}{3}\rfloor+1$ pairs in common for $n\geq 80$ take: $A=A_{2\lfloor\frac{n}{3}\rfloor-9}^n,\;B=L_4^8,\;C=hL_{12}^{\lfloor\frac{n}{3}\rfloor},\;D=S_3$; $A,B,C,D$ and $\stackrel{\leftarrow}{A},\stackrel{\leftarrow}{B}$ ($1$ pair from \cite{silvesan}, Table 1, sequence 1), $C, d(D)$.

\begin{table}[hp]
\begin{center}
\begin{tabular}{|c|c|c|}
\hline
$m$ & Sequence 1 & Sequence 2\\
\hline
$0,1\;\mathrm{(mod\;8)}$ & $\lfloor\frac{m}{8}\rfloor$ & $\lfloor\frac{m}{8}\rfloor$\\
\hline
$2,3,4\;\mathrm{(mod\;8)}$ & $\lfloor\frac{m}{8}\rfloor+1$ & $\lfloor\frac{m}{8}\rfloor$\\
\hline
$5,6,7\;\mathrm{(mod\;8)}$ & $\lfloor\frac{m}{8}\rfloor+1$ & $\lfloor\frac{m}{8}\rfloor+1$\\
\hline
\end{tabular}
\caption{The number of steps in sequences 1 and 2} \label{steps4}
\end{center}
\end{table}

Sequence 1 for $n\geq 44$, starts with $A_1=A_{t_{1}=2\lfloor\frac{n}{3}\rfloor-9}^n,\;B_{1}=S_{8},\;C_{1}=L_{d_1=9}^{m_1=\lfloor\frac{n}{3}\rfloor+3}$. The first left sequence of intervals starts with $[a_1=\lfloor\frac{n}{3}\rfloor+3,b_1=\lfloor\frac{n}{3}\rfloor+8]$ and the first right sequence of intervals starts with $[e_1=2\lfloor\frac{n}{3}\rfloor-9,f_1=2\lfloor\frac{n}{3}\rfloor-4]$.

Sequence 2 for $n\geq 92$, starts with $A_1'=A_{t_{1}=2\lfloor\frac{n}{3}\rfloor-21}^n,\;B_{1}'=S_{17},\;C_{1}'=L_{d_1=18}^{m_1=\lfloor\frac{n}{3}\rfloor+6}$. The second left sequence of intervals starts with $[c_1=\lfloor\frac{n}{3}\rfloor+6,d_1=\lfloor\frac{n}{3}\rfloor+20]$ and the second right sequence of intervals starts with $[g_1=2\lfloor\frac{n}{3}\rfloor-18,h_1=2\lfloor\frac{n}{3}\rfloor-4]$.

{\bf Case 6. \mbox {\boldmath ${n\equiv 9\;\mathrm{(mod\; 12)}}$}}

{\bf The number of pairs in the interval  \mbox {\boldmath ${[0,\lfloor\frac{n}{3}\rfloor]}$}:} To get $0,\ldots,\lfloor\frac{n}{3}\rfloor-6,\lfloor\frac{n}{3}\rfloor-3$
 pairs in common for $n\geq 21$ take: $sp(S_{\lfloor\frac{n}{3}\rfloor-3})\linebreak L^{n-\lfloor\frac{n}{3}\rfloor+3}_{\lfloor\frac{n}{3}\rfloor-2}$ and  $sp(S_{\lfloor\frac{n}{3}\rfloor-3}) \stackrel{\leftarrow}{L^{n-\lfloor\frac{n}{3}\rfloor+3}_{\lfloor\frac{n}{3}\rfloor-2}}$ ($0$ pairs from \cite{silvesan}, Table 1, sequence 3).

To get $3,4,8$ pairs in common for $n=21$ take: $sp(S_5) L_6^{16}$ and $sp(S_5) \stackrel{\leftarrow}{L_6^{16}}$ ($3$ pairs from \cite{silvesan}, Table 1, sequence 1).

To get $3,4,\ldots,\lfloor\frac{n}{3}\rfloor-2,\lfloor\frac{n}{3}\rfloor+1$ pairs in common for $\lfloor\frac{n}{3}\rfloor-1=4s+2$ and $s\equiv 1\;\mathrm{(mod\;3)},s>1$ take: $sp(S_{\lfloor\frac{n}{3}\rfloor-2}) L^{n-\lfloor\frac{n}{3}\rfloor+2}_{\lfloor\frac{n}{3}\rfloor-1}$ and  $sp(S_{\lfloor\frac{n}{3}\rfloor-2}) \stackrel{\leftarrow}{L^{n-\lfloor\frac{n}{3}\rfloor+2}_{\lfloor\frac{n}{3}\rfloor-1}}$ ($3$ pairs from \cite{silvesan}, Table 1, sequence 1).

To get $1,2,\ldots,\lfloor\frac{n}{3}\rfloor-4,\lfloor\frac{n}{3}\rfloor-1$ pairs in common for $\lfloor\frac{n}{3}\rfloor-1=4s+2$ and $s\equiv 0,2\;\mathrm{(mod\;3)}$ take: $sp(S_{\lfloor\frac{n}{3}\rfloor-2}) L^{n-\lfloor\frac{n}{3}\rfloor+2}_{\lfloor\frac{n}{3}\rfloor-1}$ and  $sp(S_{\lfloor\frac{n}{3}\rfloor-2}) \stackrel{\leftarrow}{L^{n-\lfloor\frac{n}{3}\rfloor+2}_{\lfloor\frac{n}{3}\rfloor-1}}$ ($1$ pair from \cite{silvesan}, Table 1, sequence 1).

To get $\lfloor\frac{n}{3}\rfloor-2$ pairs in common for $n\geq21$ take: $S_{\lfloor\frac{n}{3}\rfloor-2} L_{\lfloor\frac{n}{3}\rfloor-1}^{n-\lfloor\frac{n}{3}\rfloor+2}$ and $S_{\lfloor\frac{n}{3}\rfloor-2} \stackrel{\leftarrow}{L_{\lfloor\frac{n}{3}\rfloor-1}^{n-\lfloor\frac{n}{3}\rfloor+2}}$ ($0$ pairs from \cite{silvesan}, Table 1, sequence 2).

To get $\lfloor\frac{n}{3}\rfloor-1$ pairs in common for $n\geq 21$ take: $2-\textup{near}\; S_{\lfloor\frac{n}{3}\rfloor-1} hL^{n-\lfloor\frac{n}{3}\rfloor+1}_{\lfloor\frac{n}{3}\rfloor}*(2,0,2)$ and $2-\textup{near}\; S_{\lfloor\frac{n}{3}\rfloor-1} \stackrel{\leftarrow}{hL^{n-\lfloor\frac{n}{3}\rfloor+1}_{\lfloor\frac{n}{3}\rfloor}*(2,0,2)}$ ($1$ pair from Appendix E, Table \ref{hooked}, sequence 1).

To get $\lfloor\frac{n}{3}\rfloor$ pairs in common for $n\geq 69$ take: $A=A_{2\lfloor\frac{n}{3}\rfloor-9}^n,\,B=L_4^7,\,C=hL_{11}^{\lfloor\frac{n}{3}\rfloor-1},\,D=S_3$; $A,B,C,D$ and $\stackrel{\leftarrow}{A},\stackrel{\leftarrow}{B}$ ($1$ pair from \cite{silvesan}, Table 1, sequence 5), $C, d(D)$.

{\bf The number of pairs in the interval  \mbox {\boldmath ${(\lfloor\frac{n}{3}\rfloor, 2\lfloor\frac{n}{3}\rfloor]}$}:} To get $\lfloor\frac{n}{3}\rfloor+1,\lfloor\frac{n}{3}\rfloor+2,\lfloor\frac{n}{3}\rfloor+5$ pairs in common for $n\geq33$ take: $A=A_{2\lfloor\frac{n}{3}\rfloor-5}^n;\;B=S_{4},\;C=L_{5}^{\lfloor\frac{n}{3}\rfloor+1}$; $A,sp(B),C$ and $\stackrel{\leftarrow}{A},sp(B),C$.

To get $\lfloor\frac{n}{3}\rfloor +3$ pairs in common for $n\geq 69$ take: $A=A_{2\lfloor\frac{n}{3}\rfloor-9}^n,\;B=L_4^7,\;C=hL_{11}^{\lfloor\frac{n}{3}\rfloor-1},\;D=S_3$; $A,B,C,D$ and $\stackrel{\leftarrow}{A},\stackrel{\leftarrow}{B}$ ($1$ pair from \cite{silvesan}, Table 1, sequence 5), $C,D$.

To get $2\lfloor\frac{n}{3}\rfloor -6$ pairs in common for $n\geq 69$ take: $A=A_{2\lfloor\frac{n}{3}\rfloor-9}^n,\;B=L_4^7,\;C=hL_{11}^{\lfloor\frac{n}{3}\rfloor-1}*(2,0,2),\;D=(3,1,1,3)$;
$A,B,C,D$ and $A,\stackrel{\leftarrow}{B}$ ($0$ pairs from \cite{silvesan}, Table 1, sequence 5), $\stackrel{\leftarrow}{C}$ ($1$ pair from Appendix E, Table \ref{hooked}, sequence 3),  $D$.

To get $2\lfloor\frac{n}{3}\rfloor-5, 2\lfloor\frac{n}{3}\rfloor-4,2\lfloor\frac{n}{3}\rfloor-1$ pairs in common $n\geq33$ take: $A=A_{2\lfloor\frac{n}{3}\rfloor-5}^n;\;B=S_{4},\;C=L_{5}^{\lfloor\frac{n}{3}\rfloor+1}$; $A,sp(B),C$ and $A,sp(B),\stackrel{\leftarrow}{C}$ ($0$ pairs from \cite{silvesan}, Table 1, sequence 1).

To get $2\lfloor\frac{n}{3}\rfloor-3$ pairs in common for $n\geq93$ take: $A=A_{2\lfloor\frac{n}{3}\rfloor-13}^n,\;B=2-near\;S_{11},\,C=hL_{12}^{\lfloor\frac{n}{3}\rfloor+2}*(2,0,2)$; $A,B,C$ and $A,B,\stackrel{\leftarrow}{C}$ ($0$ pairs from Appendix E, Table \ref{hooked}, sequence 1).

To get $2\lfloor\frac{n}{3}\rfloor-2$ pairs in common for $n=69$  or $n=81$ take: $A=A_{2\lfloor\frac{n}{3}\rfloor-13}^n,\;B=2-near\;S_{11},\,C=hL_{12}^{\lfloor\frac{n}{3}\rfloor+2}*(2,0,2)$; $A,B,C$ and $A,B,\stackrel{\leftarrow}{C}$ ($1$ pair from Appendix E, Table \ref{hooked}, sequence 1).

To get $2\lfloor\frac{n}{3}\rfloor -2$ pairs in common for $n\geq 93$ take: $A=A_{2\lfloor\frac{n}{3}\rfloor-9}^n,\;B=L_2^7,\;C=L_9^{\lfloor\frac{n}{3}\rfloor+1}(1,1)$; $A,B,C$ and $A,B,\stackrel{\leftarrow}{C}$ ($0$ pairs from Appendix E, Table \ref{append}, sequence 2).

To get $2\lfloor\frac{n}{3}\rfloor$ pairs in common for $n\geq21$ take: $A=A_{2\lfloor\frac{n}{3}\rfloor-1}^n,\;B=S_{1},\;C=L_{2}^{\lfloor\frac{n}{3}\rfloor}$; $A,B,C$ and $A,B,d(C)$.
\begin{table}[hp]
\begin{center}
\begin{tabular}{|c|c|c|}
\hline
$m$ & Sequence 1 & Sequence 2\\
\hline
$0,1\;\mathrm{(mod\;8)}$ & $\lfloor\frac{m}{8}\rfloor$ & $\lfloor\frac{m}{8}\rfloor$\\
\hline
$2,3,4,5,6,7\;\mathrm{(mod\;8)}$ & $\lfloor\frac{m}{8}\rfloor+1$ & $\lfloor\frac{m}{8}\rfloor$\\
\hline
\end{tabular}
\caption{The number of steps in sequences 1 and 2} \label{steps5}
\end{center}
\end{table}

Sequence 1 for $n\geq 69$, starts with $A_1=A_{t_{1}=2\lfloor\frac{n}{3}\rfloor-17}^n,\;B_{1}=S_{13},\;C_{1}=L_{d_1=14}^{m_1=\lfloor\frac{n}{3}\rfloor+4}$. The first left sequence of intervals starts with $[a_1=\lfloor\frac{n}{3}\rfloor+4,b_1=\lfloor\frac{n}{3}\rfloor+14]$. The first right sequence of intervals starts with $[e_1=2\lfloor\frac{n}{3}\rfloor-17,f_1=2\lfloor\frac{n}{3}\rfloor-7]$.

Sequence 2 for $n\geq 105$, starts with $A_1'=A_{t_{1}=2\lfloor\frac{n}{3}\rfloor-21}^n,\;B_{1}'=S_{16},\;C_{1}'=L_{d_1=17}^{m_1=\lfloor\frac{n}{3}\rfloor+5}$.
 The second left sequence of intervals starts with $[c_1=\lfloor\frac{n}{3}\rfloor+5,d_1=\lfloor\frac{n}{3}\rfloor+18]$ and the second right sequence of intervals
 starts with $[g_1=2\lfloor\frac{n}{3}\rfloor-21,h_1=2\lfloor\frac{n}{3}\rfloor-8]$. Note that sequence 2 is not needed when $m\equiv 2,3,4,5,6,7\;\mathrm{(mod\;8)}$.\end{proof}

\begin{center}
{\bf Appendix A: The remaining small cases}
\end{center}

\mbox {\boldmath ${n\equiv 1\;\mathrm{(mod\;12)}}$:}

For \underline{$n=13$} the small cases are $\{6,8\}$.

$6$ pairs: $\underline{a},5,\underline{6},4,d,\underline{8},\underline{7},c,9,b,3,\underline{1},\underline{2}$ and $3,c,5,b,d,9,4$.

$8$ pairs: $5,\underline{d},6,\underline{4},\underline{9},a,\underline{c},8,\underline{b},3,\underline{7},\underline{1},\underline{2}$
and $6,a,5,8,3$.

For \underline{$n=25$} the small cases are $\{11,12,13,14,15,16\}$.

$14$ pairs: $A=A_{11}^{25},\;B=3-\textup{ext}\;S_{4}+1\;\textup{pair},\;C=L_{5}^{9}$; $A,B,C$ and $\stackrel{\leftarrow}{A},B,C$;

$16,17,19$ pairs: $A=A_{11}^{25},\;B=3-\textup{ext}\;S_{4}+1\;\textup{pair},\;C=L_{5}^{9}$; $A,B,C$ and $A,B,\stackrel{\leftarrow}{C}$ ($0,1,3$ pairs from \cite{silvesan}, Table 1, sequence 5).

$11$ pairs: $c,a,\underline{l},\underline{j},\underline{h},\underline{f},\underline{d},\underline{b},\underline{9},\underline{7},\underline{5},p,n,o,m,k,\underline{i},g,\underline{e},6,3,1,4,2,8$\\
and $p,n,3,o,m,6,k,g,2,a,c,4,8,1$.

$12$ pairs: $\underline{p},\underline{n},\underline{o},\underline{k},\underline{i},\underline{m},\underline{f},\underline{l},b,\underline{8},e,\underline{j},3,4,\underline{h},\underline{g},d,c,a,7,2,9,1,6,5$\\
and $a,5,2,e,c,d,b,9,1,3,4,7,6$.

$13$ pairs: $\underline{p},\underline{n},\underline{o},\underline{k},\underline{i},\underline{m},\underline{f},\underline{l},b,8,\underline{e},\underline{j},3,4,\underline{h},\underline{g},\underline{d},c,a,7,2,9,1,6,5$\\
and $9,a,5,2,b,c,8,6,4,1,7,3$.

$15$ pairs: $\underline{p},\underline{n},\underline{o},\underline{k},\underline{i},\underline{m},\underline{f},\underline{l},\underline{b},8,\underline{e},\underline{j},3,4,\underline{h},\underline{g},\underline{d},\underline{c},a,7,2,9,1,6,5$\\
and $7,6,2,9,3,a,5,8,1,4$.

For \underline{$n=37$} the small cases are $\{15,16,17,18,19,20,22,23\}$.

$15,16,20$ pairs: $A=A_{17}^{37},\;B=S_{5}+2\;\textup{pairs},\;C=L_{6}^{13}$; $A,sp(B),C$ and $\stackrel{\leftarrow}{A},sp(B),C$.

$17$ pairs: $S_{12}$ with two $L_{13}^{25}$ with five pairs in common. The two $L_{13}^{25}$ are:
$v,z,o,p,n,s,y,\linebreak m,u,\underline{l},\underline{B},x,d,\underline{A},i,r,w,t,h,q,e,\underline{g},\underline{j},k,f$ and $s,p,r,m,y,z,x, k,n,v,w,t,u,h,i,e,f,q,d,o$.

$18$ pairs: $S_{12}$ with two $L_{13}^{25}$ with six pairs in common. The two $L_{13}^{25}$ are:
$\underline{v},z,o,\underline{p},n,s,y,\linebreak m,u,l,B,x,d,\underline{A},\underline{i},r,\underline{w},t,h,\underline{q},e,g,j,k,f$ and $B,n,t,u,x,k,s, z,r,f,y,e,d,g,m,j,l,o,h$.

$19$ pairs: $\underline{x},\underline{z},\underline{A},\underline{w},\underline{u},\underline{y},\underline{r},\underline{x},\underline{o},\underline{m},\underline{v},i,f,\underline{t},b,\underline{s},6,7,\underline{q},2,\underline{p},\underline{n},\underline{l},\underline{k},\underline{j},h,e,g,d$\\
and $g,h,a,3,8,1,i,f,d,e,7,2,6,c,5,b,4$.

$22$ pairs $\underline{x},\underline{z},\underline{A},\underline{w},\underline{u},\underline{y},\underline{r},\underline{x},\underline{o},\underline{m},\underline{v},\underline{i},f,\underline{t},b,\underline{s},6,7,\underline{q},2,\underline{p},\underline{n},\underline{l},\underline{k},\underline{j},\underline{h},e,\underline{g},d$\\
$a,5,3,c,4,9,8,1$ and $d,a,b,4,3,f,c,6,e,5,7,1,8,9,2$.

$23$ pairs: $\underline{x},\underline{z},\underline{A},\underline{w},\underline{u},\underline{y},\underline{r},\underline{x},\underline{o},\underline{m},\underline{v},\underline{i},\underline{f},\underline{t},b,\underline{s},6,7,\underline{q},2,\underline{p},\underline{n},\underline{l},\underline{k},\underline{j},\underline{h},e,\underline{g},d$ $a,5,3,c,4,9,8,1$ and $a,9,4,3,d,c,8,e,5,6,2,b,7,1$.

For \underline{$n=49$} the small cases are $\{28,29,31,32\}$.

$21-32$ pairs: $A=A_{21}^{49},\;B=S_9,\;C=L_{10}^{19}$; $A,B,C$ and $A,B,\stackrel{\leftarrow}{C}$ ($0,1,2$ pairs from \cite{silvesan}, Table 1, sequence 5).

For \underline{$n=61$} the small cases are $\{36,39\}$.

$36$ pairs: $A=A_{33}^{61},\;B=L_3^5(1,1),\;C=hL_8^{21}*(2,0,2)$; $A,B,C$ and $A,B$ (1 pair from \cite{silvesan}, Table 1, sequence 5) $C$ (0 pairs).

The two $hL_8^{21}$ with no pairs in common are:
$o,0,p,9,q,a,n,b,s,j,m,8,l,r \linebreak h,e,k,i,f,c,d,g$ and $r,0,a,n,9,q,m,c,p,k,8,s,l,j,e,p,i,g,b,d,h,f$.

$39$ pairs: $A,B,C$ and $A,B$ (1 pair from \cite{silvesan}, Table 1, sequence 5) $C$ (3 pairs). The two $hL_8^{21}$ with three pairs in common are:
$o,0,p,\underline{9},q,a,\underline{n},b,\underline{s},j,m,8,l,r,h\linebreak e,k,i,f,c,d,g$ and $r,0,c,o,q,p,8,k,b,a,j,l,m,g,d,h,i,e,f$.

For \underline{$n=85$} the small case is $\{42\}$.

$42$ pairs: $A=A_{41}^{61},\;B=L_2^{12},\;C=L_{14}^{31},\;D=(1,1)$; $A,B,C,D$ and $A,d(B), \stackrel{\leftarrow}{C}$ ($0$ pairs from Appendix E, Table \ref{hooked}, sequence 1) $ D$;

For \underline{$n=97$} the small cases are $\{46-50\}$.

$46$ pairs: $A=A_{45}^{97},\;B=L_2^{15},\;C=L_{17}^{36},\;D=(1,1)$; $A,B,C,D$ and $A,d(B),\stackrel{\leftarrow}{C}$ ($0$ pairs from \cite{silvesan}, Table 1, sequence 1), $D$.

$47$ pairs: $A=A_{45}^{97},\;B=L_2^4L_6^{11},\;C=L_{17}^{36},\;D=(1,1)$; $A,B,C,D$ and $A,d(L_2^4),\stackrel{\leftarrow}{L_6^{11}}$ ($1$ pair from \cite{silvesan}, Table 1, sequence 5) ,$\stackrel{\leftarrow}{C}$ ($0$ pairs from \cite{silvesan}, Table 1, sequence 1) $D$.

$48$ pairs: $A=A_{49}^{97},\;B=S_{12},\;C=L_{13}^{36}$; $A,B,C$ and $\stackrel{\leftarrow}{A},B,C$.

$49$ pairs: $A=A_{45}^{97},\;B=L_2^{15},\;C=(1,1)L_{17}^{36}$; $A,B,C$ and $A,d(B),\stackrel{\leftarrow}{C} $ ($4$ pairs from Appendix E, Table \ref{append}, sequence 2).

$50$ pairs: $A=A_{49}^{97},\;B=L_2^{12},\;C=L_{14}^{35},\;D=(1,1)$; $A,B,C,D$ and $A,d(B),\stackrel{\leftarrow}{C}$ ($0$ pairs from \cite{silvesan}, Table 1, sequence 3), $D$.

\mbox {\boldmath ${n\equiv 4 \;\mathrm{(mod\;12)}}$:}

For \underline{$n=16$} the small cases are $\{4,8,9,10\}$.

$4$ pairs: $S_4,L_5^{12}$ and $S_4,\stackrel{\leftarrow}{L_5^{12}}$ ($0$ pairs from \cite{silvesan}, Table 1, sequence 1).

$8$ pairs: $\underline{g},\underline{e},\underline{f},\underline{b},7,\underline{d},3,\underline{c},5,\underline{a},\underline{9},8,6,1,2,4$
and $4,7,2,6,8,5,1,3$.

$9$ pairs: $A=A_7^{16};\;B=hS_{2}+ 2 \;\textup{ pairs};\;C=L_{3}^{5}$; $A,B,C$ and $\stackrel{\leftarrow}{A},B,C$.

$10$ pairs: $\underline{3},\underline{8},e,\underline{d},\underline{2},\underline{4},\underline{g},\underline{b},\underline{7},\underline{f},b,\underline{a},5,9,6,1$ and $c,e,9,1,6,5$.

For \underline{$n=28$} the small cases are $\{12,13,14,15,17,18\}$.

$11,12,15$ pairs: $A=A_{13}^{28};\;B=S_{4} +\;1\;\textup{pair};\;C=L_{5}^{10}$; $A,sp(B),C$ and $\stackrel{\leftarrow}{A},sp(B),C$.

$13$ pairs: $\underline{s},\underline{q},\underline{r},\underline{n},\underline{p},\underline{i},\underline{o},d,e,\underline{m},8,5,\underline{k},2,\underline{j},\underline{h},\underline{g},f,c,a,b,6,3,9,4,7,1$\\
and $e,c,5,7,3,f,d,9,2,b,a,1,8,4,6$.

$14$ pairs: $\underline{s},\underline{q},\underline{r},\underline{n},\underline{p},\underline{i},\underline{o},d,e,\underline{m},8,5,\underline{k},2,\underline{j},\underline{h},\underline{g},\underline{f},c,a,b,6,3,9,4,7,1$\\
and $b,c,6,4,9,e,3,d,a,1,5,8,2,7$.

$16$ pairs: $\underline{s},\underline{q},\underline{r},\underline{n},\underline{p},\underline{i},\underline{o},\underline{d},\underline{e},\underline{m},8,5,\underline{k},2,\underline{j},\underline{h},\underline{g},\underline{f},c,a,b,6,3,9,4,7,1$\\
and $6,4,5,b,c,a,7,1,9,2,8,3$.

$17$ pairs: $\underline{s},\underline{q},\underline{r},\underline{n},\underline{p},\underline{i},\underline{o},\underline{d},\underline{e},\underline{m},8,5,\underline{k},2,\underline{j},\underline{h},\underline{g},\underline{f},\underline{c},a,b,6,3,9,4,7,1$\\
and $5,7,3,9,2,b,a,1,8,4,6$.

$18$ pairs: $\underline{s},\underline{q},\underline{r},\underline{n},\underline{p},\underline{i},\underline{o},\underline{d},\underline{e},\underline{m},8,5,\underline{k},2,\underline{j},\underline{h},\underline{g},\underline{f},\underline{c},a,\underline{b},6,3,9,4,7,1$\\
and $6,4,5,9,8,a,3,1,7,2$.

For \underline{$n=40$} the small cases are $\{16-23,25, 26\}$.

$16,17,21$ pairs: $A=A_{19}^{40};\;B=S_{5}+\;2\;\textup{pairs};\;C=L_{6}^{14}$; $A,sp(B),C$ and $\stackrel{\leftarrow}{A},sp(B),C$.

$23$ pairs: $A=A_{17}^{40};\;B=6-\textup{ext}\;S_{7}+\;1\;\textup{pair};\;C=L_{8}^{15}$; $A,B,C$ and $\stackrel{\leftarrow}{A},B,C$;

$25,26,27,28$ pairs: $A,B,C$ and $A,B,\stackrel{\leftarrow}{C}$ ($0,1,2,3$ pairs from \cite{silvesan}, Table 1, sequence 5).

$18$ pairs: $S_{13}$ with two $L_{14}^{27}$ with five pairs in common. The two $L_{14}^{27}$ are:
$x,A,C,\underline{v},o,E,\underline{p},m,\linebreak s,\underline{D},B,r,f,h,\underline{l},z,u,y,\underline{e},k,w,n,j,r,t,i,g$ and $t,w,A,q,m, z,C,E,h,B,y,s,n,j,x,o,u,f,i,g,k,r$.

$19$ pairs: $S_{13}$ with two $L_{14}^{27}$ with six pairs in common. The two $L_{14}^{27}$ are:
$x,A,C,v,\underline{o},E,\underline{p},m,\linebreak s,D,B,r,f,\underline{h},\underline{l},z,u,y,e,\underline{k},w,n,\underline{j},q,t,i,g$ and $y,w,z,q,m, C,D,B,E,x,u,A,g,v,i,n,s,t,e,f,r$.

$20,22$ pairs: $k,A_{19}^{40},k,2,5,2,4,6,3,5,4,3,l,6,A_{19}^{40},l,L_7^{13},1,1$
and
$l,A_{19}^{40},6,l, 3,4,5,3,6,4,2,5,2,\linebreak k,A_{19}^{40},k;,\stackrel{\leftarrow}{L_7^{13}},1,1$.

For \underline{$n=52$} the small case is $\{31\}$.

$31$ pairs: $A=A_{27}^{52},\;B=L_3^5(1,1),\;C=hL_8^{18}*(2,0,2)$; $A,B,C$ and $A,B$(1 pair from \cite{silvesan}, Table 1, sequence 5) $C$ (1 pair, see below). Two $hL_8^{18}$ with one pair in common are:
$l,0,m,9,b,\underline{p},h,i,j,o,a,n,g,e,k,f,8,c,d$ and $i,0,o,a,h,j,l,9,m,n,d,k,8,e,b,f,g,c$.

For \underline{$n=100$} the small cases are $\{47-50\}$.

$47$ pairs: $A=A_{47}^{100},\;B=L_2^{15},\;C=L_{17}^{37},\;D=(1,1)$; $A,B,C,D$ and $\stackrel{\leftarrow}{A},\stackrel{\leftarrow}{B}$ (9 pairs in common, see below) $C,D$. Two $L_2^{15}$ with nine pairs in common are:
$\underline{5},\underline{b},\underline{e},g,\underline{7},5,\underline{8},f,\underline{9},c,\underline{a},d,3,6,\underline{4},\underline{2}$ and $f,c,d,g,6,3$.

$48$ pairs: $A,B,C,D$ and $A,\stackrel{\leftarrow}{B}$ (0 pairs from \cite{shalaby}),$\stackrel{\leftarrow}{C}$ (0 pairs from \cite{silvesan}, Table 1, sequence 3),$D$.

$49$ pairs: $A=A_{51}^{100},\;B=S_{12},\;C=L_{13}^{37}$; $A,B,C$ and $\stackrel{\leftarrow}{A},B, C$.

$50$ pairs: $A,B,C,D$ and $A,\stackrel{\leftarrow}{B}$ (2 pairs, see below), $\stackrel{\leftarrow}{C}$ (0 pairs from \cite{silvesan}, Table 1, sequence 3), $D$. Two $L_2^{15}$ with two pairs in common are:
$6,c,e,\underline{5},b,d,g,8,4,f\linebreak 9,a,7,2,\underline{3}$ and $g,b,7,d,f,c,6,9,a,8,e,4,2$.

\mbox {\boldmath ${n\equiv 5\;\mathrm{(mod\;12)}}$:}

For \underline{$n=17$} the small cases are $\{5,7,9,10,11\}$.

$5$ pairs: $\underline{f},2,\underline{6},4,g,3,7,h,\underline{e},a,\underline{8},5,c,9,\underline{1}$ and $4,d,g,5,h,c,2,a,9,7,b,3$.

$7$ pairs: $\underline{h},\underline{f},\underline{g},\underline{c},a,\underline{e},7,\underline{d},4,2,\underline{b},9,6,8,1,3,5$ and $7,3,5,a,9,6,8,1,4,2$.

$9$ pairs: $\underline{h},\underline{f},\underline{g},\underline{c},\underline{a},\underline{e},7,\underline{d},4,2,\underline{b},\underline{9},6,8,1,3,5$ and $6,3,4,7,2,8,5,1$.

$10$ pairs:  $\underline{d},4,9,6,\underline{8},\underline{e},\underline{g},\underline{h},\underline{5},a,c,f,\underline{b},\underline{7},1,\underline{3},\underline{2}$ and $f,1,9,c,6,a,4$.

$11$ pairs: $c,\underline{9},\underline{d},\underline{f},7,e,\underline{g},\underline{a},\underline{5},\underline{h},b,8,\underline{3},\underline{6},4,\underline{1},\underline{2}$ and $4,b,8,c,e,7$.

For \underline{$n=29$} the small case is $\{15,17\}$.

$15$ pairs: $A=A_{13}^{29},\;B=S_5,\;C=L_{6}^{11}$; $A,sp(B),C$ and $A,sp(B),\stackrel{\leftarrow}{C}$ ($1$ pair from \cite{silvesan}, Table 1, sequence 5).

$17$ pairs: $\underline{t},\underline{r},\underline{s},\underline{o},\underline{m},\underline{q},\underline{j},\underline{p},\underline{g},\underline{e},\underline{n},b,9,\underline{l},5,\underline{k},1,\underline{i},\underline{h},\underline{f},\underline{d},a,6,c,3,7,4,8,2$\\
and $a,3,7,6,b,9,5,c,4,8,1,2$.

For \underline{$n=41$} the small cases are $\{17,18,19,23,24,25,27\}$.

$16,17,18,19,22$ pairs: $A=A_{19}^{41},\;B=hS_6+1\;\textup{pair},\;C=hL_7^{15}$; $A,sp(B),C$ and $\stackrel{\leftarrow}{A},sp(B),C$.

$23$ pairs: $\underline{E},\underline{C},\underline{D},\underline{z},\underline{x},\underline{B},\underline{u},\underline{A},\underline{r},\underline{p},\underline{y},\underline{m},\underline{l},\underline{w},i,\underline{v},e,c,\underline{t},8,\underline{s},4,2,\underline{q},\underline{o},\underline{n},\underline{l},\underline{k}, h,\underline{j},g$\\
$d,5,7,f,3,6,b,a,9,1$ and $g,d,f,6,3,5,i,e,8,h,c,4,7,1,b,9,a,2$.

$24$ pairs: $\underline{E},\underline{C},\underline{D},\underline{z},\underline{x},\underline{B},\underline{u},\underline{A},\underline{r},\underline{p},\underline{y},\underline{m},\underline{l},\underline{w},\underline{i},\underline{v},e,c,\underline{t},8,\underline{s},4,2,\underline{q},\underline{o},\underline{n},\underline{l},\underline{k}, h,\underline{j},g$\\
$d,5,7,f,3,6,b,a,9,1$ and $d,a,5,9,3,g,f,c,h,4,8,6,e,1,b,7,2$.

$25$ pairs: $\underline{E},\underline{C},\underline{D},\underline{z},\underline{x},\underline{B},\underline{u},\underline{A},\underline{r},\underline{p},\underline{y},\underline{m},\underline{l},\underline{w},\underline{i},\underline{v},e,c,\underline{t},8,\underline{s},4,2,\underline{q},\underline{o},\underline{n},\underline{l},\underline{k}, \underline{h},\underline{j},g$\\
$d,5,7,f,3,6,b,a,9,1$ and $d,8,b,3,5,e,f,7,g,6,1,4,c,9,a,2$.

$27$ pairs: $\underline{E},\underline{C},\underline{D},\underline{z},\underline{x},\underline{B},\underline{u},\underline{A},\underline{r},\underline{p},\underline{y},\underline{m},\underline{l},\underline{w},\underline{i},\underline{v},e,c,\underline{t},8,\underline{s},4,2,\underline{q},\underline{o},\underline{n},\underline{l},\underline{k}, \underline{h},\underline{j},\underline{g},\linebreak d,5,7,\underline{f},3,6,b,a,9,1$ and $d,8,b,3,5,c,7,4,6,e,2,9,a,1$.

For \underline{$n=53$} the small case is $\{31,32,35\}$.

$31$ pairs: $A=A_{25}^{53},\;B=S_8,\;C=L_9^{20}$; $A,B,C$ and $A,B$ (5 pairs) $C$ (1 pair). The two $L_9^{20}$ with one pair in common are:
$q,o,c,r,n,d,p,l,g,b,9,s,k, \underline{m},e,i,j,h,f,a$ and $o,j,n,s,q,b,k,a,l,9,i,p,r,d,h,e,g,c,f$.

$32$ pairs: $A,B,C$ and $A,B$ (5 pairs) $C$ (2 pairs). The two $L_9^{20}$ with two pairs in common are:
$q,o,c,r,\underline{n},d,p,l,g,b,9,s,k,m,e,\underline{i},j,h,f,a$ and
$d,m,q,g,o,r,j,s,f,b,9,p,h,k,l,c,a,e$.

$35$ pairs: $\underline{R},\underline{P},\underline{Q},\underline{M},\underline{K},\underline{O},\underline{H},\underline{N},\underline{E},\underline{C},
\underline{L},\underline{z},\underline{x},\underline{J},\underline{u},\underline{I},\underline{r},\underline{p},\underline{G},\underline{l},\underline{F},
\underline{h},f,\underline{D}, b,8,\underline{B},5,\underline{A},1,\underline{y},\linebreak \underline{w},\underline{v},\underline{t},\underline{s},\underline{q},
\underline{o},\underline{m},\underline{k},\underline{n},g,d,a,7,j,4,i,6,9,e,c,3,2$ and $g,a,e,5, 8,4,h,f,b,6,7,9,i,2,d,c,3,1$.

For \underline{$n=65$} the small cases are $\{30,31,32,39,40\}$.

$25-31,35$ pairs: $A=A_{31}^{65},\;B=S_9+1\;\textup{pair},\;C=3-\textup{ext}\;L_{10}^{24}$; $A,sp(B),C$ and $\stackrel{\leftarrow}{A},sp(B),C$.

$32$ pairs: $A=A_{33}^{65},\;B=S_{8},\;C=L_{9}^{24}$; $A,B,C$ and $\stackrel{\leftarrow}{A},B,C$.

$39$ pairs: $A=A_{37}^{65},\;B=S_5,\;C=L_6^{23}$; $A,B,C$ and $A,B$ (1 pair) $C$ (1 pair). The two $L_6^{23}$ with one pair in common are:
$n,k,a,o,7,b,j,p,6,\underline{s},q,m,r, l,f,c,8,i,h,d,g,e,9$ and
$b,g,p,9,m,8,o,r,a,l,e,h,q,k,i,n,7,d,j,f,c,6$.

$40,41,42$ pairs: $A=A_{27}^{65},\;B=11-\textup{ext}\;S_{12}+1\;\textup{pair},\;C=L_{13}^{25}$; $A,B,C$ and $A,B,\stackrel{\leftarrow}{C}$ ($0,1,2$ pairs from \cite{silvesan}, Table 1, sequence 5).

For \underline{$n=77$} the small cases are $\{34-40\}$.

$34$ pairs: $A=A_{37}^{77},\;B=L_2^{11},\;C=L_{13}^{28},\;D=(1,1)$; $A,B,C,D$ and $\stackrel{\leftarrow}{A},B$ (5 pairs), $C,D$. The two $L_2^{11}$ with $5$ pairs in common are:
$a,\underline{b},\underline{7},\underline{8},c,3,\underline{9},6,5,\underline{4},2$ and $5,a,3,c,6,2$.

$35$ pairs: $A=A_{37}^{77},\;B=L_2^{11},\;C=L_{13}^{28},\;D=(1,1)$; $A,B,C,D$ and $\stackrel{\leftarrow}{A},B$ (6 pairs) $C,D$. The two $L_2^{11}$ with $6$ pairs in common are:
$\underline{8},c,a,6,7,\underline{5},9,\underline{b},\underline{3},\underline{4},\underline{2}$ and $a,7,c,9,6$.

$37$ pairs: $A=A_{37}^{77},\;B=L_2^{11},\;C=L_{13}^{28},\;D=(1,1)$; $A,B,C(1,1)$ and $A,d(B),\stackrel{\leftarrow}{C(1,1)}$ ($0$ pairs from Appendix E, Table \ref{append}, sequence 2).

$38$ pairs: $A,B,C,D$ and $A,\stackrel{\leftarrow}{B}, C$ (0 pairs) $D\Rightarrow$. The two $L_{13}^{28}$ with no pairs in common are:
$C,e,w,D,v,s,u,E,B,g,x,k,A,d,z,o,f,y,i,l,n,h,r,t,j,q,m,p$ and
$u,h,z,E,t,j,C,w,s,m,\linebreak g,A,D,l,v,p,x,B,n,q,y,d,f,r,e,k,o,i$

$39$ pairs: $A,B,C,D$ and $A,\stackrel{\leftarrow}{B},C$ (1 pairs) $D$. The two $L_{13}^{28}$ with $1$ pair in common are:
$C,e,w,D,v,s,u,E,B,g,x,\underline{k},A,d,z,o,f,y,i,l,n,h,r,t,j,q,m,p$ and
$p,r,z,C,y,E,B,D,m,\linebreak q,u,e,A,i,x,d,h,o,s,g,w,t,l,v,n,j,f$.

$40$ pairs: $A,B,C,D$ and $A,d(B),\stackrel{\leftarrow}{C}$ ($2$ pairs from \cite{silvesan}, Table 1, sequence 1), $D$;

$36$ pairs: $A=A_{41}^{77},\;B=S_{8},\;C=L_{9}^{28}$; $A,B,C$ and $\stackrel{\leftarrow}{A},B,C$.

\mbox {\boldmath ${n\equiv 8\;\mathrm{(mod\;12)}}$:}

For \underline{$n=20$} the small cases are $\{2,6,7,8,9,10,11,12,13\}$.

$2$ pairs: $2-\textup{near}\; S_3, hL_4^{17}*(2,0,2)$ and  $2-\textup{near} \;S_3,\stackrel{\leftarrow}{hL_4^{17}*(2,0,2)}$ ($0$ pairs from Appendix E, Table \ref{hooked}, sequence 1).

$6$ pairs: $\underline{g},\underline{d},4,j,b,\underline{i},h,c,8,f,7,e,\underline{a},5,9,1,\underline{6},\underline{2},3$ and
$9,c,3,b,e,j,k,f,7,h,4,8,5,1$.

$7$ pairs: $\underline{a},\underline{6},k,\underline{h},7,c,f,\underline{5},j,\underline{i},9,e,8,d,g,2,b,\underline{1},4,\underline{3}$ and
$7,k,d,g,4,c,j,e,f,8,2,b,9$.

$8,11$ pairs: $A=A_9^{20},\;B=hS_{3}+1\;\textup{pair},\;C=L_{4}^{7}$; $A,sp(B),C$ and $\stackrel{\leftarrow}{A},sp(B),C$;

$9$ pairs: $\underline{k},\underline{i},g,e,\underline{c},a,8,6,\underline{4},\underline{2},j,9,h,f,\underline{d},b,\underline{7},\underline{5},3,\underline{1}$
and $j,f,h,b,g,3,e,a,9,6,8$.

$10,11,12$ pairs: $A=A_9^{20},\;B=hS_{3}+1\;\textup{pair},\;C=L_{4}^{7}$; $A,sp(B),C$ and $A,sp(B),\stackrel{\leftarrow}{C}$ ($0,1,2$ pairs from \cite{silvesan}, Table 1, sequence 5).

$13$ pairs: $\underline{k},\underline{i},\underline{j},\underline{g},\underline{d},\underline{h},\underline{a},\underline{g},7,5,\underline{e},1,\underline{c},\underline{b},\underline{9},6,3,\underline{8},2,4$
and $6,3,4,7,2,1,5$.

For \underline{$n=32$} the small cases are $\{11,14,15,16,20,21\}$.

$11$ pairs: $S_{10}$ with two $L_{11}^{22}$ with one pair in common. The two $L_{11}^{22}$ are:
$r,0,t,k,h,w,s,p,i,v,\linebreak c,u,g,n,o,f,\underline{d},q,l,e,m,b,j$
and $p,0,v,w,j,q,o,f,u,h, e,g,m,t,r,s,k,i,l,c,n,b$.

$14$ pairs: $\underline{w},\underline{u},\underline{v},\underline{r},\underline{p},\underline{t},\underline{m},\underline{s},\underline{j},h,\underline{q},e,
c,\underline{o},7,\underline{n},6,2,\underline{l},\underline{k},i,g,d,f,a,b,3, 4,5,9,8,1$
and $e,b,7,3,\linebreak 5,i,g,h,d,f,c,4,2,a,9,1,8,6$.

$15$ pairs: $\underline{w},\underline{u},\underline{v},\underline{r},\underline{p},\underline{t},\underline{m},\underline{s},\underline{j},h,\underline{q},
e,c,\underline{o},7,\underline{n},6,2,\underline{l},\underline{k},\underline{i},g,d,f,a,b,3,4,5,9,8,1$
and $g,f,a,5,\linebreak 1,3,h,e,c,d,6,4,2,b,9,7,8$.

$16$ pairs: $S_{10}$ with two $L_{11}^{22}$ with five pairs in common. The two $L_{11}^{22}$ are:
$r,0,t,k,h,w,s,\underline{p},i,v,\linebreak c,u,\underline{g},n,\underline{o},\underline{f},d,r,l,e,m,b,\underline{j}$
and $t,0,k,v, n,u,f,w,l,r, c,d,s,q,m,i,b,h$.

$20$ pairs: $\underline{w},\underline{u},\underline{v},\underline{r},\underline{p},\underline{t},\underline{m},\underline{s},\underline{j},\underline{h},
\underline{q},\underline{e},c,\underline{o},7,\underline{n},6,2,\underline{l},\underline{k},\underline{i},\underline{g},\underline{d},\underline{f},a,b,3,4,5,9,8,1$
and $a,5,1,3,\linebreak c,6,4,2,b,9,7,8$.

$21$ pairs: $\underline{w},\underline{u},\underline{v},\underline{r},\underline{p},\underline{t},\underline{m},\underline{s},\underline{j},\underline{h},
\underline{q},\underline{e},\underline{c},\underline{o},7,\underline{n},6,2,\underline{l},\underline{k},\underline{i},\underline{g},\underline{d},
\underline{f},a,b,3,4,5,9,8,1$ and $8,3,4,9,\linebreak 6,a,5,b,1,7,2$.

For \underline{$n=44$} the small case is $\{15,25\}$.

$15$ pairs: $S_{12}$ with two $L_{13}^{32}$ with three pairs in common. The two $L_{13}^{32}$ are:
$m,H,I,o,t,B,F,h,u,\linebreak A,\underline{D},w,j,n,y,\underline{e},\underline{z},k,C,d,G,v,E,s,x,q,p,r,f,l,i,g$ and $p,k,G,E,f,m,I,u,s,v,n,A,w,F,t,H,\linebreak r,B,j,C,
y,d,x,l,q,o,i,g,h$.

$25$ pairs: $A=A_{19}^{44},\;B=S_8,\;C=L_9^{17}$; $A,B,C$ and $\stackrel{\leftarrow}{A},B,C$.

For \underline{$n=56$} the small case is $\{19,33\}$.

$19$ pairs: $S_{18}$ with two $L_{19}^{38}$ with one pair in common. The two $L_{18}^{38}$ are: $I,0,E,G,S,C,H,R,M,\linebreak B,L,O,A,m,n,z,Q,T,k,U,r,y,v,P,K,D,
s,N,J, x,F,l,\underline{j},w,p,u,o,t,q$ and $F,0,J,K,T,E,s,A,\linebreak G,L,x,S,q,O,n,u,H,U,M,R N,w,Q,D,B,m,p,I,o,r,C,z,s,k,t,l,v$.

$33$ pairs: $A=A_{27}^{56},\;B=S_8,\;C=L_9^{21}$; $A,B,C$ and $A,B$ (5 pairs) $C$ (1 pair). The two $L_9^{21}$ with one pair in common are:
$e,s,f,m,q,a,k,p,\underline{t},i,n,o,b, r,c,l,j,g,d,9,h$ and
$n,a,d,j,p,s,q,l,r,h,c,i,o,9,m,g,b,f,k,e$.

For \underline{$n=68$} the small cases are $\{23,31,32,33,34,41\}$.

$23$ pairs: $S_{22}$ with two $L_{23}^{46}$ with one pair in common. The two $L_{23}^{46}$ are: $V,0,P,U,64,K,F,G,68,\linebreak H,J,X,x,67,r,62,65,I,Z,Q,L,66,A,B,
S,Y,W,t, 63,p,R,y,v,E,M,T,s,\underline{z},o,N,q,O,w,u,n,C,D$ and $P,0,62,66,N,U,Y,O,67, A,I,V,68,t,y,E,W,G,v,X,63,64,H,r,w,x,L,65,S,Z,B,T,u,K,R,\linebreak p,J,D, Q,
n,M,t,s,A,C,F$.

$27-36,39$ pairs: $A=A_{29}^{68},\,B=S_{12}+1\,\textup{pair},\,C=10-\textup{ext}\,L_{13}^{26}$; $A,B,C$ and $\stackrel{\leftarrow}{A},B,C$;

$41$ pairs: $A=A_{39}^{68},\,B=L_3^5,\,C=hL_{8}^{22}*(2,0,2),\;D=(1,1)$; $A,B,C,D$ and $A,\stackrel{\leftarrow}{B}$ ($0$ pairs from \cite{silvesan}, Table 1, sequence 5) $\stackrel{\leftarrow}{C}$ ($1$ pair from Appendix E, Table \ref{hooked}, sequence 3), $D$.

For \underline{$n=80$} the small cases are $\{35-42,49\}$.

$32-44,47$ pairs: $A=A_{33}^{80},\,B=hS_{15}+1\,\textup{pair},\,C=13-\textup{ext}\,L_{16}^{31}$; $A,B,C$ and $\stackrel{\leftarrow}{A},B,C$

$49$ pairs: $A=A_{33}^{80},\;B=14-\textup{ext}\;S_{15}+1\;\textup{pair},\;C=L_{16}^{31}$; $A,B,C$ and $A,B,\stackrel{\leftarrow}{C}$ ($0,1,2$ pairs from \cite{silvesan}, Table 1, sequence 5).

\mbox {\boldmath ${n\equiv 9\;\mathrm{(mod\;12)}}$:}

For \underline{$n=21$} the small cases are $\{2,7-13\}$.

$2$ pairs: $2-\textup{near}\; S_3,L_4^{18}*(2,0,2)$ and $2-\textup{near}\; S_3,\stackrel{\leftarrow}{L_4^{18}*(2,0,2)}$ ($0$ pairs from Appendix E, Table \ref{hooked}, sequence 3);

$7$ pairs: $\underline{l},\underline{j},\underline{k},\underline{g},e,\underline{i},b,\underline{h},8,5,\underline{f},1,d,c,a,7,9,2,3,6,4$ and
$d,c,3,6,e,1,b,9,a, 8,4,2,7,5$.

$8$ pairs: $A=A_{13}^{21},\;B=S_1,\;C=L_2^7$; $A,B,C$ and $\stackrel{\leftarrow}{A},B,C$.

$9$ pairs: $\underline{l},\underline{j},\underline{k},\underline{g},\underline{e},\underline{i},b,\underline{h},8,5,\underline{f},1,\underline{d},c,a,7,9,2,3,6,4$ and
$9,4,8,3,c,b,a,2,6, 7,1,5$.

$10$ pairs:  $\underline{e},j,9,f,\underline{k},\underline{l},\underline{b},g,\underline{h},6,c,\underline{7},\underline{i},5,\underline{8},d,\underline{a},4,\underline{1},2,3$ and
$9,d,6,f,j,2,3,g, c,4,5$.

$11$ pairs: $\underline{9},k,\underline{3},\underline{j},l,i,\underline{g},7,\underline{a},6,\underline{f},d,h,\underline{c},\underline{e},b,\underline{8},2,
\underline{4},\underline{5},1$ and $7,h,k,d,2,b,i,l,1,6$.

$12$ pairs: $\underline{l},\underline{j},\underline{k},\underline{g},\underline{e},\underline{i},\underline{b},\underline{h},8,5,\underline{f},1,\underline{d},
\underline{c},\underline{a},7,9,2,3,6,4$ and $4,7,3,8,2,9,6,1,5$.

$13$ pairs:  $\underline{l},\underline{j},\underline{k},\underline{g},\underline{e},\underline{i},\underline{b},\underline{h},8,5,\underline{f},1,\underline{d},
\underline{c},\underline{a},7,\underline{9},2,3,6,4$
and $1,5,2,8,3,4,7$.

For \underline{$n=33$} the small cases are $\{11,14,15,19,20\}$.

$11$ pairs: $S_{9}$ with two $L_{10}^{24}$ with two pairs in common. The two $L_{10}^{24}$ are:
$s,\underline{b},t,q,x,v,a,p,u,\linebreak w,d,f,c,k,i,m,r,n,o,e,\underline{l},g,j,h$
and $g,w,s,v,c,x,t,p,j, m,u,o,d,f,q,a,r,h,n,i,k,e$.

$14$ pairs in common take $S_{9}$ with two $L_{10}^{24}$ with five pairs in common. The two $L_{10}^{24}$ are:
$s,b,t,\underline{q},x,v,a,\underline{p},u,w,d,\underline{f},c,k,i,m,r,n,\underline{o},e,l,g,\underline{j},h$
and $a,r,k,t,u,x,d,v,c,w,n,s,i,l,n,b,\linebreak e,h,g$.

$15$ pairs:  $\underline{x},\underline{v},\underline{w},\underline{s},\underline{q},\underline{u},\underline{n},\underline{t},\underline{k},i,\underline{r},e,c,\underline{p},7,
\underline{o},6,2,\underline{m},\underline{l},\underline{j},h,g,d,f,a,b,3,4,5,9,8,\linebreak 1$
and $h,g,4,8,7,2,i,f,c,e,b,d,5,6,1,a,9,3$.

$19$ pairs: $\underline{x},\underline{v},\underline{w},\underline{s},\underline{q},\underline{u},\underline{n},\underline{t},\underline{k},\underline{i},\underline{r},e,c,
\underline{p},7,\underline{o},6,2,\underline{m},\underline{l},\underline{j},\underline{h},\underline{g},d,\underline{f},a,b,3,\linebreak 4, 5,9,8,1$
and $d,a,b,3,4,e,9,7,5,c,2,6,8,1$.

$20$ pairs: $\underline{x},\underline{v},\underline{w},\underline{s},\underline{q},\underline{u},\underline{n},\underline{t},\underline{k},\underline{i},\underline{r},
\underline{e},c,\underline{p},7,\underline{o},6,2,\underline{m},\underline{l},\underline{j},\underline{h},\underline{g},d,\underline{f},a,b,3,4, 5,\linebreak 9,8,1$
and $a,2,4,5,c,b,7,d,3,1,9,8,6$.

For \underline{$n=45$} the small cases are $\{15,18,19,21-24,27,28\}$.

$23,24,25,26,29$ pairs: $A=A_{21}^{45},\;B=hS_{6}+2\;\textup{pairs},\;C=L_{7}^{16}$; $A,sp(B),C$ and $A,sp(B),\stackrel{\leftarrow}{C} $($0$ pairs from \cite{silvesan}, Table 1, sequence 1);

$18,19,20,21,24$ pairs: $A=A_{21}^{45},\;B=hS_{6}+2\;\textup{pairs},\;C=L_{7}^{16}$; $A,sp(B),C$ and $\stackrel{\leftarrow}{A},sp(B),C$.

$28,29,30$ pairs: $A=A_{19}^{45},\;B=7-\textup{ext}\;S_{8}+1\;\textup{pair},\;C=L_{9}^{17}$; $A,B,C$ and $A,B,\stackrel{\leftarrow}{C} $($0,1,2$ pairs from \cite{silvesan}, Table 1, sequence 5);

$18-23,26$ pairs: $A=A_{19}^{45},\;B=S_8+1\;\textup{pair},\;C=7-\textup{ext}\,L_9^{17}$; $A,sp(B),C$ and $\stackrel{\leftarrow}{A},sp(B),C$.

$15$ pairs: $S_{14}$ with two $L_{15}^{31}$ with one pair in common. The two $L_{14}^{31}$ are:
$H,0,z,\underline{C},G,D,o,i,\linebreak B,J,s,t,j,n,F,w,g,x,I,u,m,E,A,y,f,q,l,v,r,k,h,p$ and $D,0,A,s,m,I,G,H,B,p,v,i,y,\linebreak q,g,w,J,F,x,E,f,l,c,u,j,t,h,o,n,r,k$.

$27$ pairs: $\underline{J},\underline{H},\underline{I},\underline{E},\underline{C},\underline{G},\underline{z},\underline{F},\underline{w},\underline{u},
\underline{D},\underline{r},\underline{p},\underline{B},\underline{l},\underline{A},h,f,\underline{y},6,\underline{x},9,7,\underline{v},3,\underline{t},
\underline{s},\underline{q}, \underline{o},\underline{n},\underline{m},\underline{k}$\\$i,g,\underline{j},d,8,2,4,e,5,c,b,a,1$ and
 $g,d,e,8,5,1,h,f,6,i,a,7,4,2,c,b,9,3$.

For \underline{$n=57$} the small cases are $\{19,22,23,25-32,35,36\}$.

$19$ pairs: $S_{17}$ with two $L_{18}^{40}$ with 2 pairs in common. The two $L_{18}^{40}$ are:
$t,x,o,R,F,N,L,U,H,\linebreak C,\underline{G},V,u,m,M,S,T,w,Q,r,J,E,I,O,k,y,l,P,\underline{B},K, n,p,A,D,q,i,z,j,v,s$ and
$N,C,k,L,u,\linebreak s,D,P,E,J,w,O,M,F,l,U,A,V,S, v,I,R,T,H,Q,j,K,y,q,r,z,n,i,o,p,t,x,m$.

$22-27,30$ pairs: $A=A_{27}^{57},\;B=S_8+1\;\textup{pair},\;C=3-\textup{ext}\,L_9^{21}$; $A,sp(B),C$ and $\stackrel{\leftarrow}{A},sp(B),C$.

$31$ pairs: $A=A_{29}^{57},\;B=L_2^7,\;C=L_9^{20},\;D=(1,1)$; $A,B,C,D$ and $A,D,\stackrel{\leftarrow}{B},C$ (1 pair). The two $L_9^{20}$ with one pair in common are:
$q,o,c,r,n,d,p,l,g,b,9\linebreak s,k,\underline{m},e,i,j,h,f,a$ and
$o,j,n,s,q,b,k,a, l,9,i,p,r,d,h,e,g,c,f$.

$35$ pairs: $A=A_{33}^{57},\;B=S_4,\;C=L_5^{20}$; $A,B,C$ and $A,B$ ( 1 pair) $C$ (1 pair). The two $L_5^{20}$ with one pair in common are:
$b,8,a,h,o,k,l,f,7,n,i,c,m\linebreak j,e,g,9,6,d,\underline{5}$ and
$d,a,l,9,n,f,m,o,b,8,j,i,k,h,7,g,e,6,c$.

$36$ pairs: $A=A_{27}^{57},\;B=3-\textup{ext}\;S_{8}+1\;\textup{pair},\;C=L_{9}^{21}$; $A,B,C$ and $A,B,\stackrel{\leftarrow}{C}$ ($0$ pairs from \cite{silvesan}, Table 1, sequence 3);

$30$ pairs: $A=A_{27}^{57},\;B=3-\textup{ext}\;S_{8}+1\;\textup{pair},\;C=L_{9}^{21}$; $A,sp(B),C$ and $\stackrel{\leftarrow}{A},sp(B),C$.

$28$ pairs: $A=A_{29}^{57},\;B=L_2^7,\;C=L_9^{20},\;D=(1,1)$; $A,B,C,D$ and $\stackrel{\leftarrow}{A},B,C,D$.

$29$ pairs: $A=A_{29}^{57},\;B=L_2^7,\;C=L_9^{20},\;D=(1,1)$; $A,B,C(1,1)$ and $A,\stackrel{\leftarrow}{B},\stackrel{\leftarrow}{C(1,1)}$ ($0$ pairs from Appendix E, Table \ref{append}, sequence 2).

$32$ pairs: $A=A_{29}^{57},\;B=L_2^7,\;C=L_9^{20},\;D=(1,1)$; $A,B,(1,1)C$ and $A,\stackrel{\leftarrow}{B},\stackrel{\leftarrow}{(1,1)C}$ ($3$ pairs from Appendix E, Table \ref{append}, sequence 2).

For \underline{$n=69$} the small case is $\{43\}$.

$43$ pairs: $A=A_{29}^{69},\;B=S_{13},\;C=L_{14}^{27}$; $A,B,C$ and $A,B,\stackrel{\leftarrow}{C}$ ($1$ pair from \cite{silvesan}, Table 1, sequence 5).

For \underline{$n=81$} the small cases are $\{51\}$.

$51$ pairs: $A=A_{37}^{81},\,B=2-\textup{near} S_{14},\,C=hL_{15}^{30}*(2,0,2)$; $A,B,C$ and $A,B,\stackrel{\leftarrow}{C}$ ($1$ pair from Appendix E, Table \ref{hooked}, sequence 3).

\begin{center}
{\bf Appendix B: The intersection spectrum of two Skolem sequences of order \mbox {\boldmath ${1\leq n\leq 9}$}}
\end{center}

Here we list the spectrum of two (hooked) Skolem sequences of order $n$ for $1\leq n\leq 9$. For two (hooked) Skolem sequences of order $n$ with no pairs in common see \cite{shalaby}. For two hooked Skolem sequences of order $n$ with
 $n$ pairs in common take the same sequence twice.

There is only one $S_1$ and only one $hS_2$, so $sp(S_1)=\{1\}$ and $sp(hS_2)=\{2\}$. $sp(hS_3)=\{0,3\}$.

$sp(S_4)=\{0,1,4\}$. 

Two $S_4$ with one pair in common: $2,3,\underline{4},1$ and $1,3,2$.

$sp(S_5)=\{0,1,5\}$. 

Two $S_5$ with one pair in common: $5,1,3,\underline{4},2$ and $1,5,2,3$.

$sp(hS_6)=\{0,1,2,3,6\}$. 

Two $hS_6$ with one pair in common: $\underline{5},0,4,6,3,2,1$ and $1,6,2,4,3$

Two $hS_6$ with two pairs in common: $6,0,\underline{3},4,5,\underline{2},1$ and $5,0,4,6,3$

Two $hS_6$ with three pairs in common: $\underline{4},0,1,5,\underline{6},2,\underline{3}$ and $5,2,1$

$sp(hS_7)=\{0,1,2,3,4,7\}$. Two $hS_7$ with one pair in common: $6,0,\underline{7},2,3,5,4,1$ and $5,0,4,6,2,3,1$

Two $hS_7$ with two pairs in common: $\underline{7},0,\underline{6},2,5,4,3,1$ and $1,4,5,3,2$

Two $hS_7$ with three pairs in common: $\underline{7},0,\underline{3},\underline{6},4,5,2,1$ and $2,5,4,1$

Two $hS_7$ with four pairs in common: $\underline{7},0,\underline{2},\underline{5},\underline{6},4,3,1$ and $3,4,1$

$sp(S_8)=\{0,1,2,3,4,5,8\}$.

Two $S_8$ with one pair in common: $2,5,6,7,3,\underline{8},4,1$ and $1,3,7,4,6,5,2$

Two $S_8$ with two pairs in common: $3,4,6,5,\underline{7},\underline{8},1,2$ and $5,1,6,4,2,3$

Two $S_8$ with three pairs in common: $3,1,4,\underline{6},\underline{7},\underline{8},5,2$ and $1,2,5,3,4$

Two $S_8$ with four pairs in common: $\underline{1},\underline{2},5,6,\underline{7},\underline{8},3,4$ and $6,5,4,3$

Two $S_8$ with five pairs in common: $\underline{3},1,6,\underline{7},\underline{8},2,\underline{5},\underline{4}$ and $6,2,1$

$sp(S_9)=\{0,1,2,3,4,5,6,9\}$.

Two $S_9$ with one pair in common: $\underline{9},5,8,4,1,7,6,2,3$ and $7,5,8,6,1,4,2,3$

Two $S_9$ with two pairs in common: $\underline{9},7,\underline{8},3,1,6,4,5,2$ and $6,5,7,1,4,2,3$

Two $S_9$ with three pairs in common: $\underline{9},\underline{7},\underline{8},2,3,6,5,4,1$ and $3,1,6,4,5,2$

Two $S_9$ with four pairs in common: $\underline{9},\underline{7},\underline{8},\underline{4},2,6,5,3,1$ and $1,6,5,2,3$

Two $S_9$ with five pairs in common: $\underline{9},\underline{7},\underline{8},4,1,\underline{6},\underline{5},2,3$ and $2,3,4,1$

Two $S_9$ with six pairs in common: $\underline{9},\underline{7},\underline{5},\underline{8},\underline{6},\underline{1},4,2,3$ and $3,4,2$.

\begin{center}
{\bf Appendix C}
\end{center}

Here we give the constructions of Langford sequences found in \cite{bermond}, Theorems 2 and 3. Theorem 2 in \cite{bermond} gives two different constructions depending on the parity of $d$.

Let $m\equiv 2d-1\;(\mathrm{mod}\;4)$, $m\geq 2d-1$, $d\geq 2$. For $d$ even, let $m=4t+3$.

\begin{table}[hp]
\begin{center}
\begin{tabular}{|c|c|c|c|c|}
\hline
 & $a_i$ & $b_i$ & $b_i-a_i$ & \\
\hline
$(1)$ & $2t+1-j $ & $2t+d+2+j$ & $d+1+2j$ & $j=0,1\ldots,t-\frac{d}{2}$\\
$(2)$ & $t+\frac{1}{2}d-1-j $ & $3t+\frac{1}{2}d+3+j$ & $2t+4+2j$ & $j=0,1\ldots,t-\frac{d}{2}$\\
$(3)$ & $d-2-j $ & $4t+4+j$ & $4t-d+6+2j$ & $j=0,1\ldots,\frac{d}{2}-2$\\
$(4)$ & $2t+\frac{1}{2}d+1 $ & $4t+\frac{1}{2}d+3$ & $2t+2$ & \\
$(5)$ & $\frac{1}{2}d-1-j $ & $4t+\frac{1}{2}d+4+j$ & $4t+5+2j$ & $j=0,1\ldots,\frac{d}{2}-2$\\
$(6)$ & $t+\frac{1}{2}d $ & $5t+\frac{1}{2}d+4$ & $4t+4$ & \\
$(7)$ & $2t+d+1-j $ & $6t+6+j$ & $4t-d+5+2j$ & $j=0,1\ldots,\frac{d}{2}-1$\\
$(8)$ & $2t+\frac{1}{2}d-j $ & $6t+\frac{1}{2}d+6+j$ & $4t+6+2j$ & $j=0,1\ldots,\frac{d}{2}-2$\\
$(9)$ & $6t+5-j $ & $6t+d+5+j$ & $d+2j$ & $j=0,1\ldots,t-\frac{d}{2}$\\
$(10)$ & $5t+\frac{1}{2}d+3-j $ & $7t+\frac{1}{2}d+6+j$ & $2t+3+2j$ & $j=0,1\ldots,t-\frac{d}{2}$\\
\hline
\end{tabular}
\caption{\cite{bermond}, Theorem 2, $d$ even}
\end{center}
\end{table}

\newpage
Let $m\equiv 2d-1\;(\mathrm{mod}\;4)$, $m\geq 2d-1$, $d\geq 2$. For $d$ odd, let $m=4t+1$, $d=e-1$.

\begin{table}[hp]
\begin{center}
\begin{tabular}{|c|c|c|c|c|}
\hline
 & $a_i$ & $b_i$ & $b_i-a_i$ & \\
\hline
$(1)$ & $2t-j $ & $2t+e+j$ & $e+2j$ & $j=0,1\ldots,t-\frac{e}{2}$\\
$(2)$ & $t+\frac{1}{2}e-2-j $ & $3t+\frac{1}{2}e+1+j$ & $2t+3+2j$ & $j=0,1\ldots,t-\frac{e}{2}$\\
$(3)$ & $e-3-j $ & $4t+2+j$ & $4t-e+5+2j$ & $j=0,1\ldots,\frac{e}{2}-2$\\
$(4)$ & $2t+\frac{1}{2}e $ & $4t+\frac{1}{2}e+1$ & $2t+1$ & \\
$(5)$ & $\frac{1}{2}e-2-j $ & $4t+\frac{1}{2}e+2+j$ & $4t+4+2j$ & $j=0,1\ldots,\frac{e}{2}-3$\\
$(6)$ & $t+\frac{1}{2}e-1 $ & $5t+\frac{1}{2}e+1$ & $4t+2$ & \\
$(7)$ & $2t+e-1-j $ & $6t+3+j$ & $4t-e+4+2j$ & $j=0,1\ldots,\frac{e}{2}-2$\\
$(8)$ & $2t+\frac{1}{2}e-1-j $ & $6t+\frac{1}{2}e+2+j$ & $4t+3+2j$ & $j=0,1\ldots,\frac{e}{2}-2$\\
$(9)$ & $6t+2-j $ & $6t+e+1+j$ & $e-1+2j$ & $j=0,1\ldots,t-\frac{e}{2}$\\
$(10)$ & $5t+\frac{1}{2}e-j $ & $7t+\frac{1}{2}e+2+j$ & $2t+2+2j$ & $j=0,1\ldots,t-\frac{e}{2}$\\
\hline
\end{tabular}
\caption{\cite{bermond}, Theorem 2, $d$ odd}
\end{center}
\end{table}

Theorem 3 in \cite{bermond} gives the following construction. Let $m=4t$, $d=2t-e(e\geq 0)$ and $2d\geq 3t+1$.

\begin{table}[hp]
\begin{center}
\begin{tabular}{|c|c|c|c|c|}
\hline
 & $a_i$ & $b_i$ & $b_i-a_i$ & \\
\hline
$(1)$ & $2t+e+1+j $ & $8t-j$ & $6t-e-1-2j$ & $j=0,1\ldots,t-e-1$\\
$(2)$ & $3t+2e+2+j $ & $7t+e-j$ & $4t-e-2-2j$ & $j=0,1\ldots,t-1$\\
$(3)$ & $2t-j $ & $6t+e-2j$ & $4t+e-j$ & $j=0,1\ldots,e$\\
$(4)$ & $2t+e-j $ & $6t+e-1-2j$ & $4t-1-j$ & $j=0,1\ldots,e-1$\\
$(5)$ & $1+j $ & $6t-e-1-j$ & $6t-e-2-2j$ & $j=0,1\ldots,t-e-2$\\
$(6)$ & $t+e+1+j $ & $5t-j$ & $4t-e-1-2j$ & $j=0,1\ldots,t-2e-2$ \\
$(7)$ & $t-e+j $ & $3t+2e+1-j$ & $2t+3e+1-2j$ & $j=0,1\ldots,2e$\\
\hline
\end{tabular}
\caption{\cite{bermond}, Theorem 3}
\end{center}
\end{table}

\begin{center}
{\bf Appendix D}
\end{center}


{\bf The number of pairs in common between Langford sequences with the pair \mbox {\boldmath ${(1,1)}$} appended and their reverses}

We append the pair $(1,1)$ to a Langford sequence $L_d^n$ at the beginning or at the end of the sequence. We denote the Langford sequence with the pair (1,1)
 appended at the end of the sequence by $L_d^n(1,1)$, and we denote the Langford sequence with the pair $(1,1)$ appended at the beginning of the sequence by
 $(1,1)L_d^n$. The results are summarized in Table \ref{append}. For the sequence $(1,1)L_d^n$ the pair $(a_i,b_i)$ will appear in a sequence and its reverse
 if and only if $a_i+b_i=2n-1$. For the remaining sequences the pair $(a_i,b_i)$ will appear in a sequence and its reverse if and only if $a_i+b_i=2n+3$.
\begin{table}[!h]
\begin{center}
\begin{tabular}{|c|c|c|c|c|c|}
\hline
 &$n$ & $d$ & Sequence & Pairs in common & Reference\\
\hline
1&$1$ & $9$ & $L_d^n(1,1)$ &$0$ & \cite{bermond}, $2$\\
&$\mathrm{(mod\;4)}$&&&&\\
\hline
2 &$4t$ & $4s+1$  & $(1,1)L_d^n$ & $2\;\textup{if}\;s=1$ & \cite{simpson}, $1, (8),(13)$\\
&& $s\geq 1$&& $3\;\textup{if}\;s=2$ & \cite{simpson}, $1,(8)(11)$\\
&&$t\geq 2s+1$&& $s+2\;\textup{if}\;s\equiv2\ \mathrm{(mod\;3)}$ & \cite{simpson}, $1$\\
&&&&$s\geq5$&$(8),(11),(12)$\\
&&&& $s\;\textup{if}\;s\equiv 0,1\ \mathrm{(mod\;3)}$ & \cite{simpson}, $1,(8)$\\
&&&& $s\geq3$ &\\
& &   & $L_d^n(1,1)$ & $2$ if $s\equiv 1\ \mathrm{(mod\;3)}$ & \cite{simpson}, $1$\\
&&&& $s\geq 4$&$(11),(12)$\\
& && &$0$ otherwise& \cite{simpson}, $1$\\
\hline
3 &$2d+3$ & $2\ \mathrm{(mod\;4)}$  & $L_d^n(1,1)$&$0$ &  \cite{bermond}, $2$\\
&&$d\neq 2$&&&\\
\hline
\end{tabular}
\caption{The number of pairs in common between $L_d^n$ with $(1,1)$ appended and their reverses} \label{append}
\end{center}

\end{table}

{\bf The number of pairs in common between hooked Langford sequences with the triple \mbox {\boldmath ${(2,0,2)}$} attached and their reverses}
For a hooked Langford sequence we need to remove the hook before we reverse the sequence and check for pairs in common. We remove the hook by attaching the triple $(2,0,2)$ to the sequence, where $0$ takes the place of the hook and denote this by $L_d^n*(2,0,2)$. These results are summarized in Table \ref{hooked}. For sequence 2 of the table the pair $(a_i,b_i)$ will appear in a sequence and its reverse if and only if $a_i+b_i=2n+1$. For the remaining sequences the pair $(a_i,b_i)$ will appear in a sequence and its reverse if and only if $a_i+b_i=2n+3$.

\newpage

 \begin{table}[!h]
\begin{center}
\begin{tabular}{|c|c|c|c|c|c|}
\hline
 & $n$ & $d$ & $r$ & Pairs in common & Reference\\
\hline
1 & $2d+1+4r$ & $3$ & $0$ & $1$ & \cite{simpson}, $2(5)$\\
& & & $r\geq1$ & $0$& \cite{simpson}, $2$\\
&&$4$ & $0$ & $2$ & \cite{simpson}, $2,(5)(10*)$\\
& & & $r\geq1$ & $0$& \cite{simpson}, $2$\\
&& $d$ even, $d\geq6$ & $r=0$ &$1$ & \cite{simpson}, $2,(5)$\\
&& & $r=1$ &$1$ & \cite{simpson}, $2,(10)$\\
&& & $r\geq2$ & $0$& \cite{simpson}, $2$\\
&& $d$ odd, $d\geq5$ & $r=0$ & $1$& \cite{simpson}, $2,(5)$\\
&& & $r\geq1$ & $0$& \cite{simpson}, $2$\\
\hline
2 & $2d+1$ & $0\ \mathrm{(mod\;6)}$ && $1$  & \cite{linekmor}, $2A,(2)$\\
 && $4\ \mathrm{(mod\;6)}$ && $1$ & \cite{linekmor}, $2A,(5)$\\
 && $d\neq4$&&&\\
 && $2\ \mathrm{(mod\;6)}$ && $0$& \cite{linekmor}, $2A$\\
\hline
3 &$4t+2$ & $4s$ & & $1$ if $s=2$ & \cite{simpson}, $2,(11)$\\
& & $t-2s=r$& &$2$ if $s\equiv 2\ \mathrm{(mod\;3)}$& \cite{simpson}, $ 2,(14),(16)$\\
&&$r\geq0$&&$s\geq 5$&\\
 &&& &$0$ if $s\equiv 0,1\ \mathrm{(mod\;3)}$& \cite{simpson}, $ 2$\\
&  & $4s+2$  && $2$ if $s\equiv 2\ \mathrm{(mod\;3)}$ & \cite{simpson}, $ 2,(14),(17)$\\
&   &  $t-2s-1=r$  && $0$ if $s\equiv0,1\ \mathrm{(mod\;3)}$& \cite{simpson}, $ 2$\\
&&$r\geq0$&& $s\neq1$&\\
&  &  && $1$ if $s=1$ & \cite{simpson}, $ 2,(11)$\\

&& $4s+3$  && $0$ if $s\equiv1\ \mathrm{(mod\;3)}$ & \cite{simpson}, $ 2$\\
& & $t-2s-1=r$  && $1$ if $s\equiv0\ \mathrm{(mod\;3)}$& \cite{simpson}, $ 2,(17)$\\
&&$r\geq0$&& $s\neq0$&\\
 &&   && $1$ if $s\equiv2\ \mathrm{(mod\;3)}$& \cite{simpson}, $2,(14)$\\
 &&&& $s\neq0$&\\
& &  &&$0$ if $s=0$ & \cite{simpson}, $ 2$\\
& & $4s+1$ && $3$ if $s=2$& \cite{simpson}, $ 2,(9),(13),(16)$\\
&& $t-2s\geq 0$  && $2$ if $s=3$ & \cite{simpson}, $ 2,(9)$\\
&&   && $s-1$ if $s\equiv1\ \mathrm{(mod\;3)}$ & \cite{simpson}, $ 2,(9),(14)$\\
&&   && $s$ if $s\equiv0,2\;\mathrm{(mod\;3)}$ & \cite{simpson}, $ 2,(9),(16)$ \\
&&&& $s\geq5$ &\\
\hline
\end{tabular}
\caption{The number of pairs in common between $hL_d^n*(2,0,2)$ and their reverses} \label{hooked}
\end{center}
\end{table}